\documentclass[11pt,a4paper,oneside,reqno]{amsart}
\usepackage{lmodern}
\usepackage[T1]{fontenc}
\usepackage[main=english, french]{babel}
\usepackage{amsopn}

\usepackage{todonotes}

\newtheorem{theorem}{Theorem}[section]

\newtheorem{question}[theorem]{Question}
\newtheorem{corollary}[theorem]{Corollary}
\newtheorem{conjecture}[theorem]{Conjecture}
\newtheorem{lemma}[theorem]{Lemma}
\newtheorem{proposition}[theorem]{Proposition}

\newtheorem{definition}[theorem]{Definition}
\newtheorem{example}[theorem]{Example}
\newtheorem{remark}[theorem]{Remark}

\usepackage[square,sort,comma,numbers]{natbib}

\usepackage[all]{xy}
\usepackage{enumerate}
\usepackage{amsfonts,amssymb,amsmath,eucal,pinlabel,array,hhline}
\usepackage{slashed}
\usepackage{tabulary}
\usepackage{fancyhdr}
\usepackage{a4wide}
\usepackage{calrsfs,bbm,dsfont}
\usepackage[position=b]{subcaption}

\newcounter{notes}%

	\newcommand{\ignore}[1]{}  

\newcommand{\id}{\mathrm{id}}
 
\newcommand{\Li}{\mathrm{Li}_2}
\newcommand{\Vol}{\mathrm{Vol}}

	\usepackage{comment}
\usepackage{amsfonts}
\usepackage{multirow} 

\usepackage{brunnian} 

\usetikzlibrary{knots}

\usetikzlibrary{matrix}
\usetikzlibrary{arrows}

\usetikzlibrary{patterns}
\usetikzlibrary{shapes}
\usetikzlibrary{arrows}
\usetikzlibrary{decorations.markings}
\usetikzlibrary{decorations.pathreplacing}

\newcommand{\Log}{\mathrm{Log}}

\usetikzlibrary{decorations.pathreplacing,decorations.markings}

\tikzset{
  on each segment/.style={
    decorate,
    decoration={
      show path construction,
      moveto code={},
      lineto code={
        \path [#1]
        (\tikzinputsegmentfirst) -- (\tikzinputsegmentlast);
      },
      curveto code={
        \path [#1] (\tikzinputsegmentfirst)
        .. controls
        (\tikzinputsegmentsupporta) and (\tikzinputsegmentsupportb)
        ..
        (\tikzinputsegmentlast);
      },
      closepath code={
        \path [#1]
        (\tikzinputsegmentfirst) -- (\tikzinputsegmentlast);
      },
    },
  },
  mid arrow/.style={postaction={decorate,decoration={
        markings,
        mark=at position .5 with {\arrow[#1]{>}}
      }}}
      ,
}

\tikzset{
  on each segment/.style={
    decorate,
    decoration={
      show path construction,
      moveto code={},
      lineto code={
        \path [#1]
        (\tikzinputsegmentfirst) -- (\tikzinputsegmentlast);
      },
      curveto code={
        \path [#1] (\tikzinputsegmentfirst)
        .. controls
        (\tikzinputsegmentsupporta) and (\tikzinputsegmentsupportb)
        ..
        (\tikzinputsegmentlast);
      },
      closepath code={
        \path [#1]
        (\tikzinputsegmentfirst) -- (\tikzinputsegmentlast);
      },
    },
  },
  mid arrow d/.style={postaction={decorate,decoration={
        markings,
        mark=at position .5 with {\arrow[#1]{>>}}
      }}}
      ,
}

\tikzset{
  on each segment/.style={
    decorate,
    decoration={
      show path construction,
      moveto code={},
      lineto code={
        \path [#1]
        (\tikzinputsegmentfirst) -- (\tikzinputsegmentlast);
      },
      curveto code={
        \path [#1] (\tikzinputsegmentfirst)
        .. controls
        (\tikzinputsegmentsupporta) and (\tikzinputsegmentsupportb)
        ..
        (\tikzinputsegmentlast);
      },
      closepath code={
        \path [#1]
        (\tikzinputsegmentfirst) -- (\tikzinputsegmentlast);
      },
    },
  },
  mid arrow s/.style={postaction={decorate,decoration={
        markings,
        mark=at position .5 with {\arrow[#1]{stealth}}
      }}}
      ,
}

\tikzset{
  on each segment/.style={
    decorate,
    decoration={
      show path construction,
      moveto code={},
      lineto code={
        \path [#1]
        (\tikzinputsegmentfirst) -- (\tikzinputsegmentlast);
      },
      curveto code={
        \path [#1] (\tikzinputsegmentfirst)
        .. controls
        (\tikzinputsegmentsupporta) and (\tikzinputsegmentsupportb)
        ..
        (\tikzinputsegmentlast);
      },
      closepath code={
        \path [#1]
        (\tikzinputsegmentfirst) -- (\tikzinputsegmentlast);
      },
    },
  },
  mid arrow l/.style={postaction={decorate,decoration={
        markings,
        mark=at position .5 with {\arrow[#1]{latex}}
      }}}
      ,
}

\usepackage{graphicx}

\newcommand{\C}{\mathbb{C}}
\newcommand{\R}{\mathbb{R}}
\newcommand{\Q}{\mathbb{Q}}
\newcommand{\Z}{\mathbb{Z}}

\newcommand{\B}{\mathsf{b}}

\newcommand{\sarrow}{\,\mathbin{\rotatebox[origin=c]{90}{$\rightarrow$}}}
\newcommand{\darrow}{\,\mathbin{\rotatebox[origin=c]{90}{$\twoheadrightarrow$}}}

\usepackage{amssymb,amsmath,bm}
\usepackage{amsthm}
\usepackage{empheq}
\usepackage{textcomp}
\usepackage{enumerate}      
\usepackage{graphicx}        
\usepackage{caption}
\usepackage{nicematrix}
\usepackage{tikz-cd}
\usepackage{tikz}

\usepackage{colorinfo}

\usetikzlibrary{knots}

\usetikzlibrary{matrix}
\usetikzlibrary{arrows}

\usetikzlibrary{patterns}
\usetikzlibrary{shapes}
\usetikzlibrary{arrows}
\usetikzlibrary{decorations.markings}
\usetikzlibrary{decorations.pathreplacing}

\usepackage{mathrsfs}

\usepackage{url}

\renewcommand{\setminus}{{\smallsetminus}}

\usepackage{amssymb,amsmath,bm}
\usepackage{amsthm}
\usepackage{hyperref}
\usepackage{mathrsfs}
\usepackage{textcomp}
\usepackage{enumerate}
\usepackage{graphicx}
\usepackage{tikz}
\usepackage{verbatim}
\usepackage{nicematrix}

\numberwithin{equation}{section}

\def\Vol{\operatorname{Vol}}
\def \CC{\mathbb C}
\def \Re{\operatorname{Re}}
\def \arg{\operatorname{arg}}

\def\SS{{\mathbb S}}
\def\CC{{\mathbb C}}
\def\RR{{\mathbb R}}
\def\ZZ{{\mathbb Z}}

\def\QQ{{\mathbb Q}}

\def\Vol{\operatorname{Vol}}
\def\CS{\operatorname{CS}}
\def \Re{\operatorname{Re}}

\def \arg{\operatorname{arg}}
\def \Li{\operatorname{Li_2}}

\def \Log{\operatorname{Log}}
\def \BLog{\operatorname{\textbf{Log}}}
\def \Arg{\operatorname{Arg}}

\def \det{\operatorname{det}}
\def \Hess{\operatorname{Hess}}

\title[The Andersen--Kashaev volume conjecture for FAMED geometric triangulations]{The Andersen--Kashaev volume conjecture for \\ FAMED geometric triangulations}
\author{Fathi Ben Aribi and Ka Ho Wong}
\date{\today}

\address{Sorbonne Université, Université Paris Cité, CNRS, IMJ-PRG, F-75005 Paris, France}
\email{benaribi@imj-prg.fr}

\address{
	Yale University, New Haven, CT06520, USA}
\email{kaho.wong@yale.edu}

\begin{document}

\maketitle

\begin{abstract} 

We investigate the Andersen--Kashaev volume conjecture by introducing the notion of FAMED triangulations, a class of ideal triangulations of $3$-manifolds satisfying certain specific combinatorial properties. For any FAMED triangulation of a one-cusped hyperbolic $3$-manifold $M$ with trivial second homology, we prove the existence of the Jones function in the Teichm\"uller TQFT of $M$. 

For FAMED geometric triangulations of $M$, we establish an asymptotic expansion of the Jones function in terms of the Neumann--Zagier potential function and the 1-loop invariant of Dimofte--Garoufalidis. As a consequence, we prove the Andersen--Kashaev volume conjecture for $M$ and provide new insights for the AJ conjecture for the Teichmüller TQFT developed by Andersen--Malusa. We further discover a new phenomenon: for FAMED geometric triangulations, the partition function in Teichm\"uller TQFT decays exponentially with decrease rate the hyperbolic volume of a cone structure determined by the prescribed angle structure. This perspective provides a potential application to the Casson conjecture on angle structures.

Expanding the previous result of Guéritaud, Piguet-Nakazawa and the first author and complementing a parallel result of Guilloux and both authors, we prove all the above generalizations of the Andersen–Kashaev volume conjecture for every hyperbolic twist knot and for the first 42,000 hyperbolic knots in $\mathbb{S}^3$.
\end{abstract}

\tableofcontents

\section{Introduction}\label{sec:introduction}

Volume conjectures expect that certain asymptotics of quantum invariants of hyperbolic knots reveal their hyperbolic volume. The ways to prove such conjectures lie at the interface of topology, combinatorics, hyperbolic geometry, quantum algebra and asymptotic analysis.

For a one-cusped hyperbolic $3$-manifold $M$ with trivial second homology, endowed with a triangulation $X$ with an angle structure $\alpha$, the partition function $\mathscr{Z}_\hbar(X,\alpha)$ of the Andersen--Kashaev TQFT (or Teichmüller TQFT, see Section \ref{sub:AK:TQFT} and \cite{AK}) is expected in the Andersen--Kashaev volume conjecture to manifest an exponential decrease with rate the hyperbolic volume of $M$ when $\hbar \to 0^+$.
In this paper, we generalize the Andersen--Kashaev volume conjecture and we prove it for all triangulations of hyperbolic knot complements which satisfy certain combinatorial and geometric properties (namely FAMED geometric triangulations, see Section \ref{sub:intro:FAMED}). 

In particular, the subtle algebraic simplifications in the partition function of the TQFT and the technical "saddle point" constraints of asymptotic analysis now directly follow from the fact that a geometric triangulation of the knot complement satisfies \textit{combinatorial} properties... which are easy to check! As evidence of this, we mention in Section \ref{sub:intro:computer} that such triangulations were found in \cite{BAGW} for more than 42,000 knots via the computer, which provided as many new examples of knots satisfying the Andersen--Kashaev volume conjecture.

We hope parts of this work can be extended to the study of volume conjectures for other quantum invariants, for which the analytic parts of the proofs are notoriously difficult.

\subsection{FAMED ideal triangulation and volume conjectures}\label{sub:intro:FAMED}
To state our main results, we briefly discuss the definitions of several matrices that will be used to construct the Teichm\"uller TQFT invariants of \cite{AK}. 
In this paper, we consider $3$-manifolds with one toroidal boundary component and trivial second homology, notably complements of hyperbolic knots in $\SS^3$.

Recall that a $3$-dimensional triangulation is \textit{ordered} (or \textit{branched}) if every tetrahedron is endowed with an order on its four vertices such that face gluings respect the relative vertex order (or equivalently, if the edges can be oriented so that no face is a $3$-cycle). See Figure \ref{fig:41:face:matrices} for an example we will develop in detail.

Let $M$ be a $3$-manifold with one toroidal boundary component.
Let $X$ be an ordered ideal triangulation of $M$ with $N$ tetrahedra $X^3 = \{T_1,\dots, T_N\}$.
We denote the \textit{sign of $T_j$} as $\varepsilon(T_j) = \pm 1$ and define it as $+1$ if the vectors $\overrightarrow{01}, \overrightarrow{02}, \overrightarrow{03}$ in this order follow the right-hand rule, and $-1$ otherwise (see Figure \ref{fig:41:face:matrices}).
We define $\mathcal{E}=\mathcal{E}(X)$ to be the diagonal matrix of size $N$ encoding the signs $\pm 1$ of the tetrahedra.
Note that the sign of a tetrahedron depends on how it is embedded in $3$-space, but once you fix the sign of one tetrahedron, all the other signs are uniquely determined by the face gluings. Hence $\mathcal{E}(X)$ is either fixed uniquely if $X$ is realised in $3$-space, or defined up to multiplication by $\pm 1$ if $X$ is considered abstractly. 

Let $X^2$ be the set of faces of $X$, of cardinal $2N$. For $k=0,1,2,3$, let $x_k : X^3 \to X^2$ be the map defined by sending a tetrahedron to its face that is opposite to the $k$-th vertex, and $\mathcal{X}_k \in M_{N,2N}(\Z)$ the matrix of coefficients
$(\mathcal{X}_k)_{i,j}:=\delta_{j\text{-th face},x_k(T_i)}$. Finally, define matrices $\mathcal{B}:=\begin{pmatrix}
    0_N \\ \mathcal{E}
\end{pmatrix} \in M_{2N,N}(\Z)$
and $\mathcal{A} = \begin{pmatrix}
    \mathcal{X}_0-\mathcal{X}_1+\mathcal{X}_2\\\mathcal{X}_2-\mathcal{X}_3
\end{pmatrix} \in M_{2N,2N}(\Z)$. We informally call these matrices "face adjacency matrices" associated to the ordered triangulation $X$.

As an example we hope will be enlightening, consider Thurston's triangulation of the figure-eight knot complement $M=S^3 \setminus 4_1$, drawn in Figure \ref{fig:41:face:matrices}, with $N=2$ tetrahedra. It is naturally an ordered triangulation, as seen in the picture. The $3$-cells form the set $X^3=\{T_+,T_-\}$, are drawn in purple and have respective signs $1$ and $-1$. The $2$-cells form the set $X^2=\{\mathsf{A},\mathsf{B},\mathsf{C},\mathsf{D}\}$ and are drawn in blue. The face boundary maps $x_k$ are given by:
 \begin{align*}
 x_0(T_+) = \mathsf{B},\ \ x_1(T_+) = \mathsf{D},\ \ x_2(T_+) = \mathsf{C},\ \ x_3(T_+) = \mathsf{A},\\
 x_0(T_-) = \mathsf{C},\ \ x_1(T_-) = \mathsf{A},\ \ x_2(T_-) = \mathsf{B},\ \ x_3(T_-) = \mathsf{D}.
 \end{align*}
 Remark that the faces $\mathsf{A}$ and $\mathsf{D}$ were swapped between \cite[Figure 2]{BAGPN} and Figure \ref{fig:41:face:matrices}.
The $1$-cells form the set $X^1=\{\rightarrow,\twoheadrightarrow\}$, respectively drawn in orange and pink, each arrow having six pre-images. Finally, the set of $0$-cells $X^0$ of the triangulation $X$ is a singleton (with eight pre-image vertices), drawn in teal, representing the knot $4_1$ collapsed to a point that is deleted after (small teal triangles around the vertices represent truncated vertices). The shape parameters $z^{\cdot}_\pm$ are drawn in red (see below).
Figure \ref{fig:41:face:matrices} also describes the matrices $\mathcal{X}_0, \mathcal{X}_1, \mathcal{X}_2, \mathcal{X}_3, \mathcal{A}, \mathcal{B}, \mathcal{E}$.

\begin{figure}[!h]
    \centering

\begin{tikzpicture}

\begin{scope}[scale=1.2]

\begin{scope}[xshift=0cm,yshift=0cm,rotate=0,scale=1.2]

\draw[color=violet] (-0.9,1) node{T$_+$} ;

\draw[color=red] (0.2,0.5) node{\tiny $z_+$} ;
\draw[color=red] (0,-1.2) node{\tiny $z_+$} ;
\draw[color=red] (0.5,-0.1) node{\tiny $z''_+$} ;
\draw[color=red] (-1.1,0.5) node{\tiny $z''_+$} ;
\draw[color=red] (-0.55,-0.1) node{\tiny $z'_+$} ;
\draw[color=red] (1.1,0.5) node{\tiny $z'_+$} ;

\draw (0,-0.2) node{\tiny $0$} ;
\draw (0,1.2) node{\tiny $1$} ;
\draw (1,-0.6) node{\tiny $2$} ;
\draw (-1,-0.6) node{\tiny $3$} ;

\begin{scope}[xshift=0cm,yshift=0cm,rotate=0,scale=.1]
    \draw[color=teal,thick] (0,2)--(-1.732,-1)--(1.732,-1)--(0,2);
\end{scope}

\begin{scope}[xshift=0cm,yshift=1cm,rotate=180,scale=.1]
    \draw[color=teal,thick ] (-1.3,0)--(0,2)--(1.3,0);
    \draw[color=teal,thick] (-1.3,0) arc (-150:-30:1.5);
\end{scope}

\begin{scope}[xshift=-0.86cm,yshift=-.5cm,rotate=-60,scale=.1]
    \draw[color=teal,thick ] (-1.3,0)--(0,2)--(1.3,0);
    \draw[color=teal,thick] (-1.3,0) arc (-150:-30:1.5);
\end{scope}

\begin{scope}[xshift=0.86cm,yshift=-.5cm,rotate=60,scale=.1]
    \draw[color=teal,thick ] (-1.3,0)--(0,2)--(1.3,0);
    \draw[color=teal,thick] (-1.3,0) arc (-150:-30:1.5);
\end{scope}

\draw[color=blue] (0.9,0.9) node{\small $\mathsf{\color{blue}B}$} ;
\draw[color=blue] (0,-0.6) node{\small $\mathsf{\color{blue}D}$} ;
\draw[color=blue] (-0.5,0.3) node{\small $\mathsf{\color{blue}C}$} ;
\draw[color=blue] (0.5,0.3) node{\small $\mathsf{\color{blue}A}$} ;

\path [draw=black,postaction={on each segment={mid arrow =orange, thick}}]
(0,0)--(-1.732/2,-0.5);

\path [draw=black,postaction={on each segment={mid arrow =orange, thick}}]
(0,0)--(0,1);

\path [draw=black,postaction={on each segment={mid arrow d =magenta, thick}}]
(0,0)--(1.732/2,-0.5);

\draw(1.732/2,-0.5) arc (-30:-89.5:1);
\draw[->,color=orange, thick](0,-1) arc (-89.5:-90:1);
\draw (0,-1) arc (-90:-150:1);

\draw(0,1) arc (90:29.5:1);
\draw[->>,color=magenta, thick](1.732/2,0.5) arc (29.5:30:1);
\draw (1.732/2,0.5) arc (30:-30:1);

\draw(0,1) arc (90:150:1);
\draw[->>,color=magenta, thick](-1.732/2,0.5) arc (149.5:150:1);
\draw (-1.732/2,0.5) arc (150:210:1);

\end{scope}

\begin{scope}[xshift=3.8cm,yshift=0cm,rotate=0,scale=1.2]

\draw[color=violet] (0.9,1) node{T$_-$} ;

\draw[color=red] (0.2,0.5) node{\tiny $z_-$} ;
\draw[color=red] (0,-1.2) node{\tiny $z_-$} ;
\draw[color=red] (0.5,-0.1) node{\tiny $z''_-$} ;
\draw[color=red] (-1.1,0.5) node{\tiny $z''_-$} ;
\draw[color=red] (-0.55,-0.1) node{\tiny $z'_-$} ;
\draw[color=red] (1.1,0.5) node{\tiny $z'_-$} ;

\draw (0,-0.2) node{\tiny $0$} ;
\draw (0,1.2) node{\tiny $1$} ;
\draw (1,-0.6) node{\tiny $3$} ;
\draw (-1,-0.6) node{\tiny $2$} ;

\begin{scope}[xshift=0cm,yshift=0cm,rotate=0,scale=.1]
    \draw[color=teal,thick] (0,2)--(-1.732,-1)--(1.732,-1)--(0,2);
\end{scope}

\begin{scope}[xshift=0cm,yshift=1cm,rotate=180,scale=.1]
    \draw[color=teal,thick ] (-1.3,0)--(0,2)--(1.3,0);
    \draw[color=teal,thick] (-1.3,0) arc (-150:-30:1.5);
\end{scope}

\begin{scope}[xshift=-0.86cm,yshift=-.5cm,rotate=-60,scale=.1]
    \draw[color=teal,thick] (-1.3,0)--(0,2)--(1.3,0);
    \draw[color=teal,thick] (-1.3,0) arc (-150:-30:1.5);
\end{scope}

\begin{scope}[xshift=0.86cm,yshift=-.5cm,rotate=60,scale=.1]
    \draw[color=teal,thick ] (-1.3,0)--(0,2)--(1.3,0);
    \draw[color=teal,thick] (-1.3,0) arc (-150:-30:1.5);
\end{scope}

\draw[color=blue] (-0.9,0.9) node{\small $\mathsf{\color{blue}C}$} ;
\draw[color=blue] (0,-0.6) node{\small $\mathsf{\color{blue}A}$} ;
\draw[color=blue] (-0.5,0.3) node{\small $\mathsf{\color{blue}D}$} ;
\draw[color=blue] (0.5,0.3) node{\small $\mathsf{\color{blue}B}$} ;

\path [draw=black,postaction={on each segment={mid arrow =orange, thick}}]
(0,0)--(-1.732/2,-0.5);

\path [draw=black,postaction={on each segment={mid arrow d =magenta, thick}}]
(0,0)--(0,1);

\path [draw=black,postaction={on each segment={mid arrow d =magenta, thick}}]
(0,0)--(1.732/2,-0.5);

\draw(1.732/2,-0.5) arc (-30:-89.5:1);
\draw[<<-,color=magenta, thick] (0,-1) arc (-89.5:-90:1);
\draw(0,-1) arc (-90:-150:1);

\draw(0,1) arc (90:29.5:1);
\draw[->,color=orange, thick](1.732/2,0.5) arc (29.5:30:1);
\draw (1.732/2,0.5) arc (30:-30:1);

\draw(0,1) arc (90:149.5:1);
\draw[->,color=orange, thick](-1.732/2,0.5) arc (149.5:150:1);
\draw (-1.732/2,0.5) arc (150:210:1);

\end{scope}

\end{scope}

\end{tikzpicture}

\begin{tikzpicture}

\draw (0,0) node {
$\mathcal{X}_0=\begin{pNiceMatrix}[first-row, first-col]
 &  \mathsf{\color{blue}A} & \mathsf{\color{blue}B} & \mathsf{\color{blue}C} & \mathsf{\color{blue}D} \\
\textcolor{violet}{T_+} & 0 & 1 & 0 & 0 \\
\textcolor{violet}{T_-} & 0 & 0 & 1 & 0
\end{pNiceMatrix}$
};

\draw (0,-1.4) node {
$\mathcal{X}_1=\begin{pNiceMatrix}[first-row, first-col]
 &  \mathsf{\color{blue}A} & \mathsf{\color{blue}B} & \mathsf{\color{blue}C} & \mathsf{\color{blue}D} \\
\textcolor{violet}{T_+} & 0 & 0 & 0 & 1 \\
\textcolor{violet}{T_-} & 1 & 0 & 0 & 0
\end{pNiceMatrix}$};

\draw (0,-2.8) node {
$\mathcal{X}_2=\begin{pNiceMatrix}[first-row, first-col]
 &  \mathsf{\color{blue}A} & \mathsf{\color{blue}B} & \mathsf{\color{blue}C} & \mathsf{\color{blue}D} \\
\textcolor{violet}{T_+} & 0 & 0 & 1 & 0 \\
\textcolor{violet}{T_-} & 0 & 1 & 0 & 0
\end{pNiceMatrix}$};

\draw (0,-4.2) node {
$\mathcal{X}_3=\begin{pNiceMatrix}[first-row, first-col]
 &  \mathsf{\color{blue}A} & \mathsf{\color{blue}B} & \mathsf{\color{blue}C} & \mathsf{\color{blue}D} \\
\textcolor{violet}{T_+} & 1 & 0 & 0 & 0 \\
\textcolor{violet}{T_-} & 0 & 0 & 0 & 1
\end{pNiceMatrix}$};

\draw (7,-.5) node {
$\mathcal{A} = \begin{pmatrix}
    \mathcal{X}_0-\mathcal{X}_1+\mathcal{X}_2\\\mathcal{X}_2-\mathcal{X}_3
\end{pmatrix}= \begin{pmatrix}
0&1&1&-1\\
-1&1&1&0\\
-1&0&1&0\\
0&1&0&-1
\end{pmatrix}$};

\draw (5,-1.6-.5) node {
$\mathcal{B}= \begin{pmatrix}
0&0\\
0&0\\
1&0\\
0& -1
\end{pmatrix}$};

\draw (7,-3.2-.5) node {
$\mathcal{E}=\begin{pmatrix}
    1&0\\0&-1
\end{pmatrix}$ (matrix of {tetrahedra signs})};

\end{tikzpicture}

\caption{Thurston's triangulation of $M=S^3 \setminus 4_1$, and the face adjacency matrices}
    \label{fig:41:face:matrices}
\end{figure}

 Observe that our example satisfies the following property:
 \begin{center}
      \framebox{{\bf Property I}: $\mathcal{A}$ is invertible. }
 \end{center}
 Denote $\cdot^{\!\top}$ the matrix transpose.
If Property I is satsified, then we can define the  symmetric square matrix of size $N$
$$\mathcal{X}_0 \mathcal{A}^{-1} \mathcal{B} + \left ( \mathcal{X}_0 \mathcal{A}^{-1} \mathcal{B} \right )^{\!\top} + \frac{\mathcal{E}+{\rm Id}_N}{2} \in \mathcal{Sym}_N(\Q),$$
which is equal to $\begin{pmatrix}
    -1 & 0 \\ 0 &2
\end{pmatrix}$ in the example of Figure \ref{fig:41:face:matrices}.

We are now going to relate this symmetric matrix with the Neumann-Zagier matrices with respect to 
a specific peripheral curve.

Recall that $X^1=\{E_1,\dots, E_N\}$ is the set of 1-cells (edges) of the ideal triangulation $X$. It is known that for a $3$-manifold with one cusp such as a knot complement, there exists a set of $N-1$ independent edges, in the sense that if the edge equations of these edges are satisfied, the last edge equation will automatically be satisfied. Furthermore, for an ordered triangulation, there is a canonical assignment of shape parameters $z_k,z_k',z_k''$ to the edges of the $k$-th tetrahedron $T_k$. To be precise, we always assign $z_k$ to the edge with endpoints 0 and 1 and $z_k,z_k',z_k''$ have to cycle positively around each vertex in this order.  

Now we fix a peripheral curve $l$ (for a knot complement this will be the preferred longitude, which is null-homologous in the knot complement). Let $\xi \in\CC$ be a complex number. We consider the set of $N$ equations on logarithmic shape parameters $\BLog \mathbf{z},\BLog \mathbf{z'},\BLog \mathbf{z''}$ (where $\BLog \mathbf{z^{\cdot}}$ is the vector of parameters $\Log(z^{\cdot}_k)$) given by
\begin{itemize}
    \item the first $N-1$ gluing equations coming from edges of $X^1$ and
    \item the holonomy equation
    $\mathrm{H}^\C_{X,l}(\mathbf{z}) = \xi$.
\end{itemize}
This system of equations can be written
\begin{align}\label{equ}
 \mathbf{G} \BLog \mathbf{z} +\mathbf{G'} \BLog \mathbf{z'} + \mathbf{G''} \BLog \mathbf{z''} =
\begin{pmatrix}
    2i\pi \\ \ldots \\ 2i\pi \\ 
    \xi 
\end{pmatrix},
\end{align}
where $\mathbf{G},\mathbf{G'},\mathbf{G''} \in M_N(\Z)$ are the \textit{Neumann-Zagier matrices} associated to the previous numbering of edges and the choice of curve $l$.

Getting back to the example of $S^3 \setminus 4_1$, let $\nu(K)$ denote a tubular neighborhood of $K$ in $S^3$. Figure \ref{fig:41:cusp:triang:NZ} shows the cusp triangulation of the torus $\partial \nu(K) = \partial \left ( S^3 \setminus \nu(K)\right )$ obtained from the triangulation of Figure \ref{fig:41:face:matrices} by truncating the vertices. Gluing equations associated to the orange simple arrow and the pink double arrow can be read on the cusp triangulation by going around a vertex (which represents a truncated edge) and are detailed below the picture. The same can be done for the holonomy equations associated to the meridian $y$ or the preferred longitude $l=x+2y$, which is the chosen curve here. The associated Neumann-Zagier matrices $\mathbf{G},\mathbf{G'},\mathbf{G''}$ are then detailed.

\begin{figure}[!h]
    \centering

\begin{tikzpicture}

\begin{scope}[scale=1.2]

\draw (0,0)--(8,0)--(9,1.732)--(1,1.732)--(0,0);
\draw (1,1.732)--(2,0)--(3,1.732)--(4,0)--(5,1.732)--(6,0)--(7,1.732)--(8,0);

\filldraw[color=magenta] (0,0) circle (0.08cm);
\filldraw[color=magenta] (1,1.732) circle (0.08cm);
\filldraw[color=magenta] (0+8,0) circle (0.08cm);
\filldraw[color=magenta] (1+8,1.732) circle (0.08cm);
\draw[color=magenta,very thick] (0,0) circle (0.15cm);
\draw[color=magenta, very thick] (1,1.732) circle (0.15cm);
\draw[color=magenta,very thick] (8,0) circle (0.15cm);
\draw[color=magenta, very thick] (9,1.732) circle (0.15cm);
\filldraw[color=orange] (6,0) circle (0.08cm);
\filldraw[color=orange] (7,1.732) circle (0.08cm);
\draw[color=orange,very thick] (2,0) circle (0.08cm);
\draw[color=orange, very thick] (3,1.732) circle (0.08cm);
\draw[color=magenta,very thick] (4,0) circle (0.08cm);
\draw[color=magenta, very thick] (5,1.732) circle (0.08cm);
\draw[color=magenta,very thick] (4,0) circle (0.15cm);
\draw[color=magenta, very thick] (5,1.732) circle (0.15cm);

\draw[color=teal] (1,.6) node{$2_+$};
\draw[color=teal] (3,.6) node{$0_+$};
\draw[color=teal] (5,.6) node{$1_+$};
\draw[color=teal] (7,.6) node{$3_+$};
\draw[color=teal] (2,1.2) node{$1_-$};
\draw[color=teal] (4,1.2) node{$0_-$};
\draw[color=teal] (6,1.2) node{$2_-$};
\draw[color=teal] (8,1.2) node{$3_-$};

\draw[color=blue] (1,0) node{$\mathsf{\color{blue}D}$};
\draw[color=blue] (3,0) node{$\mathsf{\color{blue}D}$};
\draw[color=blue] (5,0) node{$\mathsf{\color{blue}C}$};
\draw[color=blue] (7,0) node{$\mathsf{\color{blue}C}$};
\draw[color=blue] (.5,.9) node{$\mathsf{\color{blue}A}$};
\draw[color=blue] (1.5,.9) node{$\mathsf{\color{blue}B}$};
\draw[color=blue] (2.5,.9) node{$\mathsf{\color{blue}C}$};
\draw[color=blue] (3.5,.9) node{$\mathsf{\color{blue}A}$};
\draw[color=blue] (4.5,.9) node{$\mathsf{\color{blue}B}$};
\draw[color=blue] (5.5,.9) node{$\mathsf{\color{blue}A}$};
\draw[color=blue] (6.5,.9) node{$\mathsf{\color{blue}D}$};
\draw[color=blue] (7.5,.9) node{$\mathsf{\color{blue}B}$};
\draw[color=blue] (8.5,.9) node{$\mathsf{\color{blue}A}$};
\draw[color=blue] (2,1.732) node{$\mathsf{\color{blue}D}$};
\draw[color=blue] (4,1.732) node{$\mathsf{\color{blue}D}$};
\draw[color=blue] (6,1.732) node{$\mathsf{\color{blue}C}$};
\draw[color=blue] (8,1.732) node{$\mathsf{\color{blue}C}$};

\draw[color=red] (.35,.2) node{\scriptsize $z''_+$};
\draw[color=red] (2.35,.2) node{\scriptsize $z'_+$};
\draw[color=red] (4.35,.2) node{\scriptsize $z''_+$};
\draw[color=red] (6.35,.2) node{\scriptsize $z'_+$};
\draw[color=red] (2-.35,.2) node{\scriptsize $z_+$};
\draw[color=red] (4-.35,.2) node{\scriptsize $z''_+$};
\draw[color=red] (6-.35,.2) node{\scriptsize $z_+$};
\draw[color=red] (8-.35,.2) node{\scriptsize $z''_+$};
\draw[color=red] (2,.45) node{\scriptsize $z'_-$};
\draw[color=red] (4,.45) node{\scriptsize $z''_-$};
\draw[color=red] (6,.45) node{\scriptsize $z'_-$};
\draw[color=red] (8,.45) node{\scriptsize $z''_-$};
\draw[color=red] (1,1.732-.45) node{\scriptsize $z'_+$};
\draw[color=red] (3,1.732-.45) node{\scriptsize $z_+$};
\draw[color=red] (5,1.732-.45) node{\scriptsize $z'_+$};
\draw[color=red] (7,1.732-.45) node{\scriptsize $z_+$};
\draw[color=red] (1+.35,1.732-.2) node{\scriptsize $z_-$};
\draw[color=red] (3+.35,1.732-.2) node{\scriptsize $z'_-$};
\draw[color=red] (5+.35,1.732-.2) node{\scriptsize $z_-$};
\draw[color=red] (7+.35,1.732-.2) node{\scriptsize $z'_-$};
\draw[color=red] (3-.35,1.732-.2) node{\scriptsize $z''_-$};
\draw[color=red] (5-.35,1.732-.2) node{\scriptsize $z_-$};
\draw[color=red] (7-.35,1.732-.2) node{\scriptsize $z''_-$};
\draw[color=red] (9-.35,1.732-.2) node{\scriptsize $z_-$};

\begin{scope}[xshift=4.2cm,yshift=2.3cm,rotate=-120]
\draw[color=olive,->] (0,0)--(3.2,0);    
\draw[color=olive] (3.5,0) node {$y$};
\end{scope}

\begin{scope}[xshift=0cm,yshift=1.075cm,rotate=0]
\draw[color=olive,->] (0,0)--(9.5,0);    
\draw[color=olive] (9.7,0) node {$x$};
\end{scope}

\end{scope}

\end{tikzpicture}

$(\textcolor{orange}{\rightarrow})$
starts at
\begin{tikzpicture}
\draw[color=orange, very thick] (0,0) circle (0.08cm);
\end{tikzpicture}
and ends at
\begin{tikzpicture}
\filldraw[color=orange] (0,0) circle (0.08cm);
\end{tikzpicture}
\ : \color{red}
$2 {\rm Log}(z_+)+{\rm Log}(z'_+)+2{\rm Log}(z'_-)+{\rm Log}(z''_-)=2i\pi$
\color{black}

$(\textcolor{magenta}{\twoheadrightarrow})$
starts at
\begin{tikzpicture}
\draw[color=magenta,very thick] (0,0) circle (0.08cm);
\draw[color=magenta,very thick] (0,0) circle (0.15cm);
\end{tikzpicture}
and ends at
\begin{tikzpicture}
\filldraw[color=magenta] (0,0) circle (0.08cm);
\draw[color=magenta,very thick] (0,0) circle (0.15cm);
\end{tikzpicture}
\ : \color{red}
$ {\rm Log}(z'_+)+2{\rm Log}(z''_+)+2{\rm Log}(z_-)+{\rm Log}(z''_-)=2i\pi$
\color{black}

\vspace{.2cm}

\textcolor{olive}{Meridian $y$} : \ 
\color{red}
$\mathrm{H}^\C_{X,y}(\mathbf{z})= -{\rm Log}(z'_-)+{\rm Log}(z''_+)$
\color{black}

\textcolor{olive}{\textbf{Preferred longitude} $x+2y$} : \ 
\color{red}
$ \mathrm{H}^\C_{X,x+2y}(\mathbf{z})=2i\pi -4{\rm Log}(z'_-)-2{\rm Log}(z''_-)$
\color{black}

\ 

$\mathbf{G}=\begin{pNiceMatrix}[first-row, first-col]
 &  {\color{red}z_+} & {\color{red}z_-} \\
{(\color{orange}\to\color{black})} & 2 & 0  \\
{\color{olive}x+2y} & 0 & 0 
\end{pNiceMatrix}, \ \ \
\mathbf{G'}=\begin{pNiceMatrix}[first-row, first-col]
 &  {\color{red}z'_+} & {\color{red}z'_-} \\
{(\color{orange}\to\color{black})} & 1 & 2  \\
{\color{olive}x+2y} & 0 & -4 
\end{pNiceMatrix}, \ \ \ 
\mathbf{G''}=\begin{pNiceMatrix}[first-row, first-col]
 &  {\color{red}z''_+} & {\color{red}z''_-} \\
{(\color{orange}\to\color{black})} & 0 & 1  \\
{\color{olive}x+2y} & 0 & -2 
\end{pNiceMatrix}$

    \caption{Cusp triangulation of $S^3 \setminus 4_1$, and the Neumann-Zagier matrices}
    \label{fig:41:cusp:triang:NZ}
\end{figure}

Now we define $\mathbf{A}:=\mathbf{G}-\mathbf{G'}$ and $\mathbf{B}:=\mathbf{G''}-\mathbf{G'}$ (see Section \ref{NZD}).

Observe that our example satisfies the following property: 
\begin{center}
    \framebox{{\bf Property II}: $\mathbf{B}$ is invertible and $\mathbf{B}^{-1} \mathbf{A} = \mathcal{X}_0 \mathcal{A}^{-1} \mathcal{B}  + \left ( \mathcal{X}_0 \mathcal{A}^{-1} \mathcal{B} \right )^{\!\top} + \frac{\mathcal{E}+{\rm Id}_N}{2}$. }
\end{center} 
This unexpected correspondence between a function of "face matrices" and one of "edge matrices" is actually at the heart of all the currently known proofs of the Andersen--Kashaev volume conjecture. We formalize this condition in the following definition.

\begin{definition} \label{def:FAMED}
Let $M$ be a compact $3$-manifold with one toroidal boundary component and trivial second homology, such as a knot exterior.
Let $X$ be an ordered ideal triangulation of $M$, let $\mathcal{A}$ be the face adjacency matrix associated to $X$ (as above).
Let $l\in \pi_1(\partial(M))$ be a peripheral curve (such as the preferred longitude of $K$ if $M=S^3 \setminus \nu(K)$). Let $\mathbf{A}, \mathbf{B}$ be the Neumann-Zagier matrices with respect to $l$ (as above). 
We say that $X$ is FAMED (for "Face Adjacency Matrices with Edge Duality") with respect to $l$ if 
\begin{enumerate}
\item the space of angle structures (see Section \ref{sub:angle:str}) is non empty: $\mathscr{A}_{X} \neq \emptyset$,
\item $\det \mathcal{A} \neq 0$, 
\item $\det \mathbf{B}\neq 0$, 
\item $\mathbf{B}^{-1}\mathbf{A} = \mathcal{X}_0 \mathcal{A}^{-1} \mathcal{B}  + \left ( \mathcal{X}_0 \mathcal{A}^{-1} \mathcal{B} \right )^{\!\top} + \frac{\mathcal{E}+{\rm Id}_N}{2}$. 
\end{enumerate}
\end{definition}
\begin{remark}
One should not confuse the face matrix $\mathcal{A}$ (associated to the triangulation $X$) with the space of angle structures $\mathcal{A}_X$ (nor with the face $\mathsf{A}$ nor the Neumann-Zagier matrix $\mathbf{A}$...). We tried to find a good compromise between readability and faithfulness to classical notations, and we hope the reader will find we succeeded.
\end{remark}
\begin{remark}
For simplicity, in the rest of this paper, "FAMED" will mean either "FAMED for some curve $l$", "FAMED for the curve $l$ we specified before", or 
"FAMED with respect to the preferred longitude $l$" in the case of a knot complement in $S^3$.
\end{remark}

Recall that $X$ is \textit{geometric} if there exist shape parameters $\mathbf{z}$ with positive imaginary parts that satisfy Equation (\ref{equ}) with $\xi = 0$.
In the example of Thurston's triangulation of $S^3 \setminus 4_1$, it is known that
$$\framebox{{\bf Property III}: This ideal triangulation is geometric.
}
$$

Properties I, II and III respectively allow us to show that
\begin{enumerate}[I:]
\item the equivalent, distributional-free definition of the Teichm\"uller TQFT invariant in \cite{BAGPN} holds (see Lemma \ref{lem:dirac} and Proposition \ref{Tpartiexpress1}); 
\item there is a natural correspondence between the critical point equations of the potential function and the hyperbolicity equations (see Proposition \ref{critThurscorrespondence}); and
\item the potential function has a critical point in its domain of definition.
\end{enumerate}

It is not difficult to construct triangulations that are not geometric, and the same is true for the FAMED condition. Naturally, we can ask the following questions:

\begin{question}
\begin{enumerate}
    \item Does every one-cusped hyperbolic $3$-manifold admit a FAMED geometric triangulation?
    \item (Weaker question:) Does every hyperbolic knot complement in $\mathbb{S}^3$ admit a FAMED geometric triangulation?
\end{enumerate}
\end{question}

A positive answer to the first question (resp. second question) conjugated with the results of the present paper would imply that \textit{the Andersen--Kashaev volume conjecture (Conjecture \ref{conj:vol:BAGPN}) holds for every one-cusped hyperbolic $3$-manifold (resp. hyperbolic knot complement in $\SS^3$)}.

In \cite{BAGW}, with the computer, we found a positive answer to this question (and thus a proof of the Andersen--Kashaev volume conjecture) for more than 42,000 hyperbolic knots, including all knots with 12 crossings or fewer, and almost all knots whose complement can be triangulated with 23 tetrahedra or fewer (see Section \ref{sub:intro:computer}). 
These results show that FAMED triangulations have a great potential for proving volume conjectures.

We can now present our three main results, Theorems \ref{thm:Jones:FAMED}, \ref{thm:AJ} and  \ref{mainthmZ}.

\subsection{Asymptotics of Jones functions}\label{sub:intro:jones}
Consider a one-cusped hyperbolic $3$-manifold $M$ with trivial second homology and endowed with angled ideal triangulation $(X,\alpha)$ as before.
In \cite{AK,AKnew,AKicm}, it is conjectured that there exists a Jones function ${J}_{X}$, which is an analogue of the Kashaev invariant of a knot $K$ (if you think of $M$ as $S^3 \setminus K$), such that the Andersen--Kashaev TQFT is roughly a Laplace transform of this Jones function, and such that this Jones function asymptotically yields the hyperbolic volume.  More precisely, this \textbf{Andersen--Kashaev volume conjecture} can be stated as follows:

\begin{conjecture}[see \cite{AK}, Conjecture 1.(1)\&(3) and \cite{BAGPN} Conjecture 2.13] \label{conj:vol:BAGPN}
Let $\mathcal{N}$ be a connected closed oriented $3$-manifold and let $K \subset \mathcal{N}$ be a hyperbolic knot whose complement has trivial second homology. 
There exist an ideal triangulation $X$ of $M = \mathcal{N} \setminus K$ and a Jones function $J_X\colon \R_{>0} \times \mathcal{U} \to \C$ (where $\mathcal{U}\subset \C$ is an open horizontal band in $\C$) such that 
 the following properties hold:
\begin{enumerate}
\item   There exist a linear combination of dihedral angles $\lambda_X(\alpha)$ in $X$ such that for all angle structures $\alpha \in \mathcal{A}_{X}$ and all $\hbar>0$, we have:
\begin{equation*}
\left |\mathcal{Z}_{\hbar}(X,\alpha) \right |
= \left | 
\int_{\mathbb{R} +i\frac{\mu_X(\alpha)}{2\pi\sqrt{\hbar}}} 
J_{X}(\hbar,x)
e^{\frac{1}{2 \sqrt{\hbar}}  x  \lambda_{X}(\alpha)} 
dx \right |,
\end{equation*}
where $\mu_X(\alpha)$ is 
another linear combination of dihedral angles.
\item In the semi-classical limit $\hbar \to 0^+$, we retrieve the hyperbolic volume of $K$ as:
$$
\lim_{\hbar \to 0^+} 2\pi \hbar  \log \vert J_{X}(\hbar,0) \vert
= -\Vol(\mathcal{N}\backslash K).$$
\item \begin{enumerate}
    \item If there exists a simple closed curve $c$ on the boundary $\partial M$ that is trivial in the first homology of $M$ with integer coefficients, then $\lambda_X(\alpha)$ is the angular holonomy of this curve $c$.
    \item In particular,  $c$ is the preferred longitude of the knot $K$ when $\mathcal{N}=S^3$, and $c$ is the boundary of the fiber surface when $M$ is fibered over $S^1$.
    \item Moreover, $\mu_X(\alpha)$ is the angular holonomy of a meridian of the knot $K$ when $\mathcal{N}=S^3$, and $\mu_X(\alpha)$ is the angular holonomy of a curve following along the $S^1$ direction when $M$ is fibered over $S^1$.
\end{enumerate}
\end{enumerate}
\end{conjecture}

The original version of Conjecture \ref{conj:vol:BAGPN} in \cite{AK} only contained (1) and (2) where $\mu_X(\alpha)$ was instead $0$. This is due to the fact that a Cauchy argument allows to change the horizontal contour in the integral of (1); however, in \cite{BAGPN} the authors decided to emphasize that a very specific $\mu_X(\alpha)$ appeared naturally in the proof of (1) (we will expand on this in the following sections). The topological natures of $\lambda_X(\alpha)$ and $\mu_X(\alpha)$ stated in (3) were first mentioned in \cite{BAGPN} for knots in $S^3$. Since then, the results of \cite{BAG} on fibered manifolds now lead us to state (3) fully as above. The conjecture was originally stated for a pair $(N,K)$ instead of a knot complement $M$, notably because it included a statement on asymptotics of partition functions of H-triangulations of such pairs. In this paper, we did not cover the impact of the FAMED property to such asymptotics for H-triangulations, but we still chose to keep the point of view of pairs in the statement of Conjecture \ref{conj:vol:BAGPN}.

Conjecture \ref{conj:vol:BAGPN} was proven in \cite{AK} by Andersen and Kashaev for the first two hyperbolic knots $4_1$ and $5_2$, and in \cite{AN} by Andersen and Nissen for the next one $6_1$. A proof for all twist knots was done in \cite{BAGPN} by Guéritaud, Piguet-Nakazawa and the first author; this gave the first author intuition that a pure combinatorial proof was attainable, which led us to introduce the FAMED condition in this paper. Results of \cite{BAGPN} were extended to an infinite family of knots in lens spaces by Piguet-Nakazawa in \cite{PN}, to the knot $7_3$ by Uemura \cite{U}, and to once-punctured torus bundles by Guéritaud and the first author in \cite{BAG} (this last project was started before the current paper and fruitful exchanges have been made between both projects).

The Jones function is explicitly constructed for all hyperbolic twist knots in \cite{BAGPN}. Besides, due to the asymptotic properties of the quantum dilogarithm function (Proposition \ref{prop:quant:dilog} (3)), it is natural to consider the asymptotic of the invariant in the parameter $\B>0$ instead of $\hbar>0$ ($\B$ is defined such that $\hbar = \frac{1}{(\B+\B^{-1})^2}$). 

For FAMED ideal triangulations, our first main result (Theorem \ref{thm:Jones:FAMED}) provides the existence of the Jones function and its asymptotics. Let $\nu(K)$ be a tubular neighborhood of $K \subset N$. To state the theorem, given two elements $[\gamma_1],[\gamma_2] \in \mathrm{H}_1(\partial (\mathcal{N}\setminus \nu(K));\QQ)$, let $[\gamma_1]\cap [\gamma_2] \in \Q$ denote the algebraic intersection number of $[\gamma_1]$ and $[\gamma_2]$. 
\begin{theorem}
\label{thm:Jones:FAMED}
Let $M$ be a one-cusped hyperbolic $3$-manifold with trivial second homology
and let $X$ be an ordered ideal triangulation of 
$M$.
\begin{enumerate}
\item If $X$ is FAMED for a boundary curve $l$, then there exists a Jones function $\mathfrak{J}_X \colon \R_{>0} \times \mathcal{W} \to \C$ (where $\mathcal{W} \subset \C$ is an open horizontal band) such that for all angles structures $\alpha \in \mathcal{A}_X$ and all $\hbar>0$:
\begin{align*}
&|\mathscr{Z}_{\hbar}(X, \alpha)| 
= \left| 
\int_{\RR +i\mu_X(\alpha)} 
\mathfrak{J}_X(\hbar,\mathrm{x})
e^{\frac{\mathrm{x} \lambda_X(\alpha)}{4\pi {\hbar}}} d\mathrm{x} \right| ,
\end{align*}
where $\lambda_X(\alpha)=\mathrm{H}^\R_{X,l}(\alpha)$ is the angular holonomy of the 
curve $l$, and $\mu_X(\alpha)=\mathrm{H}^\R_{X,\hat{m}}(\alpha)$ is the angular holonomy of a curve $\hat{m}$ with algebraic intersection $1$ with $l$.
\item If $X$ is FAMED for $l$ and geometric, then there exist $q\in \QQ$, a peripheral curve $\hat{m}$ with $[\hat{m}]=[m+ql] \in \mathrm{H}_1(\partial M, \QQ)$ and a small neighborhood $\mathcal{O}\subset \CC$ around $0 \in \CC$ such that 
for any $\mathrm{x} \in \mathcal{O}$,
by considering $\mathbf{z}_\mathrm{x}$ the unique shape parameters with positive imaginary parts associated to the (possibly incomplete) hyperbolic structure satisfying $\mathrm{H}_{X,\hat{m}}^\C(\mathbf{z}) =\mathrm{x}$,
we have 
\begin{align*}
|\mathfrak{J}_{X}(\hbar, \mathrm{x})|
= \left|\frac{\exp\left(\frac{i}{\pi} R(\mathbf{z}_\mathrm{x})\right)}{2\pi {\sqrt{\hbar}} (\det \mathscr{A} \sqrt{ 4 \det\mathbf{B}^{-1}})} \cdot \frac{\exp\left(\frac{i}{2\pi {\hbar}}\phi_{\hat{m},l}(\mathrm{x})\right)}{\sqrt{\pm \tau(M, \hat{m}, X, \mathbf{z}_\mathrm{x})}}  \Big(1 + O_{{\hbar}\to 0^+}({\hbar})\Big)\right|,
\end{align*} 
where 
\begin{itemize}  
\item $\mathcal{A}$ is the face gluing matrix and $\mathbf{B}$ is the Neumann-Zagier matrix, both used in the definition of the FAMED condition,
\item $\phi_{\hat{m},l}: \mathcal{O} \to \CC$ is the Neumann-Zagier potential function associated to $\hat{m},l$ (see Section \ref{NZpotentintro}), i.e. the holomorphic function satisfying
$$
\frac{\partial \phi_{\hat{m},l}}{\partial w^{loc}_{\hat{m}}} = \frac{w^{loc}_l}{2} ,\quad \phi_{\hat{m},l}(0) = i(\Vol(M) + i \CS(M)),
$$ 
where $\CS(M)$ is the Chern--Simons invariant of $M$, 
\item $\tau(M, \hat{m}, X, \mathbf{z}_\mathrm{x})$ is the 1-loop invariant of $M$ associated to the shape parameters $\mathbf{z}_\mathrm{x}$ and $\hat{m}
\in \mathrm{H}_1(\partial M,\QQ)$ (see Section \ref{1loopsection} for details), and 
\item $R$ is some analytic function independent of $\hbar$ (see Proposition \ref{hvalue} for the definition). 
\end{itemize}
In particular, we have
$$
 \lim_{\hbar\to 0} 2\pi \hbar \log|\mathfrak{J}_{X}(\hbar, 0)|
= - \Vol(M).
$$
\end{enumerate}
\end{theorem}
\begin{remark}
The functions $\mathfrak{J}_X(\mathrm{x})$ in Theorem \ref{thm:Jones:FAMED} and $J_X(x)$ in Conjecture \ref{conj:vol:BAGPN} are related by a change of variables $\mathrm{x}=2\pi \sqrt{\hbar} x$. 
\end{remark}
\begin{remark}
The rational number $q$ in Theorem~\ref{thm:Jones:FAMED} (2) can be computed explicitly by using the Neumann-Zagier datum of the triangulation. See Lemma \ref{xmeaning} for details. For the triangulations $X_n$ for hyperbolic twist knot complements constructed in \cite{BAGPN}, we have $q=0$, i.e. $\hat{m}$ is indeed the meridian $m$ of the knot.
\end{remark}

Theorem \ref{thm:Jones:FAMED} will be proved in Section \ref{sec:Jones}.
As a corollary, the Andersen--Kashaev volume conjecture is true for all hyperbolic knot complements in $\SS^3$ which admit a FAMED geometric triangulation:

\begin{corollary}\label{AEFJXn}
Conjecture \ref{conj:vol:BAGPN} is true for all FAMED geometric triangulations, including the triangulation $X_n$ of hyperbolic twist knot complements constructed in \cite{BAGPN}.
\end{corollary}

\subsection{Jones functions and the AJ-conjecture}\label{sub:intro:AJ}

Let us digress about the AJ-conjecture before going back to asymptotics of the Andersen--Kashaev TQFT in the next section.

Based on Corollary \ref{AEFJXn}, we study the connection between the asymptotic expansion conjecture of the Jones functions, the annihilating polynomial of the Jones function and the classical A-polynomial of knots.
It would be interesting to compare this result
with Andersen--Malusa's version of the AJ conjecture on the Teichmüller TQFT stated in \cite[Conjecture 1]{AM}, which studies a general level $\mathfrak{N}$ Jones function (in this paper we restrict to the case $\mathfrak{N}=1$).

To be precise, let $K$ be an hyperbolic knot in $\SS^3$.
We consider the set of functions 
$$ \mathscr{F} = \{f \mid f: \RR_{>0} \times \CC \to \CC \}. $$
Let $\hat{M}, \hat{L}: \mathscr{F} \to \mathscr{F}$ be two operators defined by
$$ \hat{M}( f(\hbar, \mathrm{x}) ) = e^{\mathrm{x}} f(\hbar,  \mathrm{x}), \quad \hat{L}( f(\hbar,  \mathrm{x}) ) = f(\hbar,  \mathrm{x} + 4\pi i\hbar). $$
Let $q=e^{4\pi i \hbar}$. Note that $\hat{L}\hat{M} =  q\hat{M} \hat{L} $. In particular, $\hat{M}, \hat{L}$ define a quantum torus
$$ \mathscr{T} = \langle \hat{M}, \hat{L} \mid \hat{L}\hat{M} =  q\hat{M} \hat{L}  \rangle.$$
Let $A_K(M,L)$ be the $PSL(2;\CC)$ A-polynomial of $K$ (see Section \ref{PSL2CApoly} for a brief review).

Our second main result (Theorem \ref{thm:AJ}) provides a partial positive answer to a version of the AJ-conjecture for knot complements with FAMED geometric triangulations: if an annihilator polynomial exists for the Jones function, then an evaluation of it at $q=1$ is a multiple of the geometric component of the A-polynomial of the knot. A similar connection between the asymptotics of the relative Reshetikhin-Turaev invariants and their annihilators will be discussed in an another preprint by G. Belletti and the second author.

\begin{theorem}\label{thm:AJ}
Let $K$ be an hyperbolic knot in $\SS^3$ and $X$ be an ideal ordered triangulation of $\SS^3 \setminus K$.
Suppose there exists $\hat{A}_K(\hat{M}, \hat{L}, q) = \sum_{n=0}^d a_n(\hat{M},q) \hat{L}^n \in \ZZ[ \hat{M}^{\pm 1}, \hat{L}, q]$ such that 
$$ \hat{A}_K(\hat{M}, \hat{L}, q) (\mathfrak{J}_X) (\hbar, \mathrm{x})
= \sum_{n=0}^d a_n(\hat{M},q) (\mathfrak{J}_X(\hbar, \mathrm{x} +4\pi i n \hbar))
= 0.$$
Furthermore, assume that there exists an open set $\mathscr{O}\subset \CC$ containing $0$ such that as $\hbar \to 0$, the asymptotic expansion formula of $\mathfrak{J}_{X}(\hbar, \mathrm{x})$ is of the form
\begin{align*}
\mathfrak{J}_{X}(\hbar, \mathrm{x})
= \frac{C(\mathrm{x})}{\sqrt{\hbar}} \cdot \exp\Big(\frac{i}{2\pi \hbar} \phi_{m,l}(\mathrm{x})\Big) \Big(1 + O(\hbar)\Big),
\end{align*}
where $ \phi_{m,l}(\mathrm{x})$ is the Neumann-Zagier potential function of $\SS^3 \setminus K$ with respect to the meridian $m$ and the preferred longitude $l$, and $C(\mathrm{x})$ is some non-zero quantity that depends continuously on $\mathrm{x}$. 
Then the polynomial $\hat{A}_K(M, L, 1)$ is divisible by the geometric component of $A_K(M,L)$.
\end{theorem}

Theorem \ref{thm:AJ} will be proved in Section \ref{sub:AJ}.

Let $X_n$ be the ideal triangulation of the $n$-th hyperbolic twist knot complement $\SS^3 \setminus K_n$ constructed in \cite{BAGPN}. The following corollary sheds some light on the AJ conjecture for the Andersen--Kashaev TQFT of the twist knots.

\begin{corollary}
Suppose there exists $\hat{A}_{K_n}(\hat{M}, \hat{L}, q) = \sum_{n=0}^d a_n(\hat{M},q) \hat{L}^n \in \ZZ[ \hat{M}^{\pm 1}, \hat{L}, q]$ such that 
$$ \hat{A}_{K_n}(\hat{M}, \hat{L}, q) (\mathfrak{J}_{X_n}(\hbar, \mathrm{x})) 
= \sum_{n=0}^d a_n(\hat{M},q) (\mathfrak{J}_{X_n}(\hbar, \mathrm{x} + 4\pi i n \hbar))
= 0.$$
Then the polynomial $ \hat{A}_{K_n}(M, L, 1)$ is divisible by the geometric component of the $A_{K_n}(M,L)$.
\end{corollary}

\subsection{Asymptotics of the Teichm\"{u}ller TQFT partition functions}\label{sub:intro:expansion}

Conjecture \ref{conj:vol:BAGPN} retrieves the volume as an asymptotics of the Jones function extracted from the Andersen--Kashaev TQFT. Seeing this, it is natural to ask if a similar geometric quantity can be extracted directly from the partition function $\mathscr{Z}_\hbar(X,\alpha)$ (see also Kashaev's conjecture at the end of \cite{KaICM}). 
In the third main result of this paper, Theorem \ref{mainthmZ}, we prove such an asymptotics, where we also see the appearances of logarithmic holonomies and the $1$-loop invariant.

\begin{theorem}\label{mainthmZ}
Let $M$ be a one-cusped hyperbolic $3$-manifold with trivial second homology
and let $X$ be an ordered ideal triangulation of 
$M$.
Let $\alpha \in \mathscr{A}_X$ be an angle structure of $X$. 
\begin{enumerate}
\item 
Suppose $X$ is FAMED for the boundary curve $l$. Then we have
\begin{align*}
\limsup_{\hbar\to 0} 2\pi {\hbar} \log|\mathscr{Z}_{\hbar}(X, \alpha) |
\leq -\sup\{ \Vol(\alpha') \mid \alpha' \in \mathscr{A}_X, 
\mathrm{H}^\R_{X,l}(\alpha')=\mathrm{H}^\R_{X,l}(\alpha)
\},
\end{align*}
where $\mathrm{H}^\R_{X,l}(\alpha)$ is the angular holonomy associated to the curve $l$ and the angle structure $\alpha$, and $\Vol(\alpha')$ is the sum of hyperbolic volumes of the tetrahedra of $X$ under the angle structure $\alpha'$.
\item 
Suppose $X$ is FAMED for the boundary curve $l$. Further assume that there exists shape parameters $\mathbf{z}$ with positive imaginary parts that satisfy Equation (\ref{equ}) with $\xi = i\mathrm{H}^\R_{X,l}(\alpha)$.
Then we have
$$|\mathscr{Z}_{\hbar}(X, \alpha) |
=
\left| \frac{\exp\left( \frac{i}{\pi}R(\mathbf{z})\right)}{\det(\mathcal{A}) \sqrt{ 2 \det\mathbf{B}^{-1}}} \cdot \frac{\exp\Big( -\frac{1}{2\pi {\hbar}}\Vol\left(M; l,\mathrm{H}^\R_{X,l}(\alpha)\right)\Big)}{\sqrt{\pm \tau(M, l, X, \mathbf{z})}}  \Big(1 + O({\hbar})\Big) \right|,$$
where 
\begin{itemize}
\item $m$ is a peripheral curve with algebraic intersection $1$ with $l$,
\item $\mathcal{A}$ is the face gluing matrix and $\mathbf{B}$ is the Neumann-Zagier matrix, both used in the definition of the FAMED condition,
\item $\Vol\left (M; l,\mathrm{H}^\R_{X,l}(\alpha)\right)=\Vol(\mathbf{z})$ is the volume of $M$ with the (possibly incomplete) hyperbolic cone structure satisfying $\mathrm{H}^\C_{X,l}(\mathbf{z})=i \mathrm{H}^\R_{X,l}(\alpha)$, i.e. the sum of hyperbolic volumes of the tetrahedra of $X$ under the shape structure $\mathbf{z}$,
\item $\tau(M, l, X, \mathbf{z})$ is the 1-loop invariant of $\SS^3 \setminus K$ associated to the shape parameters $\mathbf{z}$ and the peripheral curve $l$ (see Section \ref{1loopsection} for details), and
\item $R$ is the same analytic function as in Theorem \ref{thm:Jones:FAMED}(2). 
\end{itemize}
\item In particular, we have
\begin{align*}
\lim_{\hbar\to 0} 2\pi \hbar \log|\mathscr{Z}_{\hbar}(X, \alpha) |
= -\Vol\left(M; l,\mathrm{H}^\R_{X,l}(\alpha)\right ).
\end{align*}
\end{enumerate}
\end{theorem}

Theorem \ref{mainthmZ} will be proven in Section \ref{sec:asymp:expan}.
Following Theorem \ref{mainthmZ}, we propose the following asymptotic expansion conjecture for the Teichm\"{u}ller TQFT partition function $\mathscr{Z}_{\hbar}(X, \alpha)$:
\begin{conjecture}\label{conj:expansion:TQFT}
Let $M$ be a one-cusped hyperbolic $3$-manifold with trivial second homology
and let $X$ be an ordered ideal triangulation of 
$M$.
Let $\alpha \in \mathscr{A}_X$ be an angle structure of $X$.    

Then, following the notations of Theorem \ref{mainthmZ} (but not necessarily the assumptions on $X$), we have the asymptotic expansion
$$|\mathscr{Z}_{\hbar}(X, \alpha) |=
\left| \frac{\exp\left( -\frac{i}{\pi} R(\mathbf{z}) \right)}{\det(\mathcal{A})\sqrt{ 2 \det\mathbf{B}^{-1}}} \cdot \frac{\exp\Big( -\frac{1}{2\pi {\hbar}}\Vol(M; l,\alpha)\Big)}{\sqrt{\pm \tau(M, l, X, \mathbf{z})}}  \Big(1 + O({\hbar})\Big) \right|.$$
\end{conjecture}

\begin{remark}\label{rmk1loopconj}
The 1-loop conjecture proposed by Dimofte-Garoufalidlis suggests that the 1-loop invariant coincides with the adjoint twisted Redeimeister torsion \cite{DG} (see Section \ref{1loopsection} and Conjecture \ref{1loopconjstatement} for details). If the 1-loop conjecture is true, then we can replace the 1-loop invariant in Theorem \ref{thm:Jones:FAMED}, Theorem \ref{mainthmZ} and Conjecture  \ref{conj:expansion:TQFT} by the adjoint twisted Redeimeister torsion.
\end{remark}

\begin{remark}
As a new illustration of the rich interplay between the theoretical results of the current paper and the experimental tests in \cite{BAGW}, 
    the tens of thousand of computer tests done in \cite{BAGW} 
    (and in other ongoing projects with Guilloux and Guéritaud)
    suggest that $2 (\det (\mathcal{A}))^2 \det\mathbf{B}^{-1}$ is always an integer depending at least on some homology torsion of the manifold, and is equal to $\pm 1$ for all the ideal triangulations $X_n$ for hyperbolic twist knots constructed in \cite{BAGPN}. One can thus expect variants of Theorem \ref{mainthmZ} and Conjecture \ref{conj:expansion:TQFT} to hold, where the asymptotic expansion depends even less of the triangulation $X$.
\end{remark}

Finally, let us mention that Theorem \ref{mainthmZ} holds for the infinite family of hyperbolic twist knots, which expands the results of \cite{BAGPN}.

\begin{proposition}\label{prop:twist:FAMED}
The ideal triangulations $X_n$ for hyperbolic twist knot complements $\SS^3 \setminus K_n$ constructed in \cite{BAGPN} are FAMED for the preferred longitude $l_{K_n}$ and geometric. 
\end{proposition}

Proposition \ref{prop:twist:FAMED} follows from the results of \cite{BAGPN} and the computations done in Section \ref{ITXn}. Combined with Theorem \ref{mainthmZ} (2),  we obtain:

\begin{corollary}
Conjecture \ref{conj:expansion:TQFT} holds for the ideal triangulations $X_n$ of the twist knot complements $\SS^3 \setminus K_n$  constructed in \cite{BAGPN},
for $l=l_{K_n}$ the preferred longitude, and $m$ a meridian of $K_n$.
\end{corollary}

\subsection{Approaching the Casson conjecture via FAMED triangulations}\label{sub:intro:casson}

Next, we explore the connection between the exponential decay rate of the invariant and the Casson conjecture (Conjectures \ref{Cassonconj} and \ref{Cassonconj2}). Let $M$ be a 3-manifold with toroidal boundary whose interior admits a complete hyperbolic metric. Let $X$ be an ideal triangulation of the interior of $M$ with $\mathscr{A}_X \neq \emptyset$. Let $\gamma$ be a simple closed curve on the boundary torus $\partial M$ and let 
$\mathrm{H}^\R_{X,\gamma}(\alpha)$
be the angular holonomy of $\gamma$ at the angle structure $\alpha$.
For any $\theta \in \R$, define $\mathscr{A}^\gamma_{X}(\theta):=  \{ \alpha \in \mathscr{A}_X \mid \mathrm{H}^\R_{X,\gamma}(\alpha) = \theta \}$.
The space of angle structures can thus be decomposed into the disjoint union 
$$ \mathscr{A}_X = \bigsqcup_{\theta \in \RR} \mathscr{A}^\gamma_{X}(\theta) = \bigsqcup_{\theta \in \RR} \{ \alpha \in \mathscr{A}_X \mid \mathrm{H}^\R_{X,\gamma}(\alpha) = \theta \}. $$
The existence of hyperbolic cone structures with angle $\theta$ is closely related to the maximization problem of the volume functional on the subspace $\mathscr{A}_{X}^{\gamma}(\theta)$. More precisely, the volume functional has a maximum point on $\mathscr{A}_{X}^{\gamma}(\theta)$ if and only if there exist shape parameters $\mathbf{z}=\mathbf{z_{\gamma,\theta}}$ with positive imaginary parts satisfying $\mathrm{H}^\C_{X,\gamma}(\textbf{z}) = i \theta$. In this case, we have
$$  \sup\{ \Vol(\alpha) \mid \alpha \in \mathscr{A}_{X}^{\gamma}(\theta) \}  
= \max\{ \Vol(\alpha) \mid \alpha \in \mathscr{A}_{X}^{\gamma}(\theta) \}
= \Vol(M;\gamma,\theta),
$$
where $\Vol(M;\gamma,\theta)=\Vol(\mathbf{z_{\gamma,\theta}})$ is the volume of $M$ with the hyperbolic cone structure with $\mathrm{H}^\C_{X,\gamma}(\textbf{z}) = i \theta$.

If the volume functional does not attain its maximum, the Casson conjecture (so called by Luo in \cite{Luo}, see also the last paragraph in \cite[Section 6.4]{FG}) predicts that the volume functional is bounded above by the hyperbolic volume of the manifold.
\begin{conjecture}[Casson conjecture, Problem 5 in \cite{Luo}]\label{Cassonconj} Let $M$ be a 3-manifold with toroidal boundary whose interior admits a complete hyperbolic metric. Let $X$ be an ideal triangulation of the interior of $M$. Suppose the space of angle structures $\mathscr{A}_X$ of $X$ is non-empty. Then we have
$$
 \sup\{ \Vol(\alpha) \mid \alpha \in \mathscr{A}_X \} \leq  \Vol(M).
$$
\end{conjecture}
It is natural to ask the same question on each subspace $\mathscr{A}_{X}^{\gamma}(\theta)$, which lead us to state a "level Casson conjecture":
\begin{conjecture}\label{Cassonconj2} Let $M$ be a 3-manifold with toroidal boundary whose interior admits a complete hyperbolic metric. Let $X$ be an ideal triangulation of the interior of $M$.
Let $\gamma$ be a simple closed curve as before.
Let $\theta\in \RR$ with $\mathscr{A}_{X}^{\gamma}(\theta)\neq \emptyset$. Suppose $M$ has a hyperbolic cone structure with $\mathrm{H}^\C_{X,\gamma}(\textbf{z}) = i \theta$. Then we have
$$
 \sup\{ \Vol(\alpha) \mid \alpha \in \mathscr{A}_{X}^{\gamma}(\theta) \} \leq 
 \Vol(M;\gamma,\theta).
$$
\end{conjecture}
The following corollary to Theorem \ref{mainthmZ} provides the following partial results on Conjecture \ref{Cassonconj2}.
\begin{corollary}\label{coroCassonconj2}
Let $M$ be a one-cusped hyperbolic $3$-manifold with trivial second homology
and let $X$ be an ordered ideal triangulation of 
$M$.
Let $\alpha \in \mathscr{A}_X$ be an angle structure of $X$ and $l$ be a peripheral curve. Suppose $X$ is FAMED with respect to $l$ and Conjecture \ref{conj:expansion:TQFT} is true for $(X,\alpha)$, then we have
$$ \sup\{ \Vol(\alpha') \mid \alpha' \in \mathscr{A}_{X}^l (\mathrm{H}^\R_{X,l}(\alpha))\} \leq 
\Vol\left(M;l,\mathrm{H}^\R_{X,l}(\alpha)\right).
$$
In other words, $X$ satisfies Conjecture \ref{Cassonconj2} for the preferred curve $l$.
\end{corollary}
See Section \ref{sub:proofs:thm15} for the proof of Corollary \ref{coroCassonconj2}.

\subsection{Computational results about FAMED triangulations}\label{sub:intro:computer}
In \cite{BAGW}, A. Guilloux and the authors checked the FAMED conditions for the 
59,937
triangulations of knot complements in the Snappy census HTLinkExteriors. We have the following observations at the time of writing.
\begin{enumerate}
\item For 42,117 knots of the census, including all knots with at most 12 crossings and all but six knots whose census triangulation has at most 23 tetrahedra, we could find a FAMED geometric triangulation of the knot complement, up to retriangulation if need be. This yields certified  proofs of the Andersen--Kashaev volume conjecture for these 42,117 knots.
\item We observed that conditions (2) and (3) in Definition \ref{def:FAMED} were always satisfied at the same time, and when they were satisfied, so was condition (4).
\item When the triangulation of a knot complement is FAMED, we always observe  we have $\det \mathcal{A} = \pm 1$ and $\det \mathbf{B} = \pm 2$.
    \item The FAMED condition is not equivalent to geometricity. There exists a triangulation of a knot complement which is geometric but not FAMED. Similarly, there exists a triangulation of a knot complement which is FAMED but not geometric.
    \item There exists a knot complement with a FAMED triangulation and a non-FAMED triangulation.
\end{enumerate}

\subsection{Outline of the paper}
This paper is organized as follows. In Section \ref{sectprelim}, we review the preliminary results required to understand the rest of the paper. In Section \ref{sec:asymp:expan}, we study the asymptotics of the partition function. To do that, in Section \ref{sect:expression}, we first utilize the FAMED property to find a nice expression of the partition function in terms of Neumann-Zagier matrices (Proposition \ref{Tpartiexpress2}). In Section \ref{sect:potential:prop}, we study various properties of the potential function required for applying the saddle point approximation, including its critical point equation (Proposition \ref{critThurscorrespondence}), the critical value (Proposition \ref{dilogVolgen}), the concavity (Proposition \ref{concavSgen}) and the determinant of the Hessian matrix (Proposition \ref{Hesstotor}). Then we prove Theorem \ref{mainthmZ} and Corollary \ref{coroCassonconj2} in Section \ref{sub:proofs:thm15}. In Section \ref{subsec:existJ}, we prove the existence of the Jones function in Proposition \ref{pfexistJ}. In Section \ref{subsect:potentialpropJ}, we study the properties of the potential function, including its relationship with the Neumann-Zagier potential function (Proposition \ref{JandNZ}), concavity (Proposition \ref{concavSxgen}) and the determinant of the Hessian (Proposition \ref{1loopJ}). We prove Theorem \ref{thm:Jones:FAMED} and Theorem \ref{thm:AJ} in Section \ref{subsect:AEFJ} and Section \ref{sub:AJ} respectively. Finally, we prove Proposition \ref{prop:twist:FAMED} in Section \ref{ITXn}.

\section*{Acknowledgements}
We thank the organizers of the AIM workshop "Quantum invariants and low-dimensional topology", where the authors met and initiated this research project. We would like to thank  Stéphane Baseilhac, François Guéritaud, Antonin Guilloux, Rinat Kashaev, Feng Luo, Clément Maria, Gregor Masbaum, Andrew Nietzke, Saul Schleimer and Tian Yang for valuable discussions and suggestions. We thank Federica Fanoni for suggesting the acronym \textit{FAMED}.

\section{Preliminaries}\label{sectprelim}

\subsection{Triangulations}

In this section we follow \cite{AK, KaWB}.
A tetrahedron $T$ with faces $\mathsf{A},\mathsf{B},\mathsf{C},\mathsf{D}$ will be denoted as in Figure \ref{fig:tetrahedron}, where the face outside the circle represents the back face and the center of the circle is the opposite vertex pointing towards the reader. We always choose an \textit{order} on the four vertices of $T$ and we call them $0_T,1_T,2_T,3_T$ (or $0,1,2,3$ if the context makes it obvious). Consequently, if we rotate $T$ such that $0$ is in the center and $1$ at the top, then there are two possible places for vertices $2$ and $3$; we call $T$ a \textit{positive} tetrahedron if they are as in Figure \ref{fig:tetrahedron}, and \textit{negative} otherwise. We denote $\varepsilon(T) \in \{ \pm 1\}$ the corresponding \textit{sign} of $T$. We \textit{orient the edges} of $T$  {according} to the order on vertices, and we endow each edge with a parametrisation by $[0,1]$ respecting the orientation. Note that such a structure was called a \textit{branching} in \cite{BB} and equivalently described as an orientation of edges without $3$-cycles on faces.

Thus, up to isotopies fixing the $1$-skeleton pointwise, there is only one way of \textit{gluing} two triangular faces together while \textit{respecting the order of the vertices} and the edge parametrisations, and that is the only type of face gluing we consider in this paper.

\begin{figure}[h]
\centering
\begin{tikzpicture}

\begin{scope}[xshift=0cm,yshift=0cm,rotate=0,scale=1.5]

\draw (0,-0.15) node{\scriptsize $0$} ;
\draw (0,1.15) node{\scriptsize $1$} ;
\draw (1,-0.55) node{\scriptsize $2$} ;
\draw (-1,-0.55) node{\scriptsize $3$} ;
\draw (1,1) node{$\mathsf{A}$} ;
\draw (0,-0.6) node{$\mathsf{B}$} ;
\draw (-0.5,0.3) node{$\mathsf{C}$} ;
\draw (0.5,0.3) node{$\mathsf{D}$} ;

\path [draw=black,postaction={on each segment={mid arrow =black}}]
(0,0)--(-1.732/2,-0.5);

\path [draw=black,postaction={on each segment={mid arrow =black}}]
(0,0)--(0,1);

\path [draw=black,postaction={on each segment={mid arrow =black}}]
(0,0)--(1.732/2,-0.5);

\draw[->](1.732/2,-0.5) arc (-30:-90:1);
\draw (0,-1) arc (-90:-150:1);

\draw[->](0,1) arc (90:30:1);
\draw (1.732/2,0.5) arc (30:-30:1);

\draw[->](0,1) arc (90:150:1);
\draw (-1.732/2,0.5) arc (150:210:1);

\end{scope}

\end{tikzpicture}
\caption{The positive tetrahedron $T$} 
\label{fig:tetrahedron}
\end{figure}

Note that a tetrahedron $T$ like in Figure \ref{fig:tetrahedron} will 
represent 
an \textit{ideal} tetrahedron homeomorphic to a $3$-ball minus $4$ points in the boundary.

A \textit{triangulation} $X=(T_1,\ldots,T_N,\sim)$ is the data of $N$ distinct tetrahedra $T_1, \ldots, T_N$ and an equivalence relation $\sim$ first defined on the faces by pairing and the only gluing that respects vertex order, and also induced on edges  {and} vertices by the combined identifications. We call $M_X$ the (pseudo-)$3$-manifold 
$ M_X = T_1 \sqcup \cdots \sqcup T_N / \sim$ obtained by quotient. Note that $M_X$ may fail to be a manifold only at  (the image by the quotient map of) a vertex  of the triangulation, whose regular  {neighbourhood} might not be a $3$-ball (but for instance a cone over a torus for exteriors of links).

We denote $X^{k}$ (for $k=0, \ldots, 3$) the set of $k$-cells of $X$ after identification by $\sim$. In this paper we always  {assume} that \textit{no face is left unpaired  by $\sim$}, thus $X^{2}$ is always of  {cardinality} $2N$. By a slight abuse of notation we also call $T_j$ the $3$-cell inside the tetrahedron $T_j$, so that $X^{3} = \{T_1, \ldots, T_N\}$. Elements of $X^{1}$ are usually represented by distinct types of arrows, which are drawn on the corresponding preimage edges, see Figure 
\ref{fig:41:face:matrices}
for an example.

An \textit{ideal triangulation} $X$ contains ideal tetrahedra, and in this case the quotient space minus its vertices $M_X \setminus X^0$ is an open manifold. In this case we will denote $M=M_X \setminus X^0$ and say that the open manifold $M$ admits the ideal triangulation $X$.

Finally, for $X$ a triangulation and $k=0,1,2,3,$ we define $x_k\colon X^3 \to X^2$ the  {function} such that $x_k(T)$ is the equivalence class of the face of $T$ opposed to its vertex $k$.

\begin{example}\label{ex:41}
Figure 
\ref{fig:41:face:matrices}
displays the classical ideal triangulation of the complement of the figure-eight knot $M=S^3 \setminus 4_1$, with one positive and one negative tetrahedron. Here $X^3 =\{T_+, T_-\}$, $X^2 =\{\mathsf{A},\mathsf{B},\mathsf{C},\mathsf{D}\}$, $X^1 =\{ \sarrow , \darrow \}$ and $X^0$ is a singleton. 
\end{example}

\subsection{Angle structures}\label{sub:angle:str}

For a given triangulation $X=(T_1,\ldots,T_N,\sim)$ we denote $\mathcal{S}_X$ the set of \textit{shape structures on $X$}, defined as
$$
\mathcal{S}_X 
 = 
 \left \{ 
\alpha = \left (a_1,b_1,c_1,\ldots,
a_N,b_N,c_N\right ) \in (0,\pi)^{3N} \ \big | \
\forall k \in \{1,\ldots,N\}, \
a_k+b_k+c_k = \pi
\right \}.
$$
An angle $a_k$ (respectively $b_k,c_k$) represents the value of a dihedral angle on the edge $\overrightarrow{01}$ (respectively $\overrightarrow{02}$, $\overrightarrow{03}$) and its opposite edge in the tetrahedron $T_k$. If a particular shape structure $\alpha=(a_1,\ldots,c_N)\in \mathcal{S}_X$ is fixed, we define three associated  {function}s $\alpha_j\colon X^3 \to (0,\pi)$ (for $j=1,2,3$) that send $T_k$ to the $j$-th element of $\{a_k,b_k,c_k\}$ for each $k \in \{1,\ldots,N\}$.

Let $(X,\alpha)$ be a triangulation with a shape structure as before. We denote $\omega_{X,\alpha}\colon X^1 \to \mathbb{R}$ the associated \textit{weight function}, which sends an edge $e\in X^1$ to the sum of angles $\alpha_j(T_k)$ corresponding to tetrahedral edges that are preimages of $e$ by $\sim$. 
For example, if we denote 
$\alpha=(a_+,b_+,c_+,a_-,b_-,c_-)$ a shape structure on the triangulation $X$ of Figure \ref{fig:41:face:matrices}, then $\omega_{X,\alpha}(\sarrow) =
2 a_+ + c_+ + 2 b_- + c_-.$

We finally define $
\mathcal{A}_X := \left \{
\alpha \in \mathcal{S}_X \ \big | \
\forall e \in X^1, \ \omega_{X,\alpha}(e)=2\pi
\right \}
$
 the set of \textit{balanced shape structures on $X$}, or \textit{angle structures on $X$}.

\subsection{Thurston's complex gluing equations}\label{sub:thurston}

To a shape structure $(a,b,c)$ on an ordered tetrahedron $T$ (i.e.\ an element of $(0,\pi)^3$ of coordinate sum $\pi$) we can associate bijectively a \textit{complex shape structure} $z \in \R+i\R_{>0}$, as well as two companion complex numbers of positive imaginary part
$$z':=\frac{1}{1-z} \text{\ and \ } z'':=\frac{z-1}{z}.$$
Each of the $z, z', z''$ is associated to an edge, in a slightly different way according to $\varepsilon(T)$:
\begin{itemize}
\item In all cases, $z$ corresponds to the same two edges as the angle $a$.
\item If $\varepsilon(T)=1$, then $z'$ corresponds to $c$ and $z''$ to $b$.
\item If $\varepsilon(T)=-1$, then $z'$ corresponds to $b$ and $z''$ to $c$.
\end{itemize}
Another way of phrasing it is that $z, z', z''$ are always in a counterclockwise order around a vertex, whereas $a,b,c$ need to follow the specific vertex ordering of $T$.

In this article we will use the following definition of the complex logarithm:
\[
\Log(z) := \log\vert z \vert + i\arg(z)  \ \textrm{for} \ z \in \C^{*},
\]
where $\arg(z) \in (-\pi,\pi]$.

We now introduce a third way of describing the shape associated to a tetrahedron $T$, by the complex number
$$y := -\Log(z) + i \pi \in \R  + i(0, \pi),$$
 which lives in a horizontal strip of the complex plane. In particular, we have 
 $$z = -e^{-y}, \ \ \Log(z) = -y + i\pi, \ \ \Log(z') = -\Log(1+e^{-y}) \ \ \text{and} \ \ \Log(z'') = -\Log(1+e^{y}).$$
For clarity, we have the diffeomorphism
$$\psi_T\colon \R+i\R_{>0} \to \R + i (0,\pi), \ z \mapsto -\Log(z) + i \pi,$$
and its inverse
$$\psi^{-1}_T\colon \R + i (0,\pi) \to \R+i\R_{>0}, \ y \mapsto -e^{-y}.$$

We can now define the \textit{complex weight function}
$\omega^{\C}_{X,\alpha}\colon X^1 \to \C$ associated to a triangulation $X$ and an angle structure $\alpha \in \mathcal{A}_X$, which sends an edge $e \in X^1$ to the sum of logarithms of complex shapes associated to preimages of $e$ by $\sim$. For example, for the triangulation $X$ of Figure 
\ref{fig:41:face:matrices} and an angle structure $\alpha=(a_+,b_+,c_+,a_-,b_-,c_-)$, we have:
\begin{align*}
\omega^{\C}_{X,\alpha}(\sarrow) &=
2 \Log(z_+) + \Log(z'_+) + 2 \Log(z'_-) + \Log(z''_-)\\
&= \log\left (
\dfrac{\sin(c_+)^2 \sin(b_+) \sin(c_-)^2 \sin(a_-)}
{\sin(b_+)^2 \sin(a_+) \sin(a_-)^2 \sin(b_-)}
\right ) +i \omega_{X,\alpha}(\sarrow).
\end{align*}

Let $\mathbb{T}$ denote one toroidal boundary component of a $3$-manifold $M$ ideally triangulated by $X=(T_1,\ldots,T_N,\sim)$, and $\sigma$ an oriented normal closed curve in $S$. 
Truncating the tetrahedra $T_j$ at each vertex yields a triangulation of $S$ by triangles coming from vertices of $X$ (called the \textit{cusp triangulation}, see Figure \ref{fig:41:cusp:triang:NZ} for an example).
If the curve $\sigma$ intersects these triangles transversely (without back-tracking), then $\sigma$  cuts off corners of each such encountered triangle. Let us then denote $(z_1,\ldots,z_l)$ the sequence of (abstract) complex shape variables associated to these corners (each such $z_k$ is of the form $z_{T_{j_k}}, z'_{T_{j_k}}$ or $z''_{T_{j_k}}$).
Following \cite{FG}, we define the
 \textit{complex (or logarithmic)} holonomy $\mathrm{H}^\C_{X,\sigma}(\mathbf{z})$  as 
$\mathrm{H}^\C_{X,\sigma}(\mathbf{z}):= \sum_{k=1}^l \epsilon_k \Log(z_k),$
where $\epsilon_k$ is $1$ if the $k$-th cut corner lies on the left of $\sigma$ and $-1$ if it lies on the right. The \textit{angular holonomy} $\mathrm{H}^\R_{X,\sigma}(\alpha)$ of $\sigma$ is similarly defined, replacing the term $\Log(z_k)$ by the (abstract) angle  $d_k=\arg(z_k)=\Im (\Log (z_k))$  (which is of the form $a_{T_{j_k}}$, $b_{T_{j_k}}$ or $c_{T_{j_k}}$) lying in the $i$-th corner. The \textit{complex gluing edge equations} associated to $X$ consist in asking that the complex holonomies of each closed curve in $\partial M$ circling a vertex of the induced boundary triangulation are all equal to $2i\pi$, or in other words that
$$\forall e\in X^1, \omega_{X,\alpha}^\C(e) = 2i \pi.$$
The \textit{complex completeness equations} require that the complex holonomies of all curves generating the first homology $H_1(\partial M)$ vanish (when $M$ is of toroidal boundary).
Compare with Section \ref{sub:intro:FAMED} and Figure \ref{fig:41:cusp:triang:NZ}.

If $M$ is an orientable $3$-manifold with boundary consisting of tori, and  ideally triangulated by $X$, then an angle structure $\alpha \in \mathcal{A}_X$ corresponds to the complete hyperbolic metric on the interior of $M$ (which is unique) if and only if $\alpha$ satisfies the complex gluing edge equations and the complex completeness equations. 
We say that the triangulation $X$ is \textit{geometric} if such a unique $\alpha$ exists.

More generally, let $\partial M =\mathbb{T}_1 \coprod \dots \coprod \mathbb{T}_k$ and let $\gamma_j$ be an oriented normal closed curve in $\mathbb{T}_j$ for $j=1,\dots, k$. An angle structure $\alpha \in \mathcal{A}_X$ corresponds to the (possibly incomplete) hyperbolic metric on the interior of $M$ with $\mathrm{H}^\C_{X,\gamma_j}(\mathbf{z}) = i \theta_j$ for $j=1,\dots k$ if and only if $\alpha$ satisfies the complex gluing edge equations and the \emph{(complex) holonomy equations} $\mathrm{H}^\C_{X,\gamma_j}(\mathbf{z})  = i\theta_j$ for $j=1,\dots,k$. It is proved in \cite{Ko} that if $\theta_j <\pi$ for all $j$, then such a hyperbolic structure is unique if it exists.

\subsection{The classical dilogarithm}\label{Rogerdilog}

For the dilogarithm function, we will use the definition:
$$ \Li(z) := - \int_0^z \Log(1-u) \frac{du}{u} \ \ \ \textrm{for} \ z \in \C \setminus [1,\infty)$$
(see for example \cite{Za}).
For $z$ in the unit disk, $\Li(z)=\sum_{n\geq 1} n^{-2} z^n$.
We will use the following properties of the dilogarithm function, referring for example to \cite[Appendix A]{AH} for the proofs.

\begin{proposition}[Some properties of $\Li$]\label{prop:dilog}
\

\begin{enumerate}
\item (inversion relation) $$ \forall z \in \C \setminus [1,\infty), \
 \Li\left (\frac{1}{z}\right ) = - \Li(z) - \frac{\pi^2}{6} - \frac{1}{2}\Log(-z)^2.
$$
\item (integral form) For all $y \in \R +i(-\pi,\pi)$,
$$ \frac{-i}{2 \pi} \Li(-e^y) =
\int_{v \in \R + i 0^+}
\dfrac{\exp\left (-i \frac{y v}{\pi}\right )}{4 v^2 \sinh(v)} \,  dv.
$$
 In the previous formula and in the remainder of the paper, $\R + i 0^+$ denotes a contour in $\C$ that is deformed from the horizontal line $\R \subset \C$ by avoiding $0$ via the upper half-plane (with a small half-circle for example).
\end{enumerate}
\end{proposition}

\subsection{The Bloch--Wigner function}

The \emph{Bloch--Wigner function} $D:\C \rightarrow \R $ defined by
\[
D(z) := \Im(\Li(z)) + \arg(1-z)\log \vert z \vert \quad \text{ if $z\in \C \smallsetminus \R$, and $0$ otherwise}
\]
is continuous on $\C$, and real-analytic on $\mathbb{C} \backslash \{0,1\}$ (see \cite[Section 3]{Za} for details). 
The Bloch-Wigner function plays a central role in hyperbolic geometry.  
The following result will be important for us (for a proof, see \cite{NZ}).

\begin{proposition}
Let $T$ be an ideal tetrahedron in $\mathbb{H}^3$ with complex shape structure $z$. Then, its volume is given by
\[
\Vol(T)= D(z) = D \left( \frac{z-1}{z} \right) = D \left( \frac{1}{1-z} \right).
\]
\end{proposition}

\subsection{Faddeev's quantum dilogarithm}

Recall \cite{AK} that for $\hbar >0$ and $\B >0$ such that $$(\B+\B^{-1}) \sqrt{\hbar} = 1,$$
 \emph{Faddeev's quantum dilogarithm} $\Phi_\B$ is the holomorphic function on $\R + i \left (\frac{-1}{2 \sqrt{\hbar}}, \frac{1}{2 \sqrt{\hbar}}\right )$ given by
$$
\Phi_\B(z) = \exp\left (
\frac{1}{4} \int_{w \in \R + i 0^+}
\dfrac{e^{-2 i z w} dw}{\sinh(\B w) \sinh({\B}^{-1}w) w}
\right ) \ \ \ \ \text{for} \ z \in \R + i \left (\frac{-1}{2 \sqrt{\hbar}}, \frac{1}{2 \sqrt{\hbar}}\right ),
$$
and extended to a meromorphic function for $z\in \C$ via the functional equation 
$$\Phi_\B\left (z-i \frac{\B^{\pm  1}}{2}\right )= \left (1+e^{2\pi \B^{\pm 1} z}\right )
\Phi_\B\left (z + i \frac{\B^{\pm 1}}{2}\right ).
$$
 Recall that $\R + i 0^+$ denotes a contour in $\C$ that is deformed from the horizontal line $\R \subset \C$ by avoiding $0$ by above.

Note that $\Phi_\B$ depends only on $\hbar = \frac{1}{(\B+\B^{-1})^2}$. Furthermore, as a consquence of the functional equation, the poles of $\Phi_\B$  lie on $ i \left [\frac{1}{2 \sqrt{\hbar}}, \infty\right ) $ and the zeroes lie symmetrically on $i \left (-\infty, \frac{-1}{2 \sqrt{\hbar}}\right ]$.  We stress the fact that in this paper we always assume that $\B$ is a real positive number, which simplifies several formulas in \cite[Appendix A]{AK}; notably the poles and zeroes live in the imaginary line instead of in sectors.

We now list several useful properties of Faddeev's quantum dilogarithm. We refer to \cite[Appendix A]{AK} for these properties (and several more), and to \cite[Lemma 3]{AH} for an alternate proof of the semi-classical limit property.

\begin{proposition}[Some properties of $\Phi_\B$]\label{prop:quant:dilog}
\

\begin{enumerate}
\item (inversion relation) For any $\B \in \R_{>0}$ and any  $z \in \R + i \left (\frac{-1}{2 \sqrt{\hbar}}, \frac{1}{2 \sqrt{\hbar}}\right )$, 
$$\Phi_\B(z) \Phi_\B(-z) = e^{i\frac{\pi}{12}(\B^2 + \B^{-2})} e^{i \pi z^2}.$$
\item (unitarity) For any $\B \in \R_{>0}$ and any  $z \in \R + i \left (\frac{-1}{2 \sqrt{\hbar}}, \frac{1}{2 \sqrt{\hbar}}\right )$, 
$$\overline{\Phi_\B(z)} = \frac{1}{\Phi_\B(\overline{z})}.$$
\item (semi-classical limit in $\B$) For any $z \in \R + i \left (-\pi,\pi \right )$,
$$\Phi_\B\left (\frac{z}{2 \pi \B}\right ) = \exp\left (\frac{-i}{2 \pi \B^2} \Li (- e^z)\right ) \left ( 1 + O_{\B \to 0^+}(\B^2)\right ).$$
\item (behavior at infinity) For any  $\B \in \R_{>0}$,
\begin{align*}
 \Phi_\B(z) \ \ \underset{\Re(z)\to -\infty}{\sim} & \ \ 1, \\
 \Phi_\B(z) \ \  \underset{\Re(z)\to \infty}{\sim} & \ \ e^{i\frac{\pi}{12}(\B^2 + \B^{-2})} e^{i \pi z^2}.
\end{align*}
In particular, for any  $\B \in \R_{>0}$ and any $d \in \left (\frac{-1}{2 \sqrt{\hbar}}, \frac{1}{2 \sqrt{\hbar}}\right )$,
\begin{align*}
|\Phi_\B(x+id) | \ \ \underset{\R \ni x \to -\infty}{\sim} & \ \ 1, \\
|\Phi_\B(x+id) | \ \  \underset{\R \ni x \to +\infty}{\sim} & \ \ e^{-2 \pi x d}.
\end{align*}
\end{enumerate}
\end{proposition}

Let us expand on Proposition \ref{prop:quant:dilog} (3), via the following \textit{uniform} estimate in the semi-classical limit.

\begin{proposition}\label{prop:quant:dilog:uniform}
There exists constants $C, C'>0$ such that for all $\delta \in (0,\pi)$, for all $\B \in(0,\sqrt{\delta})$, for all 
$y \in \R+i[-(\pi-\delta),\pi-\delta]$, we have:
    \begin{enumerate}
        \item (uniform semi-classical limit with $\B$)
        $$
        e^{-
        \left (\frac{C}{\delta} + C' \right) \B^2}
        \leqslant
        \left | 
        \Phi_\B\left (\frac{y}{2 \pi \B}\right )  \exp\left (\frac{i}{2 \pi \B^2} \Li (- e^y)\right )
        \right | \leqslant 
        e^{
        \left (\frac{C}{\delta} + C' \right) \B^2},
        $$
        \item similar inequalities hold if one replaces $y$ with $y(1+\B^2)$:
        $$
        e^{-
        \left (\frac{C}{\delta} + C' \right) \B^2}
        \leqslant
        \left | 
        \Phi_\B\left (\frac{y(1+\B^2)}{2 \pi \B}\right )  \exp\left (\frac{i}{2 \pi \B^2} \Li (- e^{y(1+\B^2)})\right )
        \right | \leqslant 
        e^{
        \left (\frac{C}{\delta} + C' \right) \B^2},
        $$
        \item (comparison of classical dilogarithms with $\B$ and $\hbar$)
        $$
        e^{-C_\delta}
        \leqslant
        \left | 
        \exp\left (\frac{-i}{2 \pi \B^2}(1+\B^2)^2 \Li (- e^{y})\right )
        \exp\left (\frac{i}{2 \pi \B^2} \Li (- e^{y(1+\B^2)})\right )
        \right | \leqslant 
        e^{C_\delta},
        $$
        where $C_\delta>0$ is a constant independant of $\B$ (see \cite[Lemma 7.17]{BAGPN}).
        \item (uniform semi-classical limit with $\hbar$)
        $$
        e^{-
        \left (\frac{C}{\delta} + C' \right) \B^2-C_\delta}
        \leqslant
        \left | 
        \Phi_\B\left (\frac{y}{2 \pi \sqrt{\hbar}}\right )  \exp\left (\frac{i}{2 \pi \hbar} \Li (- e^{y})\right )
        \right | \leqslant 
        e^{
        \left (\frac{C}{\delta} + C' \right) \B^2+C_\delta},
        $$
    \end{enumerate}
Moreover, the constants $C, C'$ can be computed explicitly (see \cite[Section 7]{BAGPN}).
\end{proposition}

\begin{proof}
    The inequality (1) follows directly from \cite[Section 7]{BAGPN}. Indeed, Lemmas 7.13 and 7.14 in \cite{BAGPN} deal with $\textstyle \Phi_{\B}(\frac{y}{2\pi \B})$, with the assumptions $\Im(y)\in [\delta,\pi-\delta]$ and $\Im(y)\in [-\pi+\delta,-\delta]$  respectively. It is easy to see that these assumptions, keeping the same conclusions, can be extended to $\Im(y)\in [0,\pi-\delta]$ and $\Im(y)\in [-\pi+\delta,0]$.

    As for (2), let us prove it follows from (1). Inequality (1) holds in particular for all 
    $y \in \R+i[-\frac{(\pi-\delta)}{1+\B^2},\frac{(\pi-\delta)}{1+\B^2}]$, and for those $y$ we obtain:
    $$
    e^{-\left (\frac{C}{\pi-\frac{(\pi-\delta)}{1+\B^2}} + C' \right) \B^2} \leqslant 
        \left | 
        \Phi_\B\left (\frac{y(1+\B^2)}{2 \pi \B}\right )  \exp\left (\frac{i}{2 \pi \B^2} \Li (- e^{y(1+\B^2)})\right )
        \right | \leqslant e^{\left (\frac{C}{\pi-\frac{(\pi-\delta)}{1+\B^2}} + C' \right) \B^2}.
        $$
    We conclude by remarking that $\pi-\frac{(\pi-\delta)}{1+\B^2}>\pi-(\pi-\delta)=\delta$, and thus 
    (2) follows.

    Inequalities of (3) follow almost directly from \cite[Lemma 7.17]{BAGPN}, except that the proof of \cite[Lemma 7.17]{BAGPN} can be immediately generalized to $\Im(y)\in [0,\pi-\delta]$ and $\Im(y)\in [-\pi+\delta,0]$ (like in (1)). Moreover, the upper bound $c_\delta$ on $\B$ in the proof of \cite[Lemma 7.17]{BAGPN} is smaller than $\sqrt{\delta}$, so we use this bound instead to avoid introducing yet another notation.

    Finally, (4) is a direct application of (2) and (3), using that $\hbar^{-1/2}=\B+\B^{-1}$.
\end{proof}

To study the asymptotics of partition functions in $\hbar$, we will also need the following result about the semi-classical limit of the quantum dilogarithm function in $\hbar$, going one step further into the Taylor expansion than Proposition \ref{prop:quant:dilog:uniform} but losing an uniform bound on the error term over non-compact contours. The result has also been used in \cite[Lemma 3]{AK} by Andersen and Kashaev when they prove their conjecture for $4_1$ and $5_2$. 
\begin{proposition}\label{semihbar} (semi-classical limit in $\hbar$) For any $z \in \R + i \left (-\pi,\pi \right )$,
$$\Phi_\B\left (\frac{z}{2 \pi \sqrt{\hbar}}\right ) = \exp\left (\frac{i}{2\pi}z\log(1+e^z) + \frac{i}{\pi}\Li \left(- e^{z}\right) - \frac{i}{2 \pi \hbar}  \Li \left(- e^{z}\right) \right ) \left ( 1 + O_{\hbar \to 0^+}(\hbar)\right ).$$
\end{proposition}
\begin{proof}
For any $z \in \R + i \left (-\pi,\pi \right )$, by Proposition \ref{prop:quant:dilog}, we have 
$$\Phi_\B\left (\frac{z}{2 \pi \B}\right ) = \exp\left (\frac{-i}{2 \pi \B^2} \Li (- e^z)\right ) \left ( 1 + O_{\B \to 0^+}(\B^2)\right ).$$
In particular, since $(\sqrt{\hbar})^{-1} = \B + \B^{-1}$, we have
\begin{align*}
   \Phi_\B\left (\frac{z}{2 \pi \sqrt{\hbar}}\right ) 
   = \Phi_\B\left (\frac{z(1+\B^2)}{2 \pi \B}\right ) 
   = \exp\left (\frac{-i}{2 \pi \B^2} \Li \left(- e^{(1+\B^2)z}\right) \right ) \left ( 1 + O_{\B \to 0^+}(\B^2)\right ).
\end{align*}
By Taylor's theorem, we have
$$
\Li \left(- e^{(1+\B^2)z}\right) 
= \Li \left(- e^{z}\right) 
- \B^2 z\log(1+e^{z}) + O_{\B \to 0^+}(\B^4).
$$
Thus,
\begin{align*}
   \Phi_\B\left (\frac{z}{2 \pi \sqrt{\hbar}}\right ) 
   = &\ \exp\left (\frac{i}{2\pi}z\log(1+e^z)  - \frac{i}{2 \pi \B^2}  \Li \left(- e^{z}\right) \right ) \left ( 1 + O_{\B \to 0^+}(\B^2)\right ).
\end{align*}
Next, since
$$
\frac{1}{\hbar} = \B^2 + 2 + \frac{1}{\B^2}, 
$$
we can write
\begin{align*}
\Phi_\B\left (\frac{z}{2 \pi \sqrt{\hbar}}\right ) 
   = &\ \exp\left (\frac{i}{2\pi}z\log(1+e^z) + \frac{i}{\pi}\Li \left(- e^{z}\right) - \frac{i}{2 \pi \hbar}  \Li \left(- e^{z}\right) \right ) \left ( 1 + O_{\B \to 0^+}(\B^2)\right ).
\end{align*}
Finally, since 
$$
\lim_{\B\to 0} \frac{\B^2}{\hbar}
= \lim_{\B\to 0}  (\B^4 + 2\B^2 + 1 ) 
= 1,
$$
we can replace $O_{\B \to 0^+}(\B^2)$ by $O_{\hbar \to 0^+}(\hbar)$. This completes the proof.
\end{proof}

\subsection{The Teichm\"uller TQFT of Andersen--Kashaev}\label{sub:AK:TQFT}

In this section we follow \cite{AK, KaWB, Kan}. Let $\mathcal{S}(\R^d)$ denote the Schwartz space of smooth 
functions from $\R^d$ to $\C$ that are rapidly decaying (in the sense that any derivative decays faster than any negative power of the norm of the input). 
Its continuous dual $\mathcal{S}'(\R^d)$ is the space of tempered distributions.

Recall that the \emph{Dirac delta function} is the tempered distribution $\mathcal{S}(\R) \to \C$ denoted by $\delta(x)$ or $\delta$ and defined by
$
\delta(x) \cdot f:= \int_{x \in \R} \delta(x) f(x) dx =
f(0)
$ for all $f \in \mathcal{S}(\R)$ (where $x \in \R$ denotes the argument of $f\in \mathcal{S}(\R)$).
Furthermore, we have the equality of tempered distributions
\[
\delta(x)=\int_{w \in \R} e^{-2 i \pi  x w} \,dw,
\] 
in the sense that for all $f \in \mathcal{S}(\R)$, 
$$
\left (\int_{w \in \R} e^{-2 i \pi x w} \,dw\right ) (f) =
\int_{x \in \R} \int_{w \in \R} e^{-2 i \pi x w} f(x)  \,dw \, dx \ = f(0) = \delta(x) \cdot f.
$$
The second equality follows from applying the Fourier transform $\mathcal{F}$ twice and using the fact that $\mathcal{F}(\mathcal{F}(f))(x) = f(-x)$ for $f\in \mathcal{S}(\R), x \in \R$. Recall also that the definition of the Dirac delta function and the previous argument have multi-dimensional analogues (see for example \cite{Kan} for details).

Given a triangulation $X$, 
writing $X^k$ for its collection of $k$-cells ($k\in \{0,1,2,3\}$), we assign to the tetrahedra $T_1, \ldots,T_N \in X^3$ formal real variables $t_1, \ldots, t_N$. 
We name $\mathsf{t}\colon T_j \mapsto t_{j}$ the corresponding bijection, and $\mathbf{t} = (t_{1},\ldots,t_{N})$ the corresponding
formal vector in $\R^{X^{3}}$.

Recall the notation $x_i(T) \in X^2$ for the $i$-th face ($i\in \{0,1,2,3\}$) of the tetrahedron $T\in X^3$.

We now define the kinematical kernel of $X$, which is a tempered distribution. In many cases of interest, a distribution-free formula holds (Lemma~\ref{lem:dirac} below)  and might be used as an alternate definition. 
Note that our notation is related to the notation in \cite{BAGPN} by the change of variables $(t_1,\dots, t_N) \to (\epsilon(T_1)t_1,\dots, \epsilon(T_N)t_N)$, which only induces a sign change and does not change the modulus of the partition function.
\begin{definition}
Let $X$ be a triangulation such that $H_2(M_X\smallsetminus X^0,\Z)=0$. The \textit{kinematical kernel of $X$} is a tempered distribution $\mathcal{K}_X \in \mathcal{S}'\left (\R^{X^{3}}\right )$ defined by the integral
\begin{align*}
    &\ \mathcal{K}_X(\mathbf{t}) \\
    =&\  \int_{\mathbf{x} \in \R^{X^{2}}} d\mathbf{x} \prod_{T \in X^3} e^{ 2 i \pi  x_0(T) \mathsf{t}(T)}
\delta\left ( x_0(T)- x_1(T)+ x_2(T)\right )
\delta\left ( x_2(T)- x_3(T)+ \varepsilon(T)\mathsf{t}(T)\right )
\end{align*}
where, with a slight abuse of notation, $x_i(T)$ 
refers to the $x_i(T)$-th component of $\mathbf{x}\in \R^{X^2}$.
(This convention, of denoting by $x_i(T)$ both a $2$-cell and the formal variable associated to it, is taken from~\cite{AK}: it will help keep our formulas short.)
\end{definition}

Essentially, if $\pi : \R^{X^2 \cup X^3} \rightarrow \R^{X^3}$ denotes the canonical projection, then
 $\mathcal{K}_X(\mathbf{t})$ associates to a Schwartz function $f:\R^{X^3} \rightarrow \R$ the (normalized) integral,
 over the affine subspace of $\R^{X^2 \cup X^3}$ where the arguments of the $\delta$'s vanish,
  of the product $(f \circ \pi) \cdot g$, where $g:\R^{X^2 \cup X^3} \rightarrow \R$ is the exponential of a certain quadratic form.
  
More formally, one
should understand the integral of the previous formula as the following equality of tempered distributions, similarly as above ( ${\!\top}$ denoting the transpose):
$$
\mathcal{K}_X(\mathbf{t}) =
\int_{\mathbf{x} \in \R^{X^{2}}} d\mathbf{x} 
\int_{\mathbf{w} \in \R^{2 N}} d\mathbf{w} \
e^{ 2 i \pi \mathbf{t}^{\!\top} \mathcal{R}} \mathbf{x}
e^{ -2 i \pi \mathbf{w}^{\!\top} \mathcal{A}} \mathbf{x}
e^{ -2 i \pi \mathbf{w}^{\!\top} \mathcal{B}} \mathbf{t} \
\in \mathcal{S}'\left (\R^{X^{3}}\right ),
$$
where
$\mathbf{w}=(w_1,\ldots,w_N,w'_1, \ldots,w'_N)$ is a vector of $2N$ new real variables, such that $w_j,w'_j$ are associated to 
$\delta\left ( x_0(T_j)- x_1(T_j)+ x_2(T_j)\right )$ and
$\delta\left ( x_2(T_j)- x_3(T_j)+ \mathsf{t}(T_j)\right )$, and where
 $\mathcal{R,A,B}$ are matrices with integer coefficients depending on the values $x_k(T_j)$, i.e.\ on the combinatorics of the face gluings. More precisely, the rows (resp.\ columns) of $\mathcal{R}$ are indexed by the vector of tetrahedron variables $\mathbf{t}$ (resp.\ of face variables $\mathbf{x}$) and $\mathcal{R}$ has a coefficient $1$ at coordinate $(t_j,x_0(T_j))$ and zero everywhere else; $\mathcal{B}$ is indexed by  $\mathbf{w}$ (rows) and  $\mathbf{t}$ (columns) and has a $1$ at the coordinate $(w'_j,t_j)$; finally, $\mathcal{A}$ is such that $\mathcal{A} \mathbf{x} + \mathcal{B}  \mathbf{t}$ is a column vector indexed by  $\mathbf{w}$ containing the values 
 $\left (x_0(T_j)- x_1(T_j)+ x_2(T_j)\right )_{1\leq j \leq N}$ followed by $\left (x_2(T_j)- x_3(T_j)+ \varepsilon(T_j) t_j \right)_{1\leq j \leq N}$.

\begin{remark}
    The matrices $\mathcal{R,A,B}$
    can also be defined directly from the boundary maps $x_0, x_1, x_2, x_3$, as in Section \ref{sub:intro:FAMED}: we have $\mathcal{R} = \mathcal{X}_0$, \ 
    $\mathcal{A} = \begin{pmatrix}
    \mathcal{X}_0-\mathcal{X}_1+\mathcal{X}_2\\\mathcal{X}_2-\mathcal{X}_3
\end{pmatrix}$ and
     $\mathcal{B}:=\begin{pmatrix}
    0_N \\\mathcal{E}
\end{pmatrix}$. See also Figure \ref{fig:41:face:matrices} for the values of these matrices for the figure-eight knot complement.
\end{remark}

\begin{lemma}\label{lem:dirac}\cite{BAGPN}
If the $2N\times 2N$ matrix $\mathcal{A}$ in the previous formula is invertible, then the kinematical kernel is simply a bounded function given by:
$$ \mathcal{K}_X(\mathbf{t}) = \frac{1}{| \det(\mathcal{A}) |} e^{ 2 i \pi \mathbf{t}^{\!\top} Q \mathbf{t}},
$$
where $Q
=\frac{-1}{2}\big(  \mathcal{X}_0 \mathcal{A}^{-1} \mathcal{B} + ( \mathcal{X}_0 \mathcal{A}^{-1} \mathcal{B})^{\!\top} \big)
$.
\end{lemma}
\begin{proof}
This follows from \cite[Lemma 2.9]{BAGPN} and the change of variables 
$$(t_1,\dots, t_N) \to (\epsilon(T_1)t_1,\dots, \epsilon(T_N)t_N).$$
\end{proof}

The product of several Dirac delta functions might not be a tempered distribution in general. However the kinematical kernels in this paper will always be, thanks to the assumption that \linebreak $H_2(M_X\setminus X^0,\Z)=0$ (satisfied by  {any} knot complement). See \cite{AK} for more details, via the theory of wave fronts. The key property to notice is the linear independence of the terms $x_0(T_j)- x_1(T_j)+ x_2(T_j), \ x_2(T_j)- x_3(T_j)+ \varepsilon(T_j)t_j$.

\begin{definition}
Let $X$ be a triangulation. Its \textit{dynamical content} associated to $\hbar>0$ is a function $\mathcal{D}_{\hbar,X}\colon \mathcal{A}_X \to  \mathcal{S}\left (\R^{X^{3}}\right )$ defined on each set of angles $\alpha \in \mathcal{A}_X$ by
$$\mathcal{D}_{\hbar,X}(\mathbf{t},\alpha)= \prod_{T\in X^{3}} 
\dfrac{\exp \left( \hbar^{-1/2} \alpha_3(T) \varepsilon(T)\mathsf{t}(T) \right )}
{\Phi_\B\left ( \varepsilon(T) \left(\mathsf{t}(T) - \dfrac{i}{2 \pi \sqrt{\hbar}}(\pi-\alpha_1(T))\right)\right )^{\varepsilon(T)}}.
$$
\end{definition}

Note that $\mathcal{D}_{\hbar,X}(\cdot,\alpha)$ is in $\mathcal{S}\left (\R^{X^{3}}\right )$ thanks to the properties of $\Phi_\B$ and the positivity of the dihedral angles in $\alpha$ (see \cite{AK} for details).

More precisely, each term in the dynamical content has exponential decrease as described in the following lemma, which immediately follows from Proposition \ref{prop:quant:dilog} (4).

\begin{lemma}\label{lem:dec:exp}\cite[Lemma 2.11]{BAGPN}
	Let $\B \in \R_{>0}$ and $a,b,c \in (0,\pi)$ such that $a+b+c=\pi$. Then
	$$
	\left |
	\dfrac{e^{\frac{1}{ \sqrt{\hbar}} c x}}{\Phi_\B\left (x-\frac{i}{ 2 \pi \sqrt{\hbar}}(b+c)\right )}
	\right | \underset{\R \ni x \to \pm \infty}{\sim} \left |
	e^{\frac{1}{ \sqrt{\hbar}} c x} \Phi_\B\left (x+\frac{i}{ 2 \pi \sqrt{\hbar}}(b+c)\right )
	\right | \ \ \left \{
	\begin{matrix}
	\underset{\R \ni x \to -\infty}{\sim} e^{\frac{1}{ \sqrt{\hbar}} c x}. \\
	\ \\
	\underset{\R \ni x \to +\infty}{\sim} e^{-\frac{1}{ \sqrt{\hbar}} b x}.
	\end{matrix} \right .
	$$
\end{lemma}

Lemma \ref{lem:dec:exp} illustrates why we need the three angles $a,b,c$ to be in $(0,\pi)$: $b$ and $c$ must be positive in order to have exponential decrease in both directions, and $a$ must be  {positive} as well so that $b+c < \pi$ and $\Phi_\B\left (x \pm \frac{i}{ 2 \pi \sqrt{\hbar}}(b+c)\right )$ is always defined.
Alternatively, using the inversion formula of the quantum dilogarithm function (Lemma \ref{prop:quant:dilog} (1)), the dynamical content can be written in the following form.
\begin{lemma}\label{altDC}
Let $X$ be a triangulation. Up to a multiplicative constant with norm 1, we have
\begin{align*}
\mathcal{D}_{\hbar,X}(\mathbf{t},\alpha)
=&\ \ \prod_{T\in X^{3}} 
\exp \left( \hbar^{-1/2} \alpha_3(T) \epsilon(T) \mathsf{t}(T) - i\pi \left(\frac{\epsilon(T)+1}{2}\right)\left(\mathsf{t}(T) - \dfrac{i}{2 \pi \sqrt{\hbar}} (\pi-\alpha_1(T))\right)^2
\right ) \\
& \ \  \times \prod_{T\in X^{3}}
\Phi_\B\left ( -\mathsf{t}(T) + \dfrac{i}{2 \pi \sqrt{\hbar}} (\pi-\alpha_1(T))\right ).
\end{align*}

\end{lemma}
\begin{proof}
The result follows immediately by using Lemma \ref{prop:quant:dilog} (1) and the fact that $|e^{i\frac{\pi}{12}(\B^2 + \B^{-2})}|=1$ for all $\B\in \RR_{>0}$. 
\end{proof}

Now, for $X$ a triangulation such that $H_2(M_X\setminus X_0,\Z)=0$, $\hbar>0$ and $\alpha \in \mathcal{A}_X$ an angle structure,  the associated \textit{partition function of the Teichm\"uller TQFT} is the complex number:
$$\mathcal{Z}_{\hbar}(X,\alpha)= \int_{\mathbf{t} \in \R^{X^3}}  \mathcal{K}_X(\mathbf{t}) \mathcal{D}_{\hbar,X}(\mathbf{t},\alpha) d\mathbf{t} \ \ \ \in \C. $$

Andersen and Kashaev proved in \cite{AK} that the  {modulus} $\left |\mathcal{Z}_{\hbar}(X,\alpha) \right | \in \R_{>0}$ is invariant under Pachner moves with positive angles, and then studied a generalization of this property to a larger class of moves and triangulations with angles, using analytic continuation in complex-valued $\alpha$ \cite{AKicm}.

One natural question is :
for an hyperbolic structure on $M$ and any two corresponding angle structures $(X,\alpha), (X',\alpha')$, do we have 
$\left |\mathcal{Z}_{\hbar}(X,\alpha) \right |=\left |\mathcal{Z}_{\hbar}(X',\alpha') \right |$?
This question is still open at the time of writing.
The previously mentioned analytic continuation in complex-valued $\alpha$ may be a useful tool for proving a positive answer to this question.

Let $\alpha =(a_1,b_1,c_1,\dots, a_N, b_N, c_N) \in \mathcal{A}_X$. Define
$C(\alpha)=
(\varepsilon(T_1)c_1,\dots, \varepsilon(T_N)c_N)^{\!\top}$ and $\mathscr{W}(\alpha) = 2Q (\boldsymbol{\pi-a}) - C(\alpha),$ where the first term in $\mathscr{W}(\alpha)$ is the product of the matrix $Q \in M_{N}(\QQ)$ and the vector $\boldsymbol{\pi - a} = (\pi-a_1,\dots, \pi-a_N)^{\!\top}$.
Following \cite{BAGPN}, under the assumption that $\det \mathcal{A} \neq 0$, we have the following expression of the partition function. 
\begin{proposition}\label{Tpartiexpress1}
Let 
$$
\tilde{\mathscr{Y}}_{\hbar, \alpha} = \prod_{k=1}^N \left(\RR + i(\pi-a_k)\right) .$$
Assume that $\det \mathcal{A} \neq 0$. We have
\begin{align*}
&\ |\mathcal{Z}_{\hbar}(X,\alpha) | = \frac{1}{|\det \mathcal{A}|}
\Big(\frac{1}{2\pi \sqrt{\hbar}}\Big)^{N}\\
&\  \times \left| 
\int_{\mathscr{Y}_\alpha} 
e^{\frac{1}{2\pi\hbar}\left(i \mathbf{y}^{\!\top} Q \mathbf{y} + \mathbf{y}^{\!\top} \mathscr{W}(\alpha) 
- i\sum_{k=1}^N \left(\frac{\varepsilon(T_k)+1}{4}\right) y_k^2\right)}
\prod_{k=1}^N \Phi_\B\left(\frac{y_k}{2\pi \sqrt{\hbar}}\right) d\mathbf{y} \right|.
\end{align*}
\end{proposition}
\begin{proof}
Denote
$$
\tilde{\mathscr{Y}}_{\hbar, \alpha} = \prod_{k=1}^N \left(\RR + \frac{i}{2\pi\sqrt{\hbar}}(\pi-a_k)\right) .$$
By applying the change of variables $\tilde{\mathbf{y}} = -\mathbf{t} + \frac{i}{2\pi\sqrt{\hbar}}(\boldsymbol{\pi-a})$, we have
\begin{align*}
&\ \left|\mathcal{Z}_{\hbar}(X,\alpha)\right|\\
=&\ \left|\int_{\mathbf{t} \in \R^{X^3}}  \mathcal{K}_X(\mathbf{t}) \mathcal{D}_{\hbar,X}(\mathbf{t},\alpha) d\mathbf{t} \right| \\
=&\ \left|\int_{\mathbf{y} \in \tilde{\mathscr{Y}}_{\hbar, \alpha} }  \mathcal{K}_X\left(-\tilde{\mathbf{y}} + \frac{i}{2\pi\sqrt{\hbar}}(\boldsymbol{\pi-a})\right) \mathcal{D}_{\hbar,X}\left(- \tilde{\mathbf{y}} + \frac{i}{2\pi\sqrt{\hbar}}\boldsymbol{\pi - a}\right) d\mathbf{t} \right| \\
=&\  \left|\frac{1}{\det \mathcal{A}}\int_{\mathbf{y} \in \tilde{\mathscr{Y}}_{\hbar, \alpha} } e^{2i\pi \tilde{\mathbf{y}}^{\!\top} Q \tilde{\mathbf{y}} + \frac{2}{\sqrt{\hbar}} (\boldsymbol{\pi-a}) Q  \tilde{\mathbf{y}} - \frac{i}{2\pi\sqrt{\hbar}} (\boldsymbol{\pi-a}) Q (\boldsymbol{\pi-a}) - \frac{1}{\sqrt{\hbar}} C(\alpha)^{\!\top} \tilde{\mathbf{y}} + \frac{i}{2\pi\sqrt{\hbar}} C(\alpha)^{\!\top} (\boldsymbol{\pi-a})} \right.\\
& \left. \qquad \qquad \qquad \qquad e^{-\sum_{k=1}^N i \pi\left(\frac{\varepsilon(T_k)+1}{2}\right) \tilde y_j^2} \prod_{k=1}^N \Phi_\B(\tilde{y}_k) d\mathbf{\tilde y} \right| \\
=&\  \left|\frac{1}{\det \mathcal{A}}\int_{\mathbf{y} \in \tilde{\mathscr{Y}}_{\hbar, \alpha} } e^{2i\pi \tilde{\mathbf{y}}^{\!\top} Q \tilde{\mathbf{y}} + \frac{1}{\sqrt{\hbar}} \mathscr{W}(\alpha)^{\!\top} \tilde{\mathbf{y}} - \sum_{k=1}^N i \pi\left(\frac{\varepsilon(T_k)+1}{2}\right) \tilde y_j^2}
 \prod_{k=1}^N \Phi_\B(\tilde{y}_k) d\mathbf{\tilde y} \right|,
\end{align*}
where $\mathscr{W}(\alpha)= 2Q^{\!\top}(\boldsymbol{\pi - a}) - C(\alpha)= 2Q(\boldsymbol{\pi - a}) - C(\alpha)$.
Applying the change of variables $\mathbf{y} = 2\pi \sqrt{\hbar}\tilde{\mathbf{y}}$, we have the first result.
\end{proof}

\subsection{$\mathbf{PSL(2;\C)}$ A-polynomial}\label{PSL2CApoly}
In this section, we follow the setup in \cite{P20} to define the $PSL(2;\C)$ A-polynomial. Given a hyperbolic knot $K\subset \SS^3$, let $m$ and $l$ be a meridian and corresponding preferred longitude respectively. Let $R(\pi_1(\SS^3\setminus K)))$ be the $PSL(2;\C)$ representation variety of the fundamental group of $\SS^3\setminus K$. Let $R^U(\pi_1(\SS^3\setminus K)) \subset R(\pi_1(\SS^3\setminus K))$ be the affine algebraic set consisting of $\rho$ with the property that $\rho(m)$ and $\rho(l)$ are upper triangular, i.e. $\rho(m)$ and $\rho(l)$ are of the form
$$
\rho(m) = \pm
\begin{pmatrix}
    e^{\frac{w_m^{loc}}{2}} & * \\
    0 & e^{-\frac{w_m^{loc}}{2}}
\end{pmatrix}, \quad
\rho(l) = \pm
\begin{pmatrix}
    e^{\frac{w_l^{loc}}{2}} & * \\
    0 & e^{-\frac{w_l^{loc}}{2}}
\end{pmatrix}
$$
for some complex numbers $w_m^{loc}$ and $ w_l^{loc}$.
Define a map $\zeta: R^U(\pi_1(\SS^3\setminus K)) \to \C^2$ by 
$$ \zeta(\rho) = \left( \left(\pm e^{\frac{w_m^{loc}}{2}}\right)^2, \left(\pm e^{\frac{w_l^{loc}}{2}}\right)^2 \right) = \left( e^{w_m^{loc}}, e^{w_l^{loc}}\right).$$
We let $(M,L)$ be the coordinates of $\C^2$. For each irreducible component $C$ of the algebraic set $R^U(\pi_1(\SS^3\setminus K))$, if the Zariski closure of $\zeta(C) \subset \C^2$ is one-dimensional, we let $F_C(M,L)$ be the defining polynomial. Otherwise, we let $F_C(M,L) =1$. 
\begin{definition}
    The $PSL(2;\C)$ A-polynomial of $K$ is defined by
    $$ A_K(M,L) = \prod_{C} F_C(M,L),$$
    where the product is over all irreducible components $C$ of $R^U(\pi_1(\SS^3\setminus K))$.
\end{definition}
Let $\rho_0: \pi_1(\SS^3\setminus K) \to PSL(2;C)$ be the discrete faithful representation associated to the complete hyperbolic structure of $\SS^3 \setminus K$. Up to conjugation we assume that $\rho_0 \in R^U(\pi_1(\SS^3\setminus K))$. Since $\rho_0(m)$ and $\rho_0(l)$ are parabolic elements, we have
$$
\rho_0(m) = 
\pm \begin{pmatrix}
    e^{\frac{0}{2}} & * \\
    0 & e^{-\frac{0}{2}}
\end{pmatrix}
= \pm
\begin{pmatrix}
    1 & * \\
    0 & 1
\end{pmatrix}, \quad
\rho_0(l) = 
\pm \begin{pmatrix}
    e^{\frac{0}{2}} & * \\
    0 & e^{-\frac{0}{2}}
\end{pmatrix}
= \pm
\begin{pmatrix}
    1 & * \\
    0 & 1
\end{pmatrix}.
$$
Let $C_0$ be the irreducible component containing $\rho_0$.
\begin{definition}
The geometric component of the $PSL(2;\C)$ A-polynomial of $K$ is defined by
$$ A_{K}^0(M,L) = F_{C_0}(M,L).$$
\end{definition}
Note that generically, $\rho(m)$ and $\rho(l)$ are loxodromy elements of the form
$$
\rho(m) = \pm
\begin{pmatrix}
    e^{\frac{w_m^{loc}}{2}} & 0 \\
    0 & e^{-\frac{w_m^{loc}}{2}}
\end{pmatrix}, \quad
\rho(l) = \pm
\begin{pmatrix}
    e^{\frac{w_l^{loc}}{2}} & 0 \\
    0 & e^{-\frac{w_l^{loc}}{2}}
\end{pmatrix}.
$$
In particular, we have
$$
\begin{pmatrix}
    0 & 1 \\
    1 & 0
\end{pmatrix}
\rho(m)
\begin{pmatrix}
    0 & 1 \\
    1 & 0
\end{pmatrix}= \pm
\begin{pmatrix}
    e^{-\frac{w_m^{loc}}{2}} & 0 \\
    0 & e^{\frac{w_m^{loc}}{2}}
\end{pmatrix}, \quad
\begin{pmatrix}
    0 & 1 \\
    1 & 0
\end{pmatrix}
\rho(l)
\begin{pmatrix}
    0 & 1 \\
    1 & 0
\end{pmatrix}
= \pm
\begin{pmatrix}
    e^{-\frac{w_l^{loc}}{2}} & 0 \\
    0 & e^{\frac{w_l^{loc}}{2}}
\end{pmatrix}.
$$
As a result, these two conjugated representations correspond to two distinct points in $\CC^2$. 

We conclude the discussion by giving several remarks. First, for any $\rho$ sufficiently near $\rho_0$, we define the complex logarithmic holonomies of $m$ and $l$ to be $w_m^{loc}$ and $w_l^{loc}$ in such a way that $w_m^{loc} = w_l^{loc} = 0$ at $\rho_0$. Besides, it is known that for hyperbolic knots, $A_K^0(M,L) \neq 1$ and $(1,1)$ is a smooth point. Especially, near the complete hyperbolic structure, $w_m^{loc}$ implicitly determines $w_l^{loc}$ through the defining equation $A_K^0(M,L)=0$. Finally, the A-polynomial can be defined similarly with respect to any pair of simple closed curves on $\partial (\SS^3\setminus\nu(K))$ that generate $\pi_1(\partial (\SS^3\setminus\nu(K)))$.

\subsection{Neumann-Zagier potential function}\label{NZpotentintro}
Let 
$\mathcal{N}$ 
be a 3-manifold and let $K\subset \mathcal{N}$ be a hyperbolic knot in $\mathcal{N}$, i.e. the complement $M=\mathcal{N}\setminus K$ admits a complete hyperbolic structure with finite volume. Recall from \cite{NZ} that for any simple closed curve $\gamma \in \pi_1(\partial (\mathcal{N}\setminus \nu(K)))$, locally near the complete hyperbolic structure, the deformation space of hyperbolic structure of $M$ can be parametrized in a generically 2:1 way with one complex variable $w^{loc}_\gamma$ 
such that, for any complex number $\xi \in \C$ close enough to $0$, 
the point $w^{loc}_{\gamma}=\xi$ corresponds to the hyperbolic structure where the complex logarithmic holonomy of $\gamma$ is equal to $\xi$ (if $X$ is an ideal triangulation of $\mathcal{N} \setminus K$, this means the structure such that
$\mathrm{H}^\C_{X,\gamma}(\mathbf{z})=\xi$).

Furthermore, given a pair of simple closed curves $(\gamma_1, \gamma_2)$ that generates $\pi_1(\partial (\mathcal{N}\setminus \nu(K)))$, the transition map 
$\Psi^{loc}_{\gamma_1,\gamma_2}$
from $w^{loc}_{\gamma_1}$ to 
$w^{loc}_{\gamma_2}
=\Psi^{loc}_{\gamma_1,\gamma_2}
(w^{loc}_{\gamma_1})$
is a locally biholomorphic map around 0 that sends 0 to 0. 
For example, when $\mathbf{z}$ is a shape structure very close to the complete one, we have $\Psi^{loc}_{\gamma_1,\gamma_2}(\mathrm{H}^\C_{X,\gamma_1}(\mathbf{z}))=\mathrm{H}^\C_{X,\gamma_2}(\mathbf{z})$.

In particular, one can consider the holomorphic function $\phi_{\gamma_1,\gamma_2}$ defined locally on a simply connected neighborhood around 0 by
$$
\phi_{\gamma_1,\gamma_2}(w^{loc}_{\gamma_1})
=
i\Big(\Vol(M) + i\CS(M)\Big) + \frac{1}{2}
\int_0^{w^{loc}_{\gamma_1}} 
\Psi^{loc}_{\gamma_1,\gamma_2}
(t)
dt
,
$$
where $w^{loc}_{\gamma_2}=\Psi^{loc}_{\gamma_1,\gamma_2}
(t)$ is regarded as a function in a complex variable $t$, the integral is along any contour from $0$ to $w^{loc}_{\gamma_1}$, and $\Vol(M)$ and $\CS(M)$ are the hyperbolic volume and the Chern-Simons invariant of $M=N\setminus K$ respectively. 
Note that the holomorphic function $\phi_{\gamma_1,\gamma_2}$ satisfies the properties that
\begin{align}\label{NZprop}
\phi_{\gamma_1,\gamma_2}(0) =  i\Big(\Vol(M) + i\CS(M)\Big) \quad \text{and} \quad \frac{d \phi_{\gamma_1,\gamma_2}\big(w^{loc}_{\gamma_1}\big)}{d\big(w^{loc}_{\gamma_1}\big)} = \frac{
\Psi^{loc}_{\gamma_1,\gamma_2}
(w^{loc}_{\gamma_1})
}{2},
\end{align}
which uniquely characterize the function.

Moreover, by \cite[Equation 2.4]{WY2}, in the special case where $w^{loc}_{\gamma_1} = i\theta$ for some $\theta\in \RR$, we have
\begin{align}\label{NZandVol}
\mathrm{Im} \phi_{\gamma_1,\gamma_2} (i\theta) = \Vol \Big(M;\gamma_1,\theta\Big), 
\end{align}
where $\Vol(M;\gamma_1,\theta)$ is the hyperbolic volume of $M$ with (possibly incomplete) hyperbolic structure satisfying $\mathrm{H}^\C_{X,\gamma_1}(\mathbf{z}) = i\theta$.

\subsection{Neumann-Zagier datum and gluing equations}\label{NZD}
Let $L$ be a hyperbolic link in a closed oriented 3-manifold $\mathcal{N}$ and let $\partial (\mathcal{N}\setminus\nu(L)) = \mathbb{T}_1 \coprod \dots \coprod \mathbb{T}_k$. On each $\mathbb{T}_i$, we choose a simple closed curve $\gamma_i \in \pi_1(\mathbb{T}_i)$ and let 
$\boldsymbol{\gamma} = (\gamma_1,\dots,\gamma_k) \in \pi_1(\partial (\mathcal{N}\setminus \nu(L)))$ be the system of simple closed curves. Let $X = \{T_1,\dots,T_N\}$ be an ideal triangulation of $M$ and let $X^1 = \{E_1,\dots, E_N\}$ be the set of edges. For each $T_i$, we choose a quad type (i.e. a pair of opposite edges) and assign a shape parameter $z_i\in \CC\setminus \{0,1\}$ to the edges. For $z_i \in \CC\setminus \{0,1\}$, we define $z_i ' = \frac{1}{1-z_i}$, $z_i'' = 1-\frac{1}{z_i}$. Recall that for each ideal tetrahedron, opposite edges share the same shape parameters (See Figure \ref{edgehol}, left). By \cite{NZ}, there exists $N-k$ linearly independent edge equations, in the sense that if these $N-k$ edge equations are satisfied, the remaining $k$ edge equations will automatically be satisfied. Without loss of generality we assume that $\{E_1,\dots,E_{N-k}\}$ is a set of linearly independent edges. For each edge $E_i$, we let $E_{i,j}$ be the numbers of edges with shape parameter $z_j$ that is incident to $E_i$. We define $E_{i,j}'$ and $E_{i,j}''$ be respectively the corresponding counting with respect to $z_j'$ and $z_j''$ (see Figure \ref{edgehol}, middle). The gluing variety $\mathcal{V}_{X}$ is the affine variety in $(z_1,z_1',z_1'', \dots, z_N, z_N', z_N'') \in \CC^{3N}$ defined by the zero sets of the polynomials
\begin{align*}
p_i = z_i(1-z_i'')- 1,\quad p_i' = z_i'(1-z_i) -1,\quad p_i'' = z_i''(1-z_i') -1 
\end{align*}
for $i=1,\dots, N$ and the polynomials
\begin{align*}
\prod_{j=1}^N z_j^{E_{ij}} (z_j')^{E_{ij}'} (z_j'')^{E_{ij}''}  -1
\end{align*}
for $i=1,\dots, N-k$. By using the equations $p_i=p_i'=p_i'' =0$ for $i=1,\dots,N$, for simplicity we will use $(z_1,\dots, z_N) \in (\CC\setminus \{0,1\})^N$ to represent a point in $\mathcal{V}_{X}$. Besides, for each $\gamma_i \in \pi_1(\mathbb{T}_i)$, we let $C_{i,j}$ be the numbers of edges with shape $z_j$ on the left hand side of $\gamma_i$ minus the numbers of edges with shape $z_j$ on the right hand side of $\gamma_i$. We define $C_{i,j}'$ and $C_{i,j}''$ be respectively the corresponding counting with respect to $z_j'$ and $z_j''$ (see Figure \ref{edgehol}, right). Given $(z_1,\dots, z_n) \in \mathcal{V}_{X}$, the (logarithmic) holonomy of $\gamma_i$ is given by
$$
\mathrm{H}^\C_{X,\gamma_i}(\mathbf{z}) = \sum_{j=1}^N \left( C_{ij} \Log(z_i) + C_{ij}' \Log(z_i') + C_{ij}'' \Log (z_i'') \right)
$$
for $i=1,\dots, k$. There is a well-defined map
$$ \mathcal{P}_{X} : \mathcal{V}_{X} \to \mathrm{X}(M)$$
that sends $\mathbf{z}=(z_1,\dots, z_n)  \in \mathcal{V}_{X}$ to the character $[\rho_{\mathbf{z}}]$ of the pseudo-developing map $\rho_{\mathbf{z}}$ with 
$$
\rho_{\mathbf{z}}(\gamma_i) = \pm 
\begin{pmatrix}
e^{\frac{\mathrm{H}^\C_{X,\gamma_i}(\mathbf{z})}{2}} & * \\
0 & e^{-\frac{\mathrm{H}^\C_{X,\gamma_i}(\mathbf{z})}{2}} 
\end{pmatrix}$$
for $i=1,\dots,k$ up to conjugation.

\begin{figure}[h]
  \centering
          \includegraphics[scale=0.13]{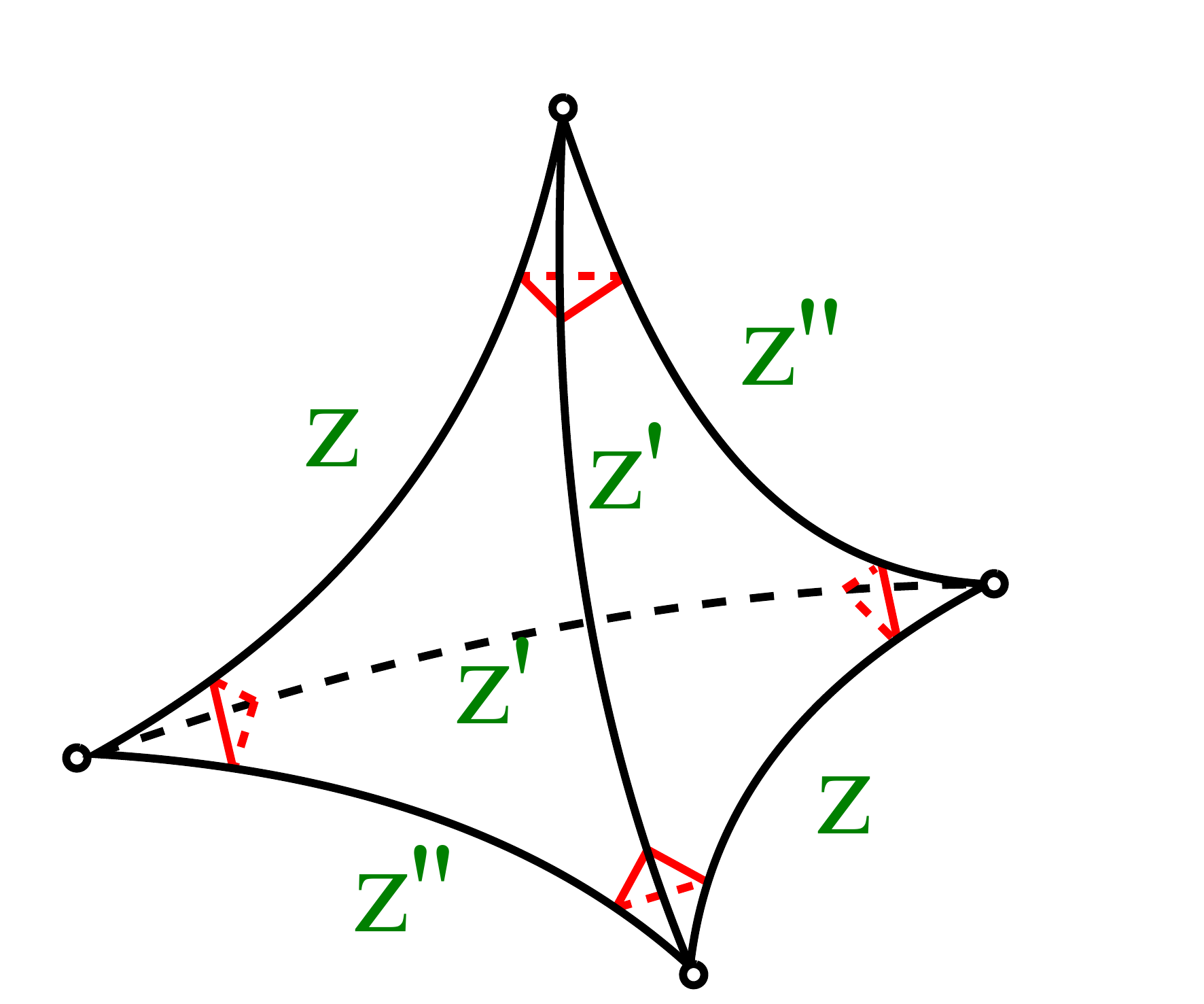} \qquad
\includegraphics[scale=0.13]{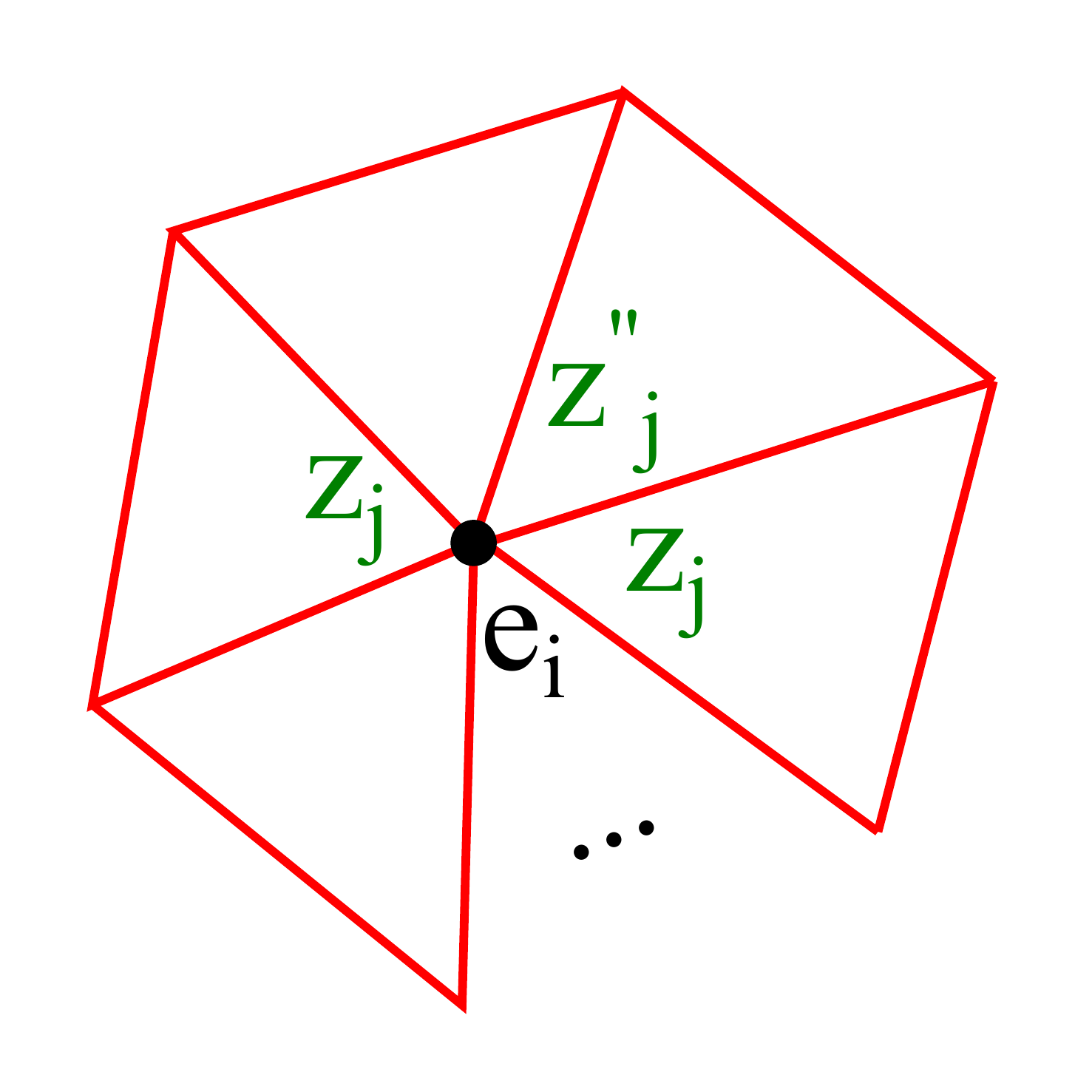}      \qquad\quad  \includegraphics[scale=0.13]{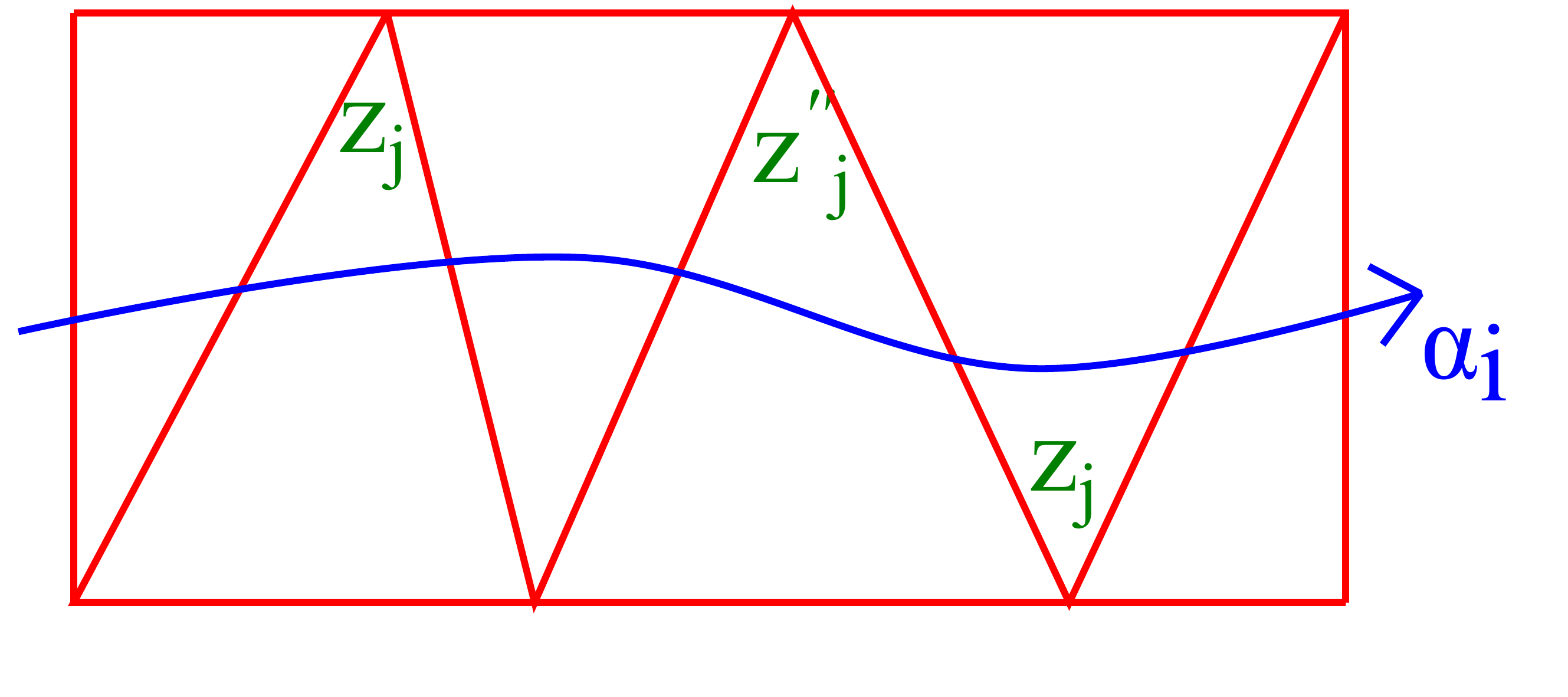}
\caption{On the left hand side, we have an ideal tetrahedron with shape parameters $z_i, z_i'$ and $z_i''$ assigned to the edges. In the middle, the black dot with a label $e_i$ corresponds to the $i$-th edge, which is surrounded by truncated triangles around the ideal vertices of the tetrahedra in the triangulation. In this example, $E_{ij} = 2$. On the right hand side, the rectangle is a fundamental domain of the boundary torus $\mathbb{T}_i$ and $\alpha_i$ is a lifting of the simple closed curve $\alpha_i \in \pi_1(\mathbb{T}_i)$ to the fundamental domain. In this example, we have $C_{ij} = 1 - 1 = 0$.}\label{edgehol}
\end{figure}
        
We define three $N \times N$ matrices $\mathbf{G},\mathbf{G'},\mathbf{G''} \in M_{N\times N}(\ZZ)$ by
\begin{align*}
\mathbf{G}=
\begin{pmatrix}
E_{1,1} & E_{1,2} & \dots & E_{1,N} \\
\vdots & &  & \vdots \\
E_{N-k,1} & E_{N-k,2} & \dots & E_{N-k,N} \\
C_{1,1} & C_{1,2} & \dots & C_{1,N} \\
\vdots & &  & \vdots   \\
C_{k,1} & C_{k,2} & \dots & C_{k,N}   
\end{pmatrix},\qquad
\mathbf{G'}=
\begin{pmatrix}
E_{1,1}' & E_{1,2}' & \dots & E_{1,N}' \\
\vdots & &  & \vdots \\
E_{N-k,1}' & E_{N-k,2}' & \dots & E_{N-k,N}' \\
C_{1,1}' & C_{1,2}' & \dots & C_{1,N}' \\
\vdots & &  & \vdots   \\
C_{k,1}' & C_{k,2}' & \dots & C_{k,N}'   
\end{pmatrix}
\end{align*}
and
\begin{align*}
\mathbf{G''}=
\begin{pmatrix}
E_{1,1}'' & E_{1,2}'' & \dots & E_{1,N}'' \\
\vdots & &  & \vdots \\
E_{N-k,1}'' & E_{N-k,2}'' & \dots & E_{N-k,N}'' \\
C_{1,1}'' & C_{1,2}'' & \dots & C_{1,N}'' \\
\vdots & &  & \vdots   \\
C_{k,1}'' & C_{k,2}'' & \dots & C_{k,N}''  
\end{pmatrix}.
\end{align*}
Given $\xi_1, \ldots, \xi_k \in \CC$, the edges equations and holonomy equations can be written in the form
$$
\mathbf{G} \BLog \mathbf{z} + \mathbf{G'} \BLog \mathbf{z'} + \mathbf{G''} \BLog \mathbf{z''}
= (
    2\pi i,
    \dots,
    2\pi i,
   \xi_1, \ldots, \xi_k
)^{\!\top}.
$$
By using the equation $\Log z + \Log z' + \Log z'' = \pi i$, the equation can be rewritten as
$$\mathbf{A} \BLog \mathbf{z} + \mathbf{B} \BLog \mathbf{z''} 
= i \boldsymbol{\nu} + \tilde{\boldsymbol u},$$
where $\mathbf{A} = \mathbf{G-G'}, \mathbf{B} = \mathbf{G''-G'}$, $\boldsymbol{\nu} \in \pi \ZZ^N$ and $\tilde{\boldsymbol u} = (0,\dots,0,\xi_1, \ldots, \xi_k)^{\!\top}$.

\subsection{Combinatorial flattening}
Recall that a system of simple closed curves $\boldsymbol\gamma = (\gamma_1,\dots,\gamma_k)\in \pi_1(\partial (\mathcal{N}\setminus \nu(L)))$ consists of $k$ non-trivial simple closed curves $\gamma_i \in \pi_1(\mathbb{T}_i)$. Given a vector $\mathbf v\in \CC^N$, we denote the transpose of $\mathbf v$ by $\mathbf v^{\!\top}$.

\begin{definition}\label{CF}
For a system of simple closed curves $\boldsymbol\gamma = (\gamma_1,\dots,\gamma_k)\in \pi_1(\partial \left(\mathcal{N}\setminus \nu(L)\right) )$, a combinatorial flattening with respect to $\boldsymbol \gamma$ consists of three vectors 
\begin{align*}
\mathbf{f} = (f_1,\dots,f_N),\quad
\mathbf{f'} = (f_1',\dots,f_N'),\quad
\mathbf{f}'' = (f_1'',\dots,f_N'') \in \mathbb{Z}^N
\end{align*}
such that 
\begin{itemize}
\item for $i=1,\dots, n$, we have $f_i + f_i' + f_i'' = 1$ and
\item the $i$-th entry of the vector 
$$\mathbf{G}\cdot \mathbf{f}^{\!\top} + \mathbf{G'} \cdot {\mathbf{f}'}^{\!\top} + \mathbf{G''} \cdot {\mathbf{f}''}^{\!\top} $$ is equal to $2$ for $i=1,\dots, N-k$ and is equal to $0$ for $i=N-k+1,\dots, N$.
\end{itemize}
\end{definition}

We have the following stronger version of combinatorial flattening, which requires the last condition in Definition \ref{CF} to be satisfied for all simple closed curves.
\begin{definition}\label{SCF}
A strong combinatorial flattening consists of three vectors 
\begin{align*}
\mathbf{f} = (f_1,\dots,f_N),\quad
\mathbf{f'} = (f_1',\dots,f_N'),\quad
\mathbf{f}'' = (f_1'',\dots,f_N'') \in \mathbb{Z}^N
\end{align*}such that 
\begin{itemize}
\item for $i=1,\dots, n$, we have $f_i + f_i' + f_i'' = 1$ and
\item for {\textbf {\textit any}} system of simple closed curves $\boldsymbol \gamma = (\gamma_1,\dots,\gamma_k)$, the $i$-th entry of the vector 
$$\mathbf{G}\cdot \mathbf{f}^{\!\top} + \mathbf{G'} \cdot {\mathbf{f}'}^{\!\top} + \mathbf{G''} \cdot {\mathbf{f}''}^{\!\top} $$  is equal to $2$ for $i=1,\dots, n-k$ and is equal to $0$ for $i=n-k+1,\dots, n$.
\end{itemize}
\end{definition}
\begin{remark}
By \cite[Lemma 6.1]{N}, a strong combinatorial flattening exists for any ideal triangulation.
\end{remark}

\subsection{1-loop invariant and torsion as rational functions on the gluing variety}\label{1loopsection}

Let $\mathcal{N},L,\boldsymbol\gamma$ be as in the previous section.
\begin{definition}\label{defnrhoregular}
Let $\rho : \pi_1(\mathcal{N}\setminus L) \to \mathrm{PSL}(2;\CC)$ be a representation. An ideal triangulation $X$ of $\mathcal{N}\setminus L$ is $\rho$-regular if there exists $\mathbf{z} \in \mathcal{V}(X)$ such that $\mathcal{P}_{X}(\mathbf z) = [\rho]$.
\end{definition}
\begin{definition}[\cite{DG}] 
Let $X$ be an ideal triangulation of $\mathcal{N}\setminus L$ and $\mathbf{z}$ a shape structure on $X$.
Then the 1-loop invariant of $(\mathcal{N}\setminus L,\boldsymbol\gamma,X, \mathbf{z})$ is defined by
$$\tau(\mathcal{N}\setminus L, \boldsymbol\gamma, X, \mathbf{z}) = 
\pm \frac{1}{2} \mathrm{det}\Big( \mathbf{A} \Delta_{\mathbf{z}''} + \mathbf{B} \Delta_{\mathbf{z}}^{-1}\Big) \prod_{i=1}^N \Big(z_i^{f_i''} z_i''^{-f_i}\Big)
,$$
where $\mathbf{A}=\mathbf{G} - \mathbf{G'}, \mathbf{B} = \mathbf{G''} - \mathbf{G}$ and $(\mathbf{f},\mathbf{f}',\mathbf{f}'')$ is a strong combinatorial flattening, 
$$\Delta_{\mathbf{z}} 
= \begin{pmatrix}
z_1 & 0 & 0 & \dots & 0 \\
0 & z_2 & 0 & \dots & 0 \\
\vdots & \vdots & \vdots & \vdots & \vdots \\
0 & 0 & 0 & 0 & z_N
\end{pmatrix}
\quad \text{and} \quad
\Delta_{\mathbf{z}''} 
= \begin{pmatrix}
z_1'' & 0 & 0 & \dots & 0 \\
0 & z_2'' & 0 & \dots & 0 \\
\vdots & \vdots & \vdots & \vdots & \vdots \\
0 & 0 & 0 & 0 & z_N''
\end{pmatrix}.$$
\end{definition}

The 1-loop conjecture proposed by Dimofte and Garoufalidis suggests that the 1-loop invariant coincides with the adjoint twisted Reidemeister torsion. We used the following formulation of the conjecture from \cite{PW}.
\begin{conjecture}\label{1loopconjstatement}
Let $M$ be a hyperbolic 3-manifold $M$ with toroidal boundary $\partial M = \mathbb{T}_1 \coprod \dots \coprod \mathbb{T}_k$ and let $\boldsymbol\gamma$ be a system of simple closed curves of $\partial M$. Let $\rho_0$ be the unique discrete faithful representation corresponding to the complete hyperbolic structure of $M$. Let $X$ be a $\rho_0$-regular ideal triangulation of $M$ with $\mathcal{P}_{X}(\mathbf z_0) = [\rho_0]$. Let $\mathcal{V}_0(X)$ be the irreducible component of $\mathcal{V}(X)$ containing $\mathbf{z_0}$. For any $\mathbf{z} \in \mathcal{V}_0(X)$ with ${\mathcal{P}}_{X}(\mathbf{z}) = [\rho_{\mathbf{z}}]$, we have
$$
\tau(M, \boldsymbol\gamma, X, \mathbf z) 
= \pm \mathbb T_{(M,\boldsymbol\gamma)}([\rho_{\mathbf{z}}]).
$$
\end{conjecture}
Recently, Garoufalidis and Yoon proved the following result.
\begin{theorem}[\cite{GY}] \label{1looptwobridge}
The 1-loop conjecture holds for all hyperbolic two-bridge knot complements with the discrete faithful representation associated to the complete hyperbolic structure.
\end{theorem}
In Proposition \ref{Hesstotor}, we will show that the 1-loop invariant naturally shows up in the asymptotic expansion formula of the partition functions of the Teichm\"{u}ller TQFT invariants.

\subsection{FAMED triangulations}\label{sub:FAMED}
In this section we will expand on the definition of FAMED triangulations stated in Section \ref{sub:intro:FAMED}, and fix some notations for the proofs in the following sections.

For the rest of this section, we let $M$ be a one-cusped hyperbolic $3$-manifold with trivial second homology, $X$ be an ordered ideal triangulation of 
$M$ with $N$ tetrahedra and $l$ be a simple closed boundary curve of $M$.

Recall that 
$$Q = \frac{1}{2}\big[ -\mathcal{R}\mathcal{A}^{-1} \mathcal{B} - (\mathcal{R}\mathcal{A}^{-1} \mathcal{B} )^{\!\top} \big] 
,$$
where $\mathcal{R},\mathcal{A},\mathcal{B}$ are the matrices in the definition of $\mathscr{Z}$, $\mathcal{A}$ is assumed to be invertible for $Q$ to be defined, and $\alpha \in \mathcal{A}_X$ is an angle structure.

Let $\mathbf{G},\mathbf{G'},\mathbf{G''},\mathbf{A},\mathbf{B}$ be the Neumann-Zagier matrices associated to the boundary curve $l$. Recall that this means that:
\begin{enumerate}
    \item For any shape structure $\mathbf{z}$ on $X$ (not necessarily related to the fixed $\alpha$), we have
\begin{align*}
\mathbf{G} \BLog \mathbf{z} + \mathbf{G'} \BLog \mathbf{z'} + \mathbf{G''} \BLog \mathbf{z''} = 
\begin{pmatrix}
2i \pi \\
\vdots \\
2i \pi\\
\mathrm{H}^\C_{X,l}(\mathbf{z})
\end{pmatrix},
\end{align*}
$$\mathbf{A}=\mathbf{G} -\mathbf{G'}, \ \ \mathbf{B}=\mathbf{G''}-\mathbf{G'},      $$
\begin{align}\label{thurstoneqs}
\mathbf{A} \BLog \mathbf{z} +  \mathbf{B} \BLog \mathbf{z''} = i \boldsymbol \nu +  
\tilde{\boldsymbol u},
\end{align}
where 
$
\boldsymbol \nu
\in \pi \ZZ^N$ and 
 $\tilde{\boldsymbol u}
= (0,\dots,0,
\mathrm{H}^\C_{X,l}(\mathbf{z}))^{\!\top}$
\item By passing to the imaginary part in (1), for any angle structure $\alpha \in \mathcal{A}_X$, we have
\begin{align*}
\mathbf{G} \begin{pmatrix}
\boldsymbol{a_+}\\
\boldsymbol{a_-}
\end{pmatrix}
+ \mathbf{G'} \begin{pmatrix}
\boldsymbol{c_+}\\
\boldsymbol{b_-}
\end{pmatrix} 
+ \mathbf{G''} 
\begin{pmatrix}
\boldsymbol{b_+}\\
\boldsymbol{c_-}
\end{pmatrix}
= 
\begin{pmatrix}
2 \pi \\
\vdots \\
2 \pi\\
\mathrm{H}^\R_{X,l}(\alpha)
\end{pmatrix},
\end{align*}
$$\mathbf{A}=\mathbf{G} -\mathbf{G'}, \ \ \mathbf{B}=\mathbf{G''}-\mathbf{G'},      $$
\begin{align}\label{thurstoneqs2}
\mathbf{A} \begin{pmatrix}
\boldsymbol{a_+}\\
\boldsymbol{a_-}
\end{pmatrix} +  \mathbf{B} 
\begin{pmatrix}
\boldsymbol{b_+}\\
\boldsymbol{c_-}
\end{pmatrix}
= \boldsymbol \nu +  
{\boldsymbol u},
\end{align}
where 
$
\boldsymbol \nu
\in \pi \ZZ^N$ and 
 $\boldsymbol u
= (0,\dots,0,
\mathrm{H}^\R_{X,l}(\alpha))^{\!\top}$.
\end{enumerate}

When $\mathcal{A}$ is invertible,
we further denote
$$\mathscr{G} := 
-2
Q
+
\frac{\mathcal{E}+{\rm Id}_N}{2}
=
\mathcal{X}_0 \mathcal{A}^{-1} \mathcal{B} + \left ( \mathcal{X}_0 \mathcal{A}^{-1} \mathcal{B} \right )^{\!\top} + \frac{\mathcal{E}+{\rm Id}_N}{2}.
$$
Hence, condition (4) on Definition \ref{def:FAMED} can be restated as
$$\mathscr{G}=\mathbf{B}^{-1}\mathbf{A}.$$

\subsection{Saddle point approximation}
The following version of saddle point approximation is a special case of \cite[Proposition 5.1]{WY2}. 
\begin{proposition}[Proposition 5.1, \cite{WY2}]\label{saddle}
Let $D$ be a region in $\mathbb C^{N}$. Let $f(\mathbf z)$ and $g(\mathbf z)$ be complex valued functions on $D$  which are holomorphic in $\mathbf z$. For each positive $r,$ let $f_r(\mathbf z)$ be a complex valued function on $D$ holomorphic in $\mathbf z$ of the form
$$ f_r(\mathbf z) = f(\mathbf z) + \frac{\upsilon_r(\mathbf z)}{r^2},$$
where $\upsilon_r(\mathbf z)$ is another complex valued function on $D$ holomorphic in $\mathbf z$.
Let $S$ be an embedded $N$-dimensional disk and $\mathbf c$ be a critical point of $f$ on $S$. If for each $r>0$
\begin{enumerate}[(1)]
\item $\mathrm{Re}f(\mathbf c) > \mathrm{Re}f(\mathbf z)$ for all $\mathbf z \in S\setminus \{\mathbf c\},$
\item the domain $\{\mathbf z\in D\ |\ \mathrm{Re} f(\mathbf z) < \mathrm{Re} f(\mathbf c)\}$ deformation retracts to $S\setminus\{\mathbf c\},$
\item $g(\mathbf c)\neq 0$,
\item $|\upsilon_r(\mathbf z)|$ is bounded from above by a constant independent of $r$ on $D,$ and
\item  the Hessian matrix $\mathrm{Hess}(f)$ of $f$ at $\mathbf c$ is non-singular,
\end{enumerate}
then as $r\to \infty$,
\begin{equation*}
\begin{split}
 \int_{S} g(\mathbf z) e^{rf_r(\mathbf z)} d\mathbf z= \Big(\frac{2\pi}{r}\Big)^{\frac{{N}}{2}}\frac{g(\mathbf c)}{\sqrt{(-1)^{{N}}\det\mathrm{Hess}(f)(\mathbf c)}} e^{rf(\mathbf c)} \Big( 1 + O \Big( \frac{1}{r} \Big) \Big).
 \end{split}
 \end{equation*}
\end{proposition}

\section{Asymptotics of the Teichm\"{u}ller TQFT partition functions}\label{sec:asymp:expan}

In this section, we will prove Theorem \ref{mainthmZ}, and in particular the asymptotic expansion for the partition function of Teichmüller TQFT at a hyperbolic cone structure. For this we will refine and expand on the analytical techniques used in \cite{BAGPN}, and highlight the role of the FAMED condition.

For the rest of this section, we let $M$ be a one-cusped hyperbolic $3$-manifold with trivial second homology, $X$ be an ordered ideal triangulation of 
$M$ with $N$ tetrahedra, $\alpha \in \mathscr{A}_X$ be an angle structure of $X$ and $l$ be a simple closed boundary curve of $M$.

\subsection{Expression of the partition functions for FAMED ideal triangulations}\label{sect:expression}
We follow the notations of Section \ref{sub:FAMED}.
Let $X = X_+ \coprod X_-$, where $X_\pm$ consists of positive/negative tetrahedra respectively.  We emphasize the difference between positive and negative tetrahedra by writing $\mathbf{y} = (\mathbf{y_+}, \mathbf{y_-})^{\!\top}$, $\boldsymbol{a = (a_+,a_-), c= (c_+,c_-)}$, etc. 
Hence, with this numbering of tetrahedra, the matrix of tetrahedra signs is simply written as $\mathcal{E}=\begin{pmatrix}
    \mathbf{1} & \mathbf{0}\\\mathbf{0}&\mathbf{-1}
\end{pmatrix}$.
Recall that 
$$Q = \frac{1}{2}\big[ -\mathcal{R}\mathcal{A}^{-1} \mathcal{B} - (\mathcal{R}\mathcal{A}^{-1} \mathcal{B})^{\!\top} \big] \quad \text{and} \quad \mathscr{W}(\alpha) = 2Q (\boldsymbol{\pi-a}) -C(\alpha),$$
where $\mathcal{R},\mathcal{A},\mathcal{B}$ are the matrices in the definition of $\mathscr{Z}$, $\mathcal{A}$ is assumed to be invertible for $Q$ to be defined, $\alpha \in \mathcal{A}_X$ is the angle structure we fixed at the beginning of the section, and 
$C(\alpha) = (-\mathbf{c_+}, \mathbf{c_-})$.

Let $\mathbf{G},\mathbf{G'},\mathbf{G''},\mathbf{A},\mathbf{B}$ be the Neumann-Zagier matrices associated to the boundary curve $l$. 
When $\mathcal{A}$ is invertible, let
$$\mathscr{G} = 
-2
Q
+
\begin{pmatrix}
\mathbf{1} & \mathbf{0} \\
\mathbf{0} & \mathbf{0}
\end{pmatrix}=
\mathcal{X}_0 \mathcal{A}^{-1} \mathcal{B} + \left ( \mathcal{X}_0 \mathcal{A}^{-1} \mathcal{B} \right )^{\!\top} + \frac{\mathcal{E}+{\rm Id}_N}{2}.$$

The following lemma and proposition will clarify how the FAMED condition implies that the only contribution of $\alpha$ in $|\mathcal{Z}_\hbar(X,\alpha)|$ is $\mathrm{H}^\R_{X,l}(\alpha)$ (like in Conjecture \ref{conj:vol:BAGPN} (1) and (3a)).

\begin{lemma}\label{computeW}
Assume that $X$ is a FAMED triangulation for the curve $l$. Then
we have 
\begin{align*}
\mathscr{W}(\alpha)
=
\mathbf{B}^{-1}(\boldsymbol \nu +\boldsymbol{u}) - 
\mathscr{G}
\boldsymbol{\pi},
\end{align*}
where $\mathbf{B}$ depends on $l$ as before,
$
\boldsymbol \nu
\in \pi \ZZ^N$ and 
 $\boldsymbol u
= (0,\dots,0,
\mathrm{H}^\R_{X,l}(\alpha))^{\!\top}$.
In particular, $\mathscr{W}(\alpha)$ only depends on $\mathrm{H}^\R_{X,l}(\alpha)$ (via $\boldsymbol u$).
\end{lemma}
\begin{proof}
Recall that for positively (resp. negatively) ordered tetrahedron, we have $a = {\Arg(z)}$ and $b = {\Arg(z'')}$ (resp. $a = {\Arg(z)}$ and $c = {\Arg(z'')}$). Besides, by definition of a FAMED triangulation (see Definition \ref{def:FAMED} and Section \ref{sub:FAMED}), we have
$$\mathscr{G} = 
-2
Q
+
\begin{pmatrix}
\mathbf{1} & \mathbf{0} \\
\mathbf{0} & \mathbf{0}
\end{pmatrix}
= \mathbf{B}^{-1} \mathbf{A}.
$$
By a direct computation, for any FAMED ideal triangulation, we have 
\begin{align*}
\mathscr{W}(\alpha)
&= 
-2
Q
\begin{pmatrix}
\boldsymbol{a_+} - \boldsymbol\pi \\
\boldsymbol{a_-} - \boldsymbol\pi
\end{pmatrix}
+
\begin{pmatrix}
\boldsymbol{a_+} + \boldsymbol{b_+} - \pi\\
\boldsymbol{c_-}
\end{pmatrix}
\\
&= 
\left(
\mathscr{G}
-
\begin{pmatrix}
\mathbf{1} & \mathbf{0} \\
\mathbf{0} & \mathbf{0}
\end{pmatrix}
\right)
\begin{pmatrix}
\boldsymbol{a_+} - \boldsymbol\pi \\
\boldsymbol{a_-} - \boldsymbol\pi
\end{pmatrix}
+
\begin{pmatrix}
\boldsymbol{a_+} + \boldsymbol{b_+} - \pi\\
\boldsymbol{c_-}
\end{pmatrix}
\\
&=
\left[
\mathscr{G}
\begin{pmatrix}
\boldsymbol{a_+}\\
\boldsymbol{a_-}
\end{pmatrix}
+
\begin{pmatrix}
\boldsymbol{b_+}\\
\boldsymbol{c_-}
\end{pmatrix}
\right]
-
\mathscr{G}
\begin{pmatrix}
\boldsymbol{\pi}\\
\boldsymbol{\pi}
\end{pmatrix}\\
&=
\mathbf{B}^{-1}
\left[
\mathbf{A}
\begin{pmatrix}
\boldsymbol{a_+}\\
\boldsymbol{a_-}
\end{pmatrix}
+
\mathbf{B}
\begin{pmatrix}
\boldsymbol{b_+}\\
\boldsymbol{c_-}
\end{pmatrix}
\right]
-
\mathscr{G}
\begin{pmatrix}
\boldsymbol{\pi}\\
\boldsymbol{\pi}
\end{pmatrix}\\
&=
\mathbf{B}^{-1}(\boldsymbol \nu +\boldsymbol{u}) - 
\mathscr{G}
\boldsymbol{\pi},
\end{align*}
where the last equality follows from the definition of the angle structure $\alpha$ (see (\ref{thurstoneqs2})).
\end{proof}
 
\begin{proposition}\label{Tpartiexpress2}
Let $X$ be a FAMED ideal triangulation. 
Let $\alpha \in \mathscr{A}_{X}$ be an angle structure given by $\alpha = (a_1,b_1,c_1,\dots,a_N, b_N, c_N)$. Let $\mathscr{Y}_\alpha $ be a multi-contour given by
$$\mathscr{Y}_\alpha = \prod_{k=1}^N \Big( \RR + i(\pi - a_k)\Big). $$
For all $\hbar>0$, we have 
\begin{align}\label{Tpartexpress2formhbar}
&\ |\mathcal{Z}_{\hbar}(X,\alpha) | = \frac{1}{|\det \mathcal{A}|}
\Big(\frac{1}{2\pi \sqrt{\hbar}}\Big)^{N}\times \notag\\
&\   \left| 
\int_{\mathscr{Y}_\alpha} 
e^{\frac{1}{2\pi\hbar}\left(i \mathbf{y}^{\!\top} Q \mathbf{y} + \mathbf{y}^{\!\top} (\mathbf{B}^{-1}(\boldsymbol \nu +\boldsymbol{u}) - 
\mathscr{G}
\boldsymbol{\pi})
- i\sum_{k=1}^N \left(\frac{\varepsilon(T_k)+1}{4}\right) y_k^2\right)}
\prod_{k=1}^N \Phi_\B\left(\frac{y_k}{2\pi \sqrt{\hbar}}\right) d\mathbf{y} \right|.
\end{align}
Moreover, for any $\alpha_1, \alpha_2 \in \mathcal{A}_X$ with $\mathrm{H}^\R_{X,l}(\alpha_1) = \mathrm{H}^\R_{X,l}(\alpha_2)$, we have $|\mathscr{Z}_{\hbar}(X, \alpha_1)|  = |\mathscr{Z}_{\hbar}(X, \alpha_2)| $.
\end{proposition}
\begin{proof}
The expression of $|\mathscr{Z}_{\hbar}(X, \alpha)|$ follows from Proposition \ref{Tpartiexpress1} and Lemma \ref{computeW}. For the second result, note that although the integrand only depends on $\mathrm{H}^{\R}_{X,l}(\alpha)$, a priori the integration multi-contour depends on $\alpha$. We are going to show that we can change the integration multi-contour from $\mathscr{Y}_{\alpha_1}$ to $\mathscr{Y}_{\alpha_2}$ without changing the integral.

Let $\alpha_1, \alpha_2 \in \mathscr{A}_X$ with same angular holonomy $\theta := \mathrm{H}^\R_{X,l}(\alpha_1)=\mathrm{H}^\R_{X,l}(\alpha_2)$ along the longitude. Consider the subspace of angle structure
$\mathscr{A}^l_{X}(\theta):=  \{ \alpha \in \mathscr{A}_X \mid \mathrm{H}^\R_{X,l}(\alpha) = \theta \}$. Since $\mathscr{A}_X$ is convex, $\mathscr{A}^l_{X}(\theta)$, which is a restriction of a convex set on a linear subspace, is also convex. Consider the projection map $\mathrm{proj}_a: \mathscr{A}^l_{X}(\theta) \to \RR^N$ defined by $\mathrm{proj}_a(\alpha) = (a_1,\dots, a_N)$. Since $\mathrm{proj}_a$ is an affine map and $\mathscr{A}^l_{X}(\theta)$ is convex, the image $\mathrm{proj}_a(\mathscr{A}^l_{X}(\theta) )$ is also convex. 

Moreover, since $\mathbf{B}^{-1}$ is invertible, for any $(a_1,\dots, a_N)$ sufficiently close to $\mathrm{proj}_a(\alpha_1)$, by using (\ref{thurstoneqs2}) and the equations $a_k + b_k + c_k = \pi$ for $k=1,\dots, N$, we can find the corresponding $(a_1,b_1,c_1,\dots, a_N, b_N, c_N)$ that depends continuously on $(a_1,\dots,a_N)$ and satisfies $(\ref{thurstoneqs2})$. In particular, by continuity, since $\alpha_1 \in \mathcal{A}_X$, for any $(a_1,\dots, a_N)$ sufficiently close to $\mathrm{proj}_a(\alpha_1)$, we have $(a_1,b_1,c_1,\dots, a_N, b_N, c_N) \in (0,\pi)^N$. Since $(a_1,b_1,c_1,\dots, a_N, b_N, c_N)$ satisfies (\ref{thurstoneqs2}), we have $(a_1,b_1,c_1,\dots, a_N, b_N, c_N)\in  \mathscr{A}^l_{X}(\theta)$. Therefore, for any $(a_1,\dots, a_N)$ sufficiently close to $\mathrm{proj}_a(\alpha_1)$, we can find an angle structure $\alpha \in \mathscr{A}^l_{X}(\theta)$ such that $\mathrm{proj}_a(\alpha) = (a_1,\dots, a_N)$. 

This shows that the image $\mathrm{proj}_a(\mathscr{A}^l_{X}(\theta) )$ has dimension $N$. Altogether, for any $\alpha'$ such that $\mathrm{proj}_a(\alpha')=(a_1', \dots, a_N')$ satisfies $a_k' \in [\min\{a^1_k, a^2_k\}, \max\{a^1_k,a^2_k\}]$, the formula for $|\mathscr{Z}_{\hbar}(X, \alpha')|$ holds. In particular, this implies the exponential decay properties of the integrand at infinity and the absolute convergence of the integral for all such $(a_1', \dots, a_N')$. As a result, by Bochner-Martinelli formula (see \cite{Kr}), we have
$|\mathscr{Z}_{\hbar}(X, \alpha_1)| = |\mathscr{Z}_{\hbar}(X, \alpha_2)|$
for all $\alpha_1, \alpha_2 \in \mathscr{A}_X$ with $\mathrm{H}^\R_{X,l}(\alpha_1)=\mathrm{H}^\R_{X,l}(\alpha_2)=\theta$. This proves the second result. 
\end{proof}

\begin{remark}
Given $\theta\in \R$, for all $\hbar>0$, the integrand in (\ref{Tpartexpress2formhbar}) does not depends on the choices of $\alpha \in  \mathscr{A}^l_{X}(\theta)$. Moreover, the integration multi-contour can be chosen to be any $\mathscr{Y}_\alpha$ as long as $\alpha \in  \mathscr{A}^l_{X}(\theta)$. In this sense, we say that $|\mathscr{Z}_{\hbar}(X, \alpha)|$ depends only on $\mathrm{H}^\R_{X,l}(\alpha)$.
\end{remark}

\subsection{Potential function and its properties}\label{sect:potential:prop}
Consider the potential function 
\begin{align*}
S\left (\mathbf{y}; \mathrm{H}^\R_{X,l}(\alpha)\right ) 
= i \mathbf{y}^{\!\top}
Q \mathbf{y}
 + \mathbf{y}^{\!\top} (\mathbf{B}^{-1}(\boldsymbol \nu + \boldsymbol{u}) - 
\mathscr{G}
\boldsymbol{\pi}) - i\sum_{k=1}^N \left(\frac{\varepsilon(T_k)+1}{4}\right) y_k^2
 - i \sum_{k=1}^N \mathrm{Li}_2\left(-e^{y_k}\right),
\end{align*}
where $\boldsymbol u
= (0,\dots,0,
\mathrm{H}^\R_{X,l}(\alpha))^{\!\top}$.
The function $S$ plays an important role in the asymptotics of the partition function. More generally, we can complexify the parameter $\mathrm{H}^\R_{X,l}(\alpha)$ and consider the potential function $\tilde{S}$ defined by
\begin{align*}
 \tilde{S}\left (\mathbf{y}; \xi\right ) 
= i \mathbf{y}^{\!\top}
Q \mathbf{y}
 + \mathbf{y}^{\!\top} (\mathbf{B}^{-1}(\boldsymbol \nu - i \boldsymbol{\tilde u}) - 
\mathscr{G}
\boldsymbol{\pi}) - i\sum_{k=1}^N \left(\frac{\varepsilon(T_k)+1}{4}\right) y_k^2
 - i \sum_{k=1}^N \mathrm{Li}_2\left(-e^{y_k}\right),
\end{align*}
where $\boldsymbol{\tilde u}
= (0,\dots,0,
\xi)$ with $\xi \in \C$.
The functions $S$ and $\tilde{S}$ are related by 
$$
S\left(\mathbf{y};\mathrm{H}^\R_{X,l}(\alpha)\right)
= \tilde{S}\left (\mathbf{y}; i\mathrm{H}^\R_{X,l}(\alpha) \right ) .
$$
In particular, in this subsection, every results on $\tilde S$ also apply to $S$. The function $\tilde{S}$ will also play a role in Section \ref{sec:Jones} when we study the asymptotics of the Jones function. We relate the critical point equations of $\tilde{S}$ with the gluing equations as follows.

\begin{proposition}\label{critThurscorrespondence}
Assume that $\mathbf{B}$ is invertible and $\mathscr{G} = \mathbf{B}^{-1} \mathbf{A}$ (for example if $X$ is FAMED for $l$). Then via the bijection 
$$y_k = -\Log z_k + i \pi$$ 
defined in Section \ref{sub:thurston}, we have the correspondence between the critical point equation
$$
\nabla \tilde S = 0
$$
and the gluing equations
\begin{align*}
\mathbf{A} \BLog \mathbf{z}
+ \mathbf{B}
\BLog \mathbf{z}''
=
i \boldsymbol \nu + \tilde{\boldsymbol{u}}.
\end{align*}
In particular, for the function $S\left (\mathbf{y}; \mathrm{H}^\R_{X,l}(\alpha)\right ) $,
if one considers the system of (\ref{thurstoneqs}) where $\mathrm{H}^\C_{X,l}(\mathbf{z})$ is replaced with $i \mathrm{H}^\R_{X,l}(\alpha)$, then any solution $\mathbf{y}$ (or $\mathbf{z}$ according to the point of view) to this system gives a critical point $\mathbf{y_c(\alpha)}$ of $\mathbf{y}\mapsto S\left(\mathbf{y},\mathrm{H}^\R_{X,l}(\alpha)\right )$.
 
Moreover, it follows that 
for each angle structure $\alpha' \in \mathcal{A}_{X}^l( \mathrm{H}^\R_{X,l}(\alpha))$, the logarithmic shape structure $\mathbf{y_c(\alpha')}$ associated to $\alpha'$ (as in Section \ref{sub:thurston}) gives a critical point of $(\Re S)\left(\mathbf{y},\mathrm{H}^\R_{X,l}(\alpha)\right )$ relatively to the real parts of the variables $y_j$.
\end{proposition}
\begin{proof}
Note that since $Q$ is symmetric, we have
\begin{align*}
\nabla \tilde{S}(\mathbf{y})
&= 2i Q 
\mathbf{y}
 + 
 (\mathbf{B}^{-1}(\boldsymbol \nu - i\boldsymbol{\tilde u}) - 
\mathscr{G}
\boldsymbol{\pi}) 
- i\left(\frac{\mathcal{E}+\rm{Id}}{2}\right)\mathbf{y}
+ i 
\mathbf{ \BLog(1+e^{y_-})}\\
&= -i\mathcal{G}\mathbf{y} + 
 (\mathbf{B}^{-1}(\boldsymbol \nu - i\boldsymbol{\tilde u}) - 
\mathscr{G}
\boldsymbol{\pi}) + i 
\mathbf{ \BLog(1+e^{y_-})} \\
&= i\mathcal{G}(-\mathbf{y} + i\boldsymbol{\pi}) +  
 \mathbf{B}^{-1}(\boldsymbol \nu - i\boldsymbol{\tilde u})  + i 
\mathbf{ \BLog(1+e^{y_-})}
\end{align*}
Recall that $y_k = -\Log z_k + i \pi$ and $\Log(1+e^{y_k}) = \Log(z_k'').$
Thus, we have
\begin{align}\label{gradScomputation}
\nabla \tilde{S}(\mathbf{y})
&= i\mathcal{G}\BLog\mathbf{z} + i 
\mathbf{ \BLog z''} +  
 \mathbf{B}^{-1}(\boldsymbol \nu - i\boldsymbol{\tilde u}) ,
\end{align}
which is equal to zero if and only if 
\begin{align}\label{eq:gradS}
i\mathcal{G}\BLog\mathbf{z} + i \BLog
\mathbf{ z''} = -  
 \mathbf{B}^{-1}(\boldsymbol \nu - i\boldsymbol{\tilde u}).
\end{align}
By assumption, since $\mathscr{G} = \mathbf{B}^{-1} \mathbf{A}$, the equation above is equivalent to
\begin{align*}
i\left(
\mathbf{A} \BLog \mathbf{z}
+ \mathbf{B}
\BLog \mathbf{z''}
\right)
=
- (\boldsymbol \nu - i\boldsymbol{\tilde u}).
\end{align*}
This gives the desired result.

Finally, by passing to the real part in the previous equivalence \ref{eq:gradS}, we deduce the last sentence of the Proposition.
\end{proof}

We temporarily focus specifically on the property of $S$, or in other words the case where $\xi = i \mathrm{H}^{R}_{X,l}(\alpha)$. The following proposition relates
 the critical values of $S\left (\mathbf{y};\mathrm{H}^\R_{X,l}(\alpha)\right )$ at the critical points defined in Proposition \ref{critThurscorrespondence} to hyperbolic volumes.
 
\begin{proposition}\label{dilogVolgen}
Write $y_l = h_l + i d_l \in \CC$ for $l=1,\dots,N$.
Under the assumption and the correspondence described in Proposition \ref{critThurscorrespondence}, we have 
\begin{align*}
\mathrm{Re}S\left (\mathbf{y};\mathrm{H}^\R_{X,l}(\alpha)\right )
= - \sum_{l}^N D(z_l) + \sum_{l=1}^N h_l \frac{\partial}{\partial h_l} 
\mathrm{Re}S\left (\mathbf{y};\mathrm{H}^\R_{X,l}(\alpha)\right ),
\end{align*}
where $D(z)$ is the Bloch-Wigner dilogarithm function given by
$$ D(z) 
= \mathrm{Im} \mathrm{Li}_2(z) + \log|z| \mathrm{Arg}(1-z).$$
In particular, for every angle structure $\alpha'\in \mathcal{A}_{X}^l( \mathrm{H}^\R_{X,l}(\alpha) )$ and the corresponding $\mathbf{y}_c(\alpha')$ defined in Proposition \ref{critThurscorrespondence}, we have
$$
\mathrm{Re}S\left (\mathbf{y}_c(\alpha');\mathrm{H}^\R_{X,l}(\alpha)\right )
= - \Vol(\alpha'),
$$
where $\Vol(\alpha')$ is the volume of the angle structure $\alpha'$.

In particular, if we let $\mathbf{y}_c(\alpha)$ be the critical point of $S$ described in Proposition \ref{critThurscorrespondence}, then we have
$$
\mathrm{Re}S\left (\mathbf{y}_c(\alpha);\mathrm{H}^\R_{X,l}(\alpha)\right ) =  -\Vol\Big(
M; l, \mathrm{H}^\R_{X,l}(\alpha)
\Big),
$$
 is the volume of $M$ with the (possibly incomplete) hyperbolic cone structure satisfying $\mathrm{H}^\C_{X,l}(\mathbf{z})=i \mathrm{H}^\R_{X,l}(\alpha)$. 
\end{proposition}
\begin{proof}
Recall that
\begin{align*}
S\left (\mathbf{y}; \mathrm{H}^\R_{X,l}(\alpha)\right ) 
= i \mathbf{y}^{\!\top}
Q \mathbf{y}
 + \mathbf{y}^{\!\top} (\mathbf{B}^{-1}(\boldsymbol \nu + \boldsymbol{u}) - 
\mathscr{G}
\boldsymbol{\pi}) - i\sum_{k=1}^N \left(\frac{\varepsilon(T_k)+1}{4}\right) y_k^2
 - i \sum_{k=1}^N \mathrm{Li}_2\left(-e^{y_k}\right).
\end{align*}
Write $y=h+id$. Note that 
$$
\frac{d}{d y} \Big(i\mathrm{Li}_2(-e^{y}) \Big) =  -i \Log (1+e^{y}).
$$
By using the Cauchy-Riemann equation, we have
$$
\frac{\partial}{\partial h} \mathrm{Re}\Big(i \mathrm{Li}_2(-e^{y})\Big)  = \mathrm{Re} \Bigg( \frac{d}{dy} \Big(i \mathrm{Li}_2(-e^{y})\Big) \Bigg) =  \mathrm{Arg} (1+e^{-y}).
$$
As a result, by the definition of the Bloch-Wigner dilogarithm function,
$$
\mathrm{Re}\Big(i \mathrm{Li}_2(-e^{y})\Big)  
=
-\mathrm{Im} \mathrm{Li}_2(-e^y)
=  - D(-e^y) + h \mathrm{Arg}(1+e^y)
= -D(-e^y) + h\frac{\partial}{\partial h} \mathrm{Re}\Big(i \mathrm{Li}_2(-e^y)\Big).
$$
Thus, we have
$$
\mathrm{Re}\Big(- i \mathrm{Li}_2(-e^{y_l})\Big)  
= D(-e^{y_l}) + h_l\frac{\partial}{\partial h_l} \mathrm{Re}\Big(-i \mathrm{Li}_2(-e^{y_l})\Big)
= - D(z_l) + h_l\frac{\partial}{\partial h_l} \mathrm{Re}\Big(-i \mathrm{Li}_2(-e^{y_l})\Big).
$$ 
Finally, note that the real part of all the remaining terms (the non-dilogarithm terms) are linear in $\{h_l \mid l=1,\dots, N\}$, thus
\begin{align*}
&\ i \mathbf{y}^{\!\top}
Q \mathbf{y}
 + \mathbf{y}^{\!\top} (\mathbf{B}^{-1}(\boldsymbol \nu + \boldsymbol{u}) - 
\mathscr{G}
\boldsymbol{\pi}) - i\pi \sum_{k=1}^N \left(\frac{\varepsilon(T_k)+1}{4}\right) y_k^2 \\
=&\ \sum_{l=1}^N h_l \frac{\partial }{\partial h_l} \mathrm{Re}
\left(  i (\mathbf{y_+}, \mathbf{y_-})
Q  \begin{pmatrix}
\mathbf{y_+} \\
 \mathbf{y_-}
\end{pmatrix}
 + (-\mathbf{y_+}, \mathbf{y_-}) (\mathbf{B}^{-1}(\boldsymbol \nu +\boldsymbol{u}) - 
\mathscr{G}
\boldsymbol{\pi}) 
- i \sum_{k=1}^N \left(\frac{\varepsilon(T_k)+1}{4}\right) y_k^2 \right).
\end{align*}
The result then follows from summing up the equations.
\end{proof}

Let $\mathscr{U}$ be the product of horizontal bands defined by 
$$
\mathscr{U} = \prod_{k=1}^N (\RR + i (0,\pi)).
$$
\begin{proposition}\label{concavSgen}
For any $\alpha \in \mathscr{A}_{X}$, the real part of $S(\mathbf{y}; \mathrm{H}^\R_{X,l}(\alpha))$ 
is strictly concave in the real variables $\Re(\mathbf{y})$ on any real multi-contour of the form
$$\mathcal{Y}_{\alpha^0}=  \prod_{k=1}^N (\RR + i (\pi-a^0_k)),$$
where $\alpha^0\in \mathcal{A}_X$ is an angle structure.
\end{proposition}
\begin{proof}
Note that the hessian of the real part of $S$ on $\mathscr{Y}_\alpha$ is the same as the real part of the holomorphic hessian of $S$. By a direct computation, for any $\mathbf{y} = \mathbf{h} + i \mathbf{d}$ with $ \mathbf{h} = (h_1,\dots,  h_N), \mathbf{d} = (d_1,\dots,d_N) \in \RR^{N}$, we have 
\begin{align*}
\big(\mathrm{Hess}(\Re (S)|_{\mathscr{Y}^0_\alpha}) \big) (\mathbf{h} + i \mathbf{d};\mathrm{H}^\R_{X,l}(\alpha))
= \mathrm{Re} ( \mathrm{Hess}( S ) (\mathbf{h} + i \mathbf{d};\mathrm{H}^\R_{X,l}(\alpha)))=\Delta
\end{align*} 
where $\Delta$ is the $N \times N$ diagonal matrix with entries 
$$ - \mathrm{Im} \bigg( \frac{1}{1+e^{-h_k - i d_k}}\bigg).$$
The conditions that $d_k \in (0,\pi)$ implies that the matrix above is negative definite for all $\mathbf{h} \in \RR^{N}$. 
\end{proof}

Now, we move our attention back to the function $\tilde S$, and
explain how the $1$-loop invariant $\tau$ discussed in Section \ref{1loopsection} appears in the asymptotic expansion of the partition function.

\begin{proposition}\label{Hesstotor}
Assume that $\mathbf{B}$ is invertible and $\mathscr{G} = \mathbf{B}^{-1} \mathbf{A}$. At the critical point $\mathbf{y}=\mathbf{y_c}(\xi)$ of $\tilde S$ described in Proposition \ref{critThurscorrespondence}, we have 
\begin{align*}
\frac{1}{\sqrt{\pm \det \Hess (\tilde S)(\mathbf{y};\xi})}
&= \frac{1}{\sqrt{\pm 2 i^N \det\mathbf{B}^{-1} \left(\prod_{i=1}^N z_i^{-f_i''} z_i''^{f_i - 1} \right)
\tau(M, l, X, \mathbf{z})
}} 
\end{align*}
\end{proposition}
\begin{proof}
Recall from (\ref{gradScomputation}) that 
\begin{align*}
\nabla \tilde{S}(\mathbf{y})
&= i\mathcal{G}\BLog\mathbf{z} + i 
\mathbf{ \BLog z''} +  
 \mathbf{B}^{-1}(\boldsymbol \nu - i\boldsymbol{\tilde u}) .
\end{align*}
Thus, we have
\begin{align*}
- i \mathbf{B} 
\nabla \tilde S
= \mathbf{A} \BLog \mathbf{z}
+ \mathbf{B} \BLog \mathbf{z''}
- i(\boldsymbol \nu -i\boldsymbol{\tilde u}).
\end{align*}
Furthermore, since 
$$ \frac{\partial}{\partial y_k } =  \frac{\partial z_k}{\partial y_k } \cdot \frac{\partial}{\partial z_k } = \pm z_k \frac{\partial}{\partial z_k } , $$ 
we have
\begin{align*}
&\pm \det \Hess (\tilde S(\mathbf{y}; \xi))\\
=& \pm \det D_\mathbf{y} \nabla \tilde S (\mathbf{y}; \xi) \\
=&\pm i^{N} \det \mathbf{B}^{-1} \left(\prod_{k=1}^N z_k \right)\det D_{\mathbf{z}} \left(\mathbf{A} \BLog \mathbf{z}
+ \mathbf{B} \BLog \mathbf{z''} \right) \\
=&\pm i^{N} \det \mathbf{B}^{-1}  \left(\prod_{k=1}^N z_k \right)\det (\mathbf{A} \Delta_{z''} + \mathbf{B} \Delta_{z}^{-1}) \left(\prod_{k=1}^N \frac{1}{1-z_k} \right)\\
=& \pm 2 i^N \det \mathbf{B}^{-1}   \left(\prod_{k=1}^N z_k \right)\left( \prod_{k=1}^N z_k^{-f_k''} z_k''^{f_k} \right) 
\tau(M,X, l, \mathbf{z})
\left(\prod_{k=1}^N \frac{1}{1-z_k} \right)\\
=& \pm 2 i^N \det \mathbf{B}^{-1} \left(\prod_{k=1}^N z_k^{-f_k''} z_k''^{f_k - 1} \right)
\tau(M,X, l, \mathbf{z}).
\end{align*}
This gives the desired result.
\end{proof}

Consider the function
\begin{align}\label{hdefn}
h( \mathbf{y}) = \exp\left(\sum_{k=1}^N \left(\frac{i}{2\pi}y_k\log(1+e^{y_k}) + \frac{i}{\pi}\Li \left(- e^{y_k}\right) \right) \right) .
\end{align}
\begin{proposition}\label{hvalue}
    At the critical point $\mathbf{y}=\mathbf{y_c}(\xi)$ of $\tilde S$ with the corresponding shape parameters $\mathbf{z} = \mathbf{z}(\xi)$ described in Proposition \ref{critThurscorrespondence}, we have 
    \begin{align*}
    \left|\sqrt{\left(\prod_{k=1}^N z_k^{f_k''} z_k''^{1-f_k} \right)}h( \mathbf{y}) \right|
    = \left|\exp\left(\frac{i}{\pi} R(\mathbf{z}) \right) \right|,
    \end{align*}
    where
    $$
    R(\mathbf{z}) = -\frac{1}{2}(\BLog \mathbf{z} - i\pi \mathbf{f})\cdot (\BLog \mathbf{z''} + i\pi \mathbf{f''}) + \sum_{k=1}^N \Li \left(e^{-\Log z_k}\right) .
    $$
\end{proposition}
\begin{proof}
    Note that 
    \begin{align*}
        \sqrt{\left(\prod_{k=1}^N z_k^{f_k''} z_k''^{1-f_k}\right)}
        =&\ \exp\left( \frac{1}{2} \sum_{k=1}^N \left(f_k'' \Log z_k + (1-f_k) \Log z_k'' \right)\right) \\
        =&\ \exp\left( \frac{1}{2} \sum_{k=1}^N \Log z_k'' + \frac{i}{\pi} \sum_{k=1}^N \frac{1}{2}\left( -\pi if_k'' \Log z_k + \pi i f_k \Log z_k'' \right)\right) 
    \end{align*}
    Recall that $y_k = -\Log z_k + i \pi$ and $\Log(1+e^{y_k}) = \Log(z_k'').$ Thus,
    \begin{align*}
        h(\mathbf{y})
        =&\ \exp\left( \sum_{k=1}^N \left( \frac{i}{2\pi} (-\Log z_k + \pi i) \Log z_k'' + \frac{i}{\pi} \Li \left(e^{-\Log z_k}\right) \right)\right)\\
        =&\ \exp\left( -\sum_{k=1}^N \frac{1}{2} \Log z_k'' + \frac{i}{\pi}\sum_{k=1}^N \left( -\frac{1}{2} \Log z_k \Log z_k'' + \Li \left(e^{-\Log z_k}\right) \right)\right).
    \end{align*}
    Besides, 
    $$\left|\exp\left( \frac{i}{\pi} \left(-\frac{1}{2} (-i\pi \mathbf{f}) \cdot (i\pi \mathbf{f''})\right)  \right) \right| = 1.$$
    Summing up the above equations, we get the desired result.
\end{proof}
\begin{remark}
    The formula of $R$ should be compared with \cite[Equation (5-10) and (5-26)]{DG}. Besides, by \cite[Equation (4-15)]{DG}, at $\mathbf{z} = \mathbf{z}(\xi)$ described in Proposition \ref{critThurscorrespondence}, the expression
    $$\left|\sqrt{\left(\prod_{k=1}^N z_k^{f_k''} z_k''^{1-f_k} \right)}\right| $$
    is actually independent of the choice of flattening.
\end{remark}

\subsection{Asymptotic expansion formula of the Teichm\"{u}ller TQFT partition function}\label{sub:proofs:thm15}

We can now prove the second main theorem, Theorem \ref{mainthmZ}. For the sake of clarity, we recall that with the assumptions and notations used in Theorem \ref{mainthmZ}, 
\begin{itemize}
    \item Theorem \ref{mainthmZ} (1) states that if $X$ is FAMED then
    $$\limsup_{{\hbar}\to 0} 2\pi {\hbar} \log|\mathscr{Z}_{\hbar}(X, \alpha) |
\leq -\sup\{ \Vol(\alpha') \mid \alpha' \in \mathscr{A}_X, 
\mathrm{H}^\R_{X,l}(\alpha')=\mathrm{H}^\R_{X,l}(\alpha)
\},$$
    \item Theorem \ref{mainthmZ} (2) states that if $X$ is FAMED and admits a precise hyperbolic cone structure then
    $$|\mathscr{Z}_{\hbar}(X, \alpha) |
=
\left| \frac{\exp\left( \frac{i}{\pi}R(\mathbf{z})\right)}{\det(\mathcal{A}) \sqrt{ 2 \det\mathbf{B}^{-1}}} \cdot \frac{\exp\Big( -\frac{1}{2\pi {\hbar}}\Vol\left(M; l,\mathrm{H}^\R_{X,l}(\alpha)\right)\Big)}{\sqrt{\pm \tau(M, l, X, \mathbf{z})}}  \Big(1 + O({\hbar})\Big) \right|.$$ 
\end{itemize}

\begin{proof}[Proof of Theorem \ref{mainthmZ}]
First we remark that (3) follows directly from (2). Let us prove (1) and (2).

We assume $X$ is FAMED for the curve $l$.

For any $\alpha = (a_1,b_1,c_1,\dots, a_N,b_N,c_N) \in \mathscr{A}_{X}$, let $\mathscr{Y}_\alpha $ be a multi-contour given by  
$$\mathscr{Y}_\alpha = \prod_{k=1}^N\Big( \RR + i(\pi - a_k)\Big). $$

\underline{Step 1: Re-writing the partition function}

We start by re-writing the partition function with a convenient form, thanks to the FAMED property.

By Proposition \ref{Tpartiexpress2}, since X is FAMED, then for all $\hbar>0$, we have 
\begin{align*}
&\ |\mathcal{Z}_{\hbar}(X,\alpha) | = \frac{1}{|\det \mathcal{A}|}
\Big(\frac{1}{2\pi \sqrt{\hbar}}\Big)^{N}\times \notag\\
&\   \left| 
\int_{\mathscr{Y}_\alpha} 
e^{\frac{1}{2\pi\hbar}\left(i \mathbf{y}^{\!\top} Q \mathbf{y} + \mathbf{y}^{\!\top} (\mathbf{B}^{-1}(\boldsymbol \nu +\boldsymbol{u}) - 
\mathscr{G}
\boldsymbol{\pi})
- i\sum_{k=1}^N \left(\frac{\varepsilon(T_k)+1}{4}\right) y_k^2\right)}
\prod_{k=1}^N \Phi_\B\left(\frac{y_k}{2\pi \sqrt{\hbar}}\right) d\mathbf{y} \right|.
\end{align*}

Recall that 
\begin{align*}
S\left (\mathbf{y}; \mathrm{H}^\R_{X,l}(\alpha)\right ) 
= i \mathbf{y}^{\!\top}
Q \mathbf{y}
 + \mathbf{y}^{\!\top} (\mathbf{B}^{-1}(\boldsymbol \nu + \boldsymbol{u}) - 
\mathscr{G}
\boldsymbol{\pi}) - i\sum_{k=1}^N \left(\frac{\varepsilon(T_k)+1}{4}\right) y_k^2
 - i \sum_{k=1}^N \mathrm{Li}_2\left(-e^{y_k}\right),
\end{align*}
where $\boldsymbol u
= (0,\dots,0,
\mathrm{H}^\R_{X,l}(\alpha))^{\!\top}$.

Thus we have
$$
|\mathscr{Z}_{\hbar}(X, \alpha) |
= 
\left|\frac{1}{\det \mathcal{A}} \Big(\frac{1}{2\pi \sqrt{\hbar}}\Big)^{N} 
\int_{\mathscr{Y}_\alpha} 
e^{\frac{1}{2\pi \hbar} S(\mathbf{y};\mathrm{H}^\R_{X,l}(\alpha))} 
e^{
\sum_{k=1}^N \Log \left( \Phi_\B\left(\frac{y_k}{2\pi \sqrt{\hbar}}\right) \right)
-\frac{-i}{2\pi\hbar} \mathrm{Li}_2\left(-e^{y_k}\right)
}
d\mathbf{y} 
\right|.
$$

Recall that
\begin{align*}
&\left | 
\exp \left (
\sum_{k=1}^N \Log \left( \Phi_\B\left(\frac{y_k}{2\pi \sqrt{\hbar}}\right) \right)
-\frac{-i}{2\pi\hbar} \mathrm{Li}_2\left(-e^{y_k}\right)
\right) \right | \\
&= \exp \left ( \Re \left (
\sum_{k=1}^N \Log \left( \Phi_\B\left(\frac{y_k}{2\pi \sqrt{\hbar}}\right) \right)
-\frac{-i}{2\pi\hbar} \mathrm{Li}_2\left(-e^{y_k}\right)
\right ) \right )\\
&= \exp \left ( \sum_{k=1}^N  \Re \left (
\Log \left( \Phi_\B\left(\frac{y_k}{2\pi \sqrt{\hbar}}\right) \right)
-\frac{-i}{2\pi\hbar} \mathrm{Li}_2\left(-e^{y_k}\right)
\right ) \right )
\end{align*}
and thus, by Proposition \ref{prop:quant:dilog:uniform} (4), we have for all $\B \in (0,\sqrt{\delta})$,
$$
e^{-N(C_\delta +(C/\delta+C')\B^2)}
\leqslant
\left | 
\exp \left (
\sum_{k=1}^N \Log \left( \Phi_\B\left(\frac{y_k}{2\pi \sqrt{\hbar}}\right) \right)
-\frac{-i}{2\pi\hbar} \mathrm{Li}_2\left(-e^{y_k}\right)
\right) \right |
\leqslant
e^{N(C_\delta +(C/\delta+C')\B^2)}
$$
where $\delta= \pi-\max_{k=1, \ldots,N}(|\Im(y_k)|)>0$ and $C,C',C_\delta>0$ as in Proposition \ref{prop:quant:dilog:uniform}.

\underline{Step 2: Proving (1)}

To prove (1), i.e. to prove that
$$\limsup_{{\hbar}\to 0} 2\pi {\hbar} \log|\mathscr{Z}_{\hbar}(X, \alpha) |
\leq -\sup\{ \Vol(\alpha') \mid \alpha' \in \mathscr{A}_X, 
\mathrm{H}^\R_{X,l}(\alpha')=\mathrm{H}^\R_{X,l}(\alpha)
\},$$
it suffices to prove that for every $\alpha' \in \mathcal{A}_X$ such that $\mathrm{H}^\R_{X,l}(\alpha')=\mathrm{H}^\R_{X,l}(\alpha)$, we have:
$$\limsup_{{\hbar}\to 0} 2\pi {\hbar} \log|\mathscr{Z}_{\hbar}(X, \alpha) |
\leq - \Vol(\alpha').$$

Let thus $\alpha' \in \mathcal{A}_X$ such that $\mathrm{H}^\R_{X,l}(\alpha')=\mathrm{H}^\R_{X,l}(\alpha)$. It now suffices to prove that 
$$
|\mathscr{Z}_{\hbar}(X, \alpha)| = \exp\left(\frac{
- \Vol(\alpha')
}{2\pi\hbar}
\right)
 O_{\hbar \to 0^+}(\hbar^{-N/2}).
$$
Let us prove this.

It follows from Proposition \ref{Tpartiexpress2} that $|\mathscr{Z}_{\hbar}(X, \alpha)|=|\mathscr{Z}_{\hbar}(X, \alpha')|$.

Following the discussion at the beginning of the proof applied to $\alpha'$ instead of $\alpha$, we have 
\begin{align}
    |\mathscr{Z}_{\hbar}(X, \alpha') |
&= 
\left|\frac{1}{\det \mathcal{A}} \Big(\frac{1}{2\pi \sqrt{\hbar}}\Big)^{N} 
\int_{\mathscr{Y}_{\alpha'}} 
e^{\frac{1}{2\pi \hbar} S(\mathbf{y};\mathrm{H}^\R_{X,l}(\alpha))} 
e^{
\sum_{k=1}^N \Log \left( \Phi_\B\left(\frac{y_k}{2\pi \sqrt{\hbar}}\right) \right)
-\frac{-i}{2\pi\hbar} \mathrm{Li}_2\left(-e^{y_k}\right)
}
d\mathbf{y} 
\right| \notag\\
&\leqslant
\left|\frac{1}{\det \mathcal{A}} \Big(\frac{1}{2\pi \sqrt{\hbar}}\Big)^{N} \right| \cdot 
\int_{\mathscr{Y}_{\alpha'}} 
\left|
e^{\frac{1}{2\pi \hbar} S(\mathbf{y};\mathrm{H}^\R_{X,l}(\alpha))} \right| \cdot 
\left |
e^{
\sum_{k=1}^N \Log \left( \Phi_\B\left(\frac{y_k}{2\pi \sqrt{\hbar}}\right) \right)
-\frac{-i}{2\pi\hbar} \mathrm{Li}_2\left(-e^{y_k}\right)
} \right |
d\mathbf{y} \notag
\\
&=
\left|\frac{1}{\det \mathcal{A}} \Big(\frac{1}{2\pi \sqrt{\hbar}}\Big)^{N} \right| \cdot 
\int_{\mathscr{Y}_{\alpha'}} 
e^{\frac{1}{2\pi \hbar} \Re S(\mathbf{y};\mathrm{H}^\R_{X,l}(\alpha))} 
e^{
\sum_{k=1}^N 
\Re \left (
\Log \left( \Phi_\B\left(\frac{y_k}{2\pi \sqrt{\hbar}}\right) \right)
-\frac{-i}{2\pi\hbar} \mathrm{Li}_2\left(-e^{y_k}\right)
\right )
} 
d\mathbf{y} \notag
\\
&\leqslant
\left|\frac{1}{\det \mathcal{A}} \Big(\frac{1}{2\pi \sqrt{\hbar}}\Big)^{N} \right| \cdot 
e^{N(C_\delta +(C/\delta+C')\B^2)} \cdot 
\int_{\mathscr{Y}_{\alpha'}} 
e^{\frac{1}{2\pi \hbar} \Re S(\mathbf{y};\mathrm{H}^\R_{X,l}(\alpha))} 
d\mathbf{y},\label{step2.1}
\end{align}
where the first equality uses that $\mathrm{H}^\R_{X,l}(\alpha')=\mathrm{H}^\R_{X,l}(\alpha)$ and the last inequality follows from Proposition \ref{prop:quant:dilog:uniform} (4) as recalled above.

Let $\mathbf{y_\mathbf{c}(\alpha')}$ be the critical point of  $\Re  S$ in Lemma \ref{critThurscorrespondence} that corresponds to the angle structure $\alpha'$. On the contour $\mathscr{Y}_{\alpha'}$, by Propositions \ref{dilogVolgen} and \ref{concavSgen}, the function 
$\mathbf{y} \mapsto \Re S(\mathbf{y};\mathrm{H}^\R_{X,l}(\alpha))$ 
attains its maximum at $\mathbf{y_c(\alpha')}$ with maximum value $-\Vol(\alpha')$. We let $r_0>0$ and $\Gamma_{\alpha'} = \{\mathbf{y} \in \mathscr{Y}_{\alpha'} \mid || \mathbf{y} - \mathbf{y_c(\alpha')} || \leq r_0\}$ be a 
ball of real dimension $N$
inside $\mathscr{Y}_{\alpha'}$ centered at the critical point $\mathbf{y_c(\alpha')}$. We split the integral into two parts: one over the compact part $\Gamma_{\alpha'}$ and another one over 
the remaining part
$\mathscr{Y}_{\alpha'} \setminus \Gamma_{\alpha'}$
of the multi-contour. Similarly to \cite[Lemma 7.10]{BAGPN}, there exists constants $A,B>0$ such that whenever $\hbar < A$,
\begin{align}\label{step2.2}
\int_{\mathscr{Y}_{\alpha'}\setminus \Gamma_{\alpha'}} 
e^{\frac{1}{2\pi \hbar} \Re S(\mathbf{y};\mathrm{H}^\R_{X,l}(\alpha))} 
d\mathbf{y}
< Be^{\frac{1}{2\pi\hbar} M},
\end{align}
where $M=\max_{\mathbf{y} \in \partial \Gamma_{\alpha'}}\{\Re S(\mathbf{y};\mathrm{H}^\R_{X,l}(\alpha))\}< - \Vol(\alpha').$ Besides, on the compact part, 
\begin{align}\label{step2.3}
\int_{\Gamma_{\alpha'}} 
e^{\frac{1}{2\pi \hbar} \Re S(\mathbf{y};\mathrm{H}^\R_{X,l}(\alpha))} 
d\mathbf{y} \ 
 \leqslant  \ B'e^{\frac{1}{2\pi\hbar} (-\Vol(\alpha'))},
\end{align}
where $B'$ is the volume of $\Gamma_{\alpha'}$.
From (\ref{step2.1}), (\ref{step2.2}) and (\ref{step2.3}), 
$$
 |\mathscr{Z}_{\hbar}(X, \alpha') | \leqslant 
 \left|\frac{1}{\det \mathcal{A}} \Big(\frac{1}{2\pi \sqrt{\hbar}}\Big)^{N} \right| \cdot 
e^{N(C_\delta +(C/\delta+C')\B^2)} \cdot (B+B')\cdot
e^{\frac{1}{2\pi \hbar} (-\Vol(\alpha'))}.
$$
We conclude that
$$
|\mathscr{Z}_{\hbar}(X, \alpha) |=
|\mathscr{Z}_{\hbar}(X, \alpha') | =
 e^{\frac{1}{2\pi \hbar} (-\Vol(\alpha'))}
 O_{\hbar \to 0^+}(\hbar^{-N/2}).
$$

As explained previously, this allows us to prove (1).

\underline{Step 3: Reducing the proof of (2) to a compact integral}

Now we further assume that there exists shape parameters $\mathbf{z}$ with positive imaginary parts that satisfy Equation (\ref{equ}) with $\xi = i\mathrm{H}^\R_{X,l}(\alpha)$.
Let $\alpha_{geo}$ be the corresponding angle structure.
Let us prove that
$$|\mathscr{Z}_{\hbar}(X, \alpha) |
=
\left| \frac{\exp\left( \frac{i}{\pi}R(\mathbf{z})\right)}{\det(\mathcal{A}) \sqrt{ 2 \det\mathbf{B}^{-1}}} \cdot \frac{\exp\Big( -\frac{1}{2\pi {\hbar}}\Vol\left(M; l,\mathrm{H}^\R_{X,l}(\alpha)\right)\Big)}{\sqrt{\pm \tau(M, l, X, \mathbf{z})}}  \Big(1 + O({\hbar})\Big) \right|.$$

From Step 1, we have
\begin{align*}
|\mathscr{Z}_{\hbar}(X, \alpha) |
&=|\mathscr{Z}_{\hbar}(X, \alpha_{geo}) |\\
&= 
\left|\frac{1}{\det \mathcal{A}} \Big(\frac{1}{2\pi \sqrt{\hbar}}\Big)^{N} 
\int_{\mathscr{Y}_{\alpha_{geo}}} 
e^{\frac{1}{2\pi \hbar} S(\mathbf{y};\mathrm{H}^\R_{X,l}(\alpha))} 
e^{
\sum_{k=1}^N \Log \left( \Phi_\B\left(\frac{y_k}{2\pi \sqrt{\hbar}}\right) \right)
-\frac{-i}{2\pi\hbar} \mathrm{Li}_2\left(-e^{y_k}\right)
}
d\mathbf{y} 
\right|.
\end{align*}

Let $\mathbf{y_\mathbf{c}(\alpha_{geo})}$ be the critical point of  $\Re  S$ in Lemma \ref{critThurscorrespondence} that corresponds to the angle structure $\alpha_{geo}$. We let $r_0>0$ and $\Gamma_{\alpha_{geo}} = \{\mathbf{y} \in \mathscr{Y}_{\alpha_{geo}} \mid || \mathbf{y} - \mathbf{y_c(\alpha_{geo})} || \leq r_0\}$ be a 
ball of real dimension $N$
inside $\mathscr{Y}_{\alpha_{geo}}$ centered at the critical point $\mathbf{y_c(\alpha_{geo})}$. We split the integral into two parts: one over the compact part $\Gamma_{\alpha_{geo}}$ and another one over 
the remaining part
$\mathscr{Y}_{\alpha_{geo}} \setminus \Gamma_{\alpha_{geo}}$
of the multi-contour.

As in Step 2, and similarly to \cite[Lemma 7.10]{BAGPN}, we obtain the following bound on the non-compact part of the contour: there exists constants $A,B>0$ such that whenever $\hbar < A$,
\begin{align*}
& \left|
\int_{\mathscr{Y}_{\alpha_{geo}}\setminus\Gamma_{\alpha_{geo}}} 
e^{\frac{1}{2\pi \hbar} S(\mathbf{y};\mathrm{H}^\R_{X,l}(\alpha))} 
e^{
\sum_{k=1}^N \Log \left( \Phi_\B\left(\frac{y_k}{2\pi \sqrt{\hbar}}\right) \right)
-\frac{-i}{2\pi\hbar} \mathrm{Li}_2\left(-e^{y_k}\right)
}
d\mathbf{y}\right| \\
&\leqslant
e^{N(C_\delta +(C/\delta+C')\B^2)} \cdot 
\int_{\mathscr{Y}_{\alpha_{geo}}\setminus\Gamma_{\alpha_{geo}}} 
e^{\frac{1}{2\pi \hbar} \Re S(\mathbf{y};\mathrm{H}^\R_{X,l}(\alpha))} 
d\mathbf{y}
\leqslant
e^{N(C_\delta +(C/\delta+C')\B^2)} \cdot B \cdot
e^{\frac{1}{2\pi \hbar} M} ,
\end{align*}
where $M =  \max\{\Re S (\mathbf{y}) \mid  \mathbf{y} \in \partial\Gamma_{\alpha_{geo}}\}$. By Propositions \ref{dilogVolgen} and  \ref{concavSgen}, we have 
$$M < -\Vol(\alpha_{geo}) = -\Vol\Big(M;l, \mathrm{H}^\R_{X,l}(\alpha)\Big).$$
Hence, 
$$\left|\frac{1}{\det \mathcal{A}} \Big(\frac{1}{2\pi \sqrt{\hbar}}\Big)^{N} 
\int_{\mathscr{Y}_{\alpha_{geo}}\setminus\Gamma_{\alpha_{geo}}} 
e^{\frac{1}{2\pi \hbar} S(\mathbf{y};\mathrm{H}^\R_{X,l}(\alpha))} 
e^{
\sum_{k=1}^N \Log \left( \Phi_\B\left(\frac{y_k}{2\pi \sqrt{\hbar}}\right) \right)
-\frac{-i}{2\pi\hbar} \mathrm{Li}_2\left(-e^{y_k}\right)
}
d\mathbf{y} 
\right|$$
is exponentially smaller than the desired asymptotics
$$\left| \frac{\exp\left( \frac{i}{\pi}R(\mathbf{z})\right)}{\det(\mathcal{A}) \sqrt{ 2 \det\mathbf{B}^{-1}}} \cdot \frac{\exp\Big( -\frac{1}{2\pi {\hbar}}\Vol\left(M; l,\mathrm{H}^\R_{X,l}(\alpha)\right)\Big)}{\sqrt{\pm \tau(M, l, X, \mathbf{z})}}  \Big(1 + O({\hbar})\Big) \right|$$
when $\hbar \to 0^+$. This implies that we only need to prove that the compact part
$$\left|\frac{1}{\det \mathcal{A}} \Big(\frac{1}{2\pi \sqrt{\hbar}}\Big)^{N} 
\int_{\Gamma_{\alpha_{geo}}} 
e^{\frac{1}{2\pi \hbar} S(\mathbf{y};\mathrm{H}^\R_{X,l}(\alpha))} 
e^{
\sum_{k=1}^N \Log \left( \Phi_\B\left(\frac{y_k}{2\pi \sqrt{\hbar}}\right) \right)
-\frac{-i}{2\pi\hbar} \mathrm{Li}_2\left(-e^{y_k}\right)
}
d\mathbf{y} 
\right|$$
has the desired asymptotics.

\underline{Step 4: End of the proof of (2)}

Let us denote $$h( \mathbf{y}) = \exp\left(\sum_{k=1}^N \left(\frac{i}{2\pi}y_k\log(1+e^{y_k}) + \frac{i}{\pi}\Li \left(- e^{y_k}\right) \right) \right).
$$
From Proposition \ref{semihbar}, for 
$\mathbf{y}=(y_1,\ldots,y_N)\in\Gamma_{\alpha_{geo}}$, for $k \in\{1,\ldots,N\}$, we have
$$\Phi_\B\left (\frac{y_k}{2 \pi \sqrt{\hbar}}\right ) = 
\exp\left (
- \frac{i}{2 \pi \hbar}  \Li \left(- e^{y_k}\right) +
\frac{i}{2\pi}y_k\log(1+e^{y_k}) + \frac{i}{\pi}\Li \left(- e^{y_k}\right) 
+ O_{\hbar \to 0^+}(\hbar)
\right ),$$
and thus
$$e^{
\sum_{k=1}^N \Log \left( \Phi_\B\left(\frac{y_k}{2\pi \sqrt{\hbar}}\right) \right)
-\frac{-i}{2\pi\hbar} \mathrm{Li}_2\left(-e^{y_k}\right)
}= h( \mathbf{y}) \exp \left ( O_{\hbar \to 0^+}(\hbar)\right ),
$$
where $O_{\hbar \to 0^+}(\hbar)$ is a function of $\mathbf{y}$ holomorphic in a neighborhood $D$ of $\Gamma_{\alpha_{geo}}$ (a simple Taylor remainder in fact), and since $\Gamma_{\alpha_{geo}}$ is compact this function of $\mathbf{y}$ is uniformly a $O_{\hbar \to 0^+}(\hbar)$.

Now we can apply Proposition \ref{saddle}
to the compact integral 
\begin{align*}
& \int_{\Gamma_{\alpha_{geo}}} 
e^{\frac{1}{2\pi \hbar} S(\mathbf{y};\mathrm{H}^\R_{X,l}(\alpha))} 
e^{
\sum_{k=1}^N \Log \left( \Phi_\B\left(\frac{y_k}{2\pi \sqrt{\hbar}}\right) \right)
-\frac{-i}{2\pi\hbar} \mathrm{Li}_2\left(-e^{y_k}\right)
}
d\mathbf{y} \\
&= \int_{\Gamma_{\alpha_{geo}}} 
h(\mathbf{y})
e^{\frac{1}{2\pi \hbar} S(\mathbf{y};\mathrm{H}^\R_{X,l}(\alpha))} 
e^{O(\hbar)
}
d\mathbf{y},
\end{align*}
    where $h(\mathbf{y})$ plays the role of $g(\mathbf{z})$,
$S(\mathbf{y})$ plays the role of $f(\mathbf{z})$, and
$\frac{1}{2\pi \sqrt{\hbar}}\left (S(\mathbf{y})+2\pi \hbar O(\hbar)\right )$ plays the role of $r f_r(\mathbf{z})$. Remark that condition (4) in Proposition \ref{saddle} is satisfied thanks to the above uniformity argument coming from the compactness of the contour piece. We thus obtain
\begin{align*}
&\left |\int_{\Gamma_{\alpha_{geo}}} 
e^{\frac{1}{2\pi \hbar} S(\mathbf{y};\mathrm{H}^\R_{X,l}(\alpha))} 
e^{
\sum_{k=1}^N \Log \left( \Phi_\B\left(\frac{y_k}{2\pi \sqrt{\hbar}}\right) \right)
-\frac{-i}{2\pi\hbar} \mathrm{Li}_2\left(-e^{y_k}\right)
}
d\mathbf{y} \right |\\
&= 
\left|\frac{1}{\det \mathcal{A}} \left(\frac{1}{2\pi \sqrt{\hbar}}\right)^{N}  \big(2\pi \hbar\big)^{\frac{N}{2}} \frac{h(\mathbf{y_c})}{\sqrt{(-1)^{N}\frac{\det(\Hess S)(\mathbf{y_c};\mathrm{H}^\R_{X,l}(\alpha))}{(2\pi)^{N}}}} e^{\frac{1}{2\pi \hbar}S(\mathbf{y_c};\mathrm{H}^\R_{X,l}(\alpha))} \Big(1 + O(\hbar)\Big)\right|,
\end{align*}
where $\mathbf{y_c}$ is the critical point of $S$ described in Proposition \ref{critThurscorrespondence}.

The result then follows from Propositions {\ref{dilogVolgen}}, \ref{Hesstotor} and \ref{hvalue}, as we obtain the desired asymptotics
$$\left| \frac{\exp\left( \frac{i}{\pi}R(\mathbf{z})\right)}{\det(\mathcal{A}) \sqrt{ 2 \det\mathbf{B}^{-1}}} \cdot \frac{\exp\Big( -\frac{1}{2\pi {\hbar}}\Vol\left(M; l,\mathrm{H}^\R_{X,l}(\alpha)\right)\Big)}{\sqrt{\pm \tau(M, l, X, \mathbf{z})}}  \Big(1 + O({\hbar})\Big) \right|.$$
\end{proof}

To conclude this section, let us prove Corollary \ref{coroCassonconj2}.

\begin{proof}[Proof of Corollary \ref{coroCassonconj2}]
We use proof by contradiction. Let us assume that
$X$ is FAMED for $l$, that  Conjecture \ref{conj:expansion:TQFT} is true for $(X,\alpha)$ and that
$$ \sup\{ \Vol(\alpha') \mid \alpha' \in \mathscr{A}_{X}^l (\mathrm{H}^\R_{X,l}(\alpha))\} > \Vol\Big(M;l, \mathrm{H}^\R_{X,l}(\alpha)\Big).$$
By Theorem \ref{mainthmZ}(1), since $X$ is FAMED for $l$, we have
$$
\limsup_{\hbar\to 0} 2\pi \hbar\log|\mathscr{Z}_{\hbar}(X, \alpha) |
< -\Vol\Big(M;l, \mathrm{H}^\R_{X,l}(\alpha)\Big).
$$
However, since $(X,\alpha)$ satisfies Conjecture \ref{conj:expansion:TQFT}, we thus have
$$
\lim_{\hbar\to 0} 2\pi \hbar \log|\mathscr{Z}_{\hbar}(X, \alpha) |
= \limsup_{\hbar\to 0} 2\pi \hbar \log|\mathscr{Z}_{\hbar}(X, \alpha) |
= -\Vol\Big(M;l, \mathrm{H}^\R_{X,l}(\alpha)\Big),
$$
which leads to a contradiction.
\end{proof}

\section{Asymptotics of Jones functions}\label{sec:Jones}
\subsection{Existence of the Jones functions}\label{subsec:existJ}
From Proposition \ref{Tpartiexpress2}, we can see that in the partition function, the information of the angle structure $\alpha\in \mathscr{A}_X$ is contained in the expression
\begin{align*}
\mathbf{y}^{\!\top}
\left(
\mathbf{B}^{-1}\boldsymbol{u}
\right)
= \left( (\mathbf{B}^{-1})^{\!\top}\mathbf{y} \right) \cdot \boldsymbol{u}
= \frac{\tilde{\mathrm{x}} \mathrm{H}^\R_{X,l}(\alpha)}{2},
\end{align*}
where $\tilde{\mathrm{x}}$ is a linear combination of $y$'s given by the last entry of
$$2 (\mathbf{B}^{-1})^{\!\top}\mathbf{y}. $$
Let $m \in \pi_1(\partial M)$ such that the algebraic intersection number of $l$ and $m$ is 1. Assume that the Neumann-Zagier data of $m$ is given by
$$
(c_1, \dots ,c_n) \cdot \BLog\mathbf{z} + (d_1,\dots, d_n) \BLog\mathbf{z}'' = 
\mathrm{H}^\C_{X,m}(\mathbf{z})
+ i \pi \nu_m,
$$
where $c_i, d_i, \nu_m \in \ZZ$.
Recall that by \cite{NZ} (see also \cite[Equation (4-13)]{DG}), we can extend the matrices $\mathbf{A},\mathbf{B}$ into a symplectic matrix
$$
\begin{pmatrix}
\mathbf{A} & \mathbf{B} \\
\mathbf{C} & \mathbf{D}
\end{pmatrix}
$$
such that
\begin{align}
\mathbf{A} \mathbf{D}^{\!\top}  -   \mathbf{B}\mathbf{C}^{\!\top} &= \rm{Id}_N, \label{NZsymp1}\\
\text{the last coefficient of $\mathbf{C} \BLog\mathbf{z} + \mathbf{D} \BLog\mathbf{z}''$} &\text{ is } (\mathrm{H}^\C_{X,m}(\mathbf{z}) + i \pi \nu_m)/2,\label{NZsymp2}
\end{align}
where $\rm{Id}_N$ in (\ref{NZsymp1}) is the $N\times N$ identity matrix.
Furthermore, by \cite[Lemma A.2]{DG}, $\mathscr{G} = \mathbf{B}^{-1}\mathbf{A}$ is symmetric.
The following lemma provides a geometric interpretation of this quantity.
\begin{lemma}\label{xmeaning}
Under the correspondence in Proposition \ref{critThurscorrespondence}, at the critical point of $\tilde S$, we have
$$
\tilde{\mathrm{x}} = \mathrm{H}^\C_{X,m}(\mathbf{z})+ n_1 \mathrm{H}^\C_{X,l}(\mathbf{z}) + i n_2 \pi
$$
for some $n_1,n_2 \in \QQ$. In particular, at the critical point of $S$, we have
$$
\tilde{\mathrm{x}} = \mathrm{H}^\C_{X,m}(\mathbf{z})+ n_1 i\mathrm{H}^\R_{X,l}(\alpha) + i n_2 \pi.
$$
\end{lemma}
\begin{proof}
Under the condition that $\mathbf{B}$ is invertible, by (\ref{NZsymp1}), we have $\mathbf{B}^{-1} = (\mathbf{B}^{-1}\mathbf{A})\mathbf{D}^{\!\top} - \mathbf{C}^{\!\top}$. Besides, recall from \cite[Lemma A.2]{DG} that $\mathscr{G}=\mathbf{B}^{-1}\mathbf{A}$ is symmetric. As a result, at the critical point of $S$,
\begin{align}\label{lcompute}
(\mathbf{B}^{-1})^{\!\top} 
\mathbf{y}
&= (\mathbf{B}^{-1})^{\!\top} (-\BLog (\mathbf{z}) + i \boldsymbol{\pi}) \notag\\
&= (\mathbf{D}(\mathbf{B}^{-1}\mathbf{A}) - \mathbf{C}) 
\cdot (-\BLog(\mathbf{z}) ) + (\mathbf{B}^{-1})^{\!\top} i \boldsymbol{\pi} \notag\\
&= -\left( \mathbf{D}(\mathbf{B}^{-1}\mathbf{A} \BLog(\mathbf{z}) )  - \mathbf{C}\BLog(\mathbf{z})  \right) 
+ (\mathbf{B}^{-1})^{\!\top} i \boldsymbol{\pi} \notag \\
&=  -\left( - \mathbf{D} \BLog(\mathbf{z}'')  - \mathbf{C}\BLog(\mathbf{z}) 
+ \mathbf{D}\mathbf{B}^{-1} i(\boldsymbol \nu -i\boldsymbol{\tilde u}) \right) 
+ (\mathbf{B}^{-1})^{\!\top} i \boldsymbol{\pi}.
\end{align}
By (\ref{NZsymp2}), we have
$$
\tilde{\mathrm{x}} = \mathrm{H}^\C_{X,m}(\mathbf{z}) + n_1 \mathrm{H}^\C_{X,l}(\mathbf{z}) + i n_2 \pi
$$
for some $n_1,n_2 \in \QQ$. 
\end{proof}

The following proposition will cover Theorem \ref{thm:Jones:FAMED} (1).

\begin{proposition}\label{pfexistJ} 
Assume that $X$ is FAMED for $l$.
Then there exists a Jones function $\mathfrak{J}_X\colon \R_{>0} \times \mathcal{W} \to \C$ (where $\mathcal{W}\subset \C$ is an open horizontal band) that is independent of $\alpha\in\mathscr{A}_X$ such that 
\begin{align*}
&|\mathscr{Z}_{\hbar}(X, \alpha)| 
=\left| 
\int_{\RR + i\mu_X(\alpha) } 
\mathfrak{J}_X(\hbar,\mathrm{x})
e^{\frac{\mathrm{x} \lambda_X(\alpha)}{4\pi \hbar}} d\mathrm{x} \right|,
\end{align*}
where $\lambda_X(\alpha):=\mathrm{H}^\R_{X,l}(\alpha)$, $n_1 \in \Q$ is given in Lemma \ref{xmeaning} and
$\mu_X(\alpha)=\mathrm{H}^\R_{X,m+n_1 l}(\alpha)$ is the angular holonomy of the curve $m+n_1 l$. 
\end{proposition}
\begin{proof} 
Let us denote $\lambda_X(\alpha):=\mathrm{H}^\R_{X,l}(\alpha)$.
From Proposition \ref{Tpartiexpress2}, we can write 
\begin{align*}
&\ |\mathcal{Z}_{\hbar}(X,\alpha) | = \frac{1}{|\det \mathcal{A}|}
\Big(\frac{1}{2\pi \sqrt{\hbar}}\Big)^{N}\times \notag\\
&\   \left| 
\int_{\mathscr{Y}_\alpha} \left(
e^{\frac{1}{2\pi\hbar}\left(i \mathbf{y}^{\!\top} Q \mathbf{y} + \mathbf{y}^{\!\top} (\mathbf{B}^{-1}\boldsymbol \nu  - 
\mathscr{G}
\boldsymbol{\pi})
- i\sum_{k=1}^N \left(\frac{\varepsilon(T_k)+1}{4}\right) y_k^2\right)}
\prod_{k=1}^N \Phi_\B\left(\frac{y_k}{2\pi \sqrt{\hbar}}\right) \right)
e^{\frac{\mathrm{\tilde x}\lambda_X(\alpha)}{4\pi\hbar}}d\mathbf{y} \right|.
\end{align*}
where $\tilde{\mathrm{x}}$ is the last entry of 
$$2 (\mathbf{B}^{-1})^{\!\top}\mathbf{y}. $$
Since $\lambda_X(\alpha)\in \RR$, we can also write
\begin{align*}
&\ |\mathcal{Z}_{\hbar}(X,\alpha) | = \frac{1}{|\det \mathcal{A}|}
\Big(\frac{1}{2\pi \sqrt{\hbar}}\Big)^{N}\times \notag\\
&\   \left| 
\int_{\mathscr{Y}_\alpha} \left(
e^{\frac{1}{2\pi\hbar}\left(i \mathbf{y}^{\!\top} Q \mathbf{y} + \mathbf{y}^{\!\top} (\mathbf{B}^{-1}\boldsymbol \nu  - 
\mathscr{G}
\boldsymbol{\pi})
- i\sum_{k=1}^N \left(\frac{\varepsilon(T_k)+1}{4}\right) y_k^2\right)}
\prod_{k=1}^N \Phi_\B\left(\frac{y_k}{2\pi \sqrt{\hbar}}\right) \right)
e^{\frac{\mathrm{x}\lambda_X(\alpha)}{4\pi\hbar}}d\mathbf{y} \right|.
\end{align*}
where $\mathrm{x} = \tilde{\mathrm{x}} - in_2\pi$ and $n_2$ is the rational number defined in Lemma \ref{xmeaning}. 
Since $\mathbf{B}$ is a matrix with integer coefficients, we can let $\mathrm{x} = \sum_{k=1}^N q_k y_k - i n_2 \pi$, where $q_k \in \QQ$. Without loss of generality, assume that $q_1 \neq 0$. Consider the affine isomorphism $L: \CC^N \to \CC^N$  defined by 
$$L: (y_1,y_2,\dots, y_N) \mapsto (\mathrm{x}, y_2,\dots, y_N) = \left( \sum_{k=1}^N q_k y_k - in_2\pi , y_2,\dots, y_N \right).$$
Observe that the jacobian of $L$ is $q_1 \neq 0$.
Let $\mathscr{Y}_\alpha'$ be the multi-contour contour defined by
$$\mathscr{Y}_\alpha' = \prod_{k=2}^N \Big(\RR + i(\pi - a_k)\Big) . $$
Then we have
\begin{align}\label{ZtoJconv}
&|\mathscr{Z}_{\hbar}(X, \alpha)| 
= \left| 
\int_{\RR + i\mu_X(\alpha)} 
\mathfrak{J}_X(\mathrm{x},\hbar)
e^{\frac{\mathrm{x} \lambda_X(\alpha)}{4 \pi \hbar}} d\mathrm{x} \right| ,
\end{align}
where $\mu_X(\alpha)=\mathrm{H}^\R_{X,m+n_1 l}(\alpha)$,
\begin{align}\label{ZtoJconv2}
&\mathfrak{J}_X(\mathrm{x},\hbar)= \frac{q_1}{\det \mathcal{A}}\Big(\frac{1}{2\pi \sqrt{\hbar}}\Big)^{N} \notag\\
&\times \int_{\mathcal{Y}_\alpha'} \left( e^{\frac{1}{2\pi\hbar}\left(i \mathbf{y}^{\!\top} Q \mathbf{y} + \mathbf{y}^{\!\top} (\mathbf{B}^{-1}\boldsymbol \nu  - 
\mathscr{G}
\boldsymbol{\pi})
- i\sum_{k=1}^N \left(\frac{\varepsilon(T_k)+1}{4}\right) y_k^2\right)}
\prod_{k=1}^N \Phi_\B\left(\frac{y_k}{2\pi \sqrt{\hbar}}\right) \right)
 d y_2 \dots dy_N
\end{align}
and
$$ y_1 = \frac{\mathrm{x} -  \sum_{k=2}^N q_k y_k + in_2\pi }{q_1}. $$ 
Note that (\ref{ZtoJconv}) holds for any $\alpha \in \mathcal{A}_X$. In particular, the integrand in (\ref{ZtoJconv2}) decays exponentially at infinity and the integral converges absolutely. Similar to the proof of Proposition \ref{Tpartiexpress2}, we can deform the integration multi-contour to any $\mathscr{Y}_{\alpha'}' $ for any $\alpha' \in \mathcal{A}_X$. This shows that $\mathfrak{J}_X$ is independent on the choice of angle structure. This completes the proof.
\end{proof}

\subsection{Potential function and its properties}\label{subsect:potentialpropJ}
We still denote $\lambda_X(\alpha):=\mathrm{H}^\R_{X,l}(\alpha)$.
Note that under the isomorphism $L$, the potential function $\tilde S$ can be written as
\begin{align*}
\tilde{S}^J(\mathrm{x}, y_2,\dots, y_N ; \xi ) 
&:= \tilde{S}\left (L^{-1}(\mathrm{x}, y_2,\dots, y_N); \xi\right) \\
&\ = J(\mathrm{x},y_2,\dots, y_N) - \frac{i\xi\mathrm{x}}{2}, 
\end{align*}
where 
\begin{align*}
&J(\mathrm{x},y_2,\dots, y_N) 
:=\ i \mathbf{y}^{\!\top}
Q \mathbf{y}
 + \mathbf{y}^{\!\top} (\mathbf{B}^{-1}\boldsymbol \nu  - 
\mathscr{G}
\boldsymbol{\pi}) -  i\sum_{k=1}^N \left(\frac{\varepsilon(T_k)+1}{4}\right) y_k^2
 - i \sum_{k=1}^N \mathrm{Li}_2\left(-e^{y_k}\right)
\end{align*}
is independent on $\xi$.

As we will now connect the potential function $J$ to the Neumann-Zagier potential function, we refer the reader to Section \ref{NZpotentintro} for a reminder on the local complex variables $w^{loc}_\gamma$ which represent the complex holonomies $\mathrm{H}^\C_{X,\gamma}(\mathbf{z})$ associated to a curve $\gamma$.

Here we consider the pair of curves $(l,m)$, the associated local variables $w^{loc}_l, w^{loc}_m$ and the transition map $\Psi^{loc}_{l,m}$ such that
$w^{loc}_{m}
=\Psi^{loc}_{l,m}
(w^{loc}_{l})$.

\begin{lemma}\label{xbiholo}
The map
$$
\left(\Psi^{loc}_{l,m}+n_1\id\right): w^{loc}_{l} \mapsto w^{loc}_{m} + n_1 w^{loc}_{l}
$$
is a local biholomorphism at $0$. Moreover, this map sends $0$ to $0$.
\end{lemma}
\begin{proof}
By \cite[Lemma 4.1 (a)]{NZ} at the complete hyperbolic structure, we have $\frac{\partial w^{loc}_{m}}{\partial w^{loc}_{l}} \not \in \RR$. Thus,
$$
\frac{\partial {(w^{loc}_{m}+ n_1 w^{loc}_{l})}}{\partial w^{loc}_{l}} = \frac{\partial  w^{loc}_{m}}{\partial w^{loc}_{l}} + n_1 \neq 0 .$$ 
The first claim follows from the inverse function theorem. The second claim follows from the fact that at the complete hyperbolic structure, we have $w^{loc}_{l}=w^{loc}_{m} = 0$.
\end{proof}
From Lemma \ref{xmeaning}, we can interpret $\mathrm{x} = \mathrm{x}' - in_2\pi$ as  
$w^{loc}_{m}+n_1 w^{loc}_{l}$ or
$\mathrm{H}^\C_{X,m}(\mathbf{z})+n_1 \mathrm{H}^\C_{X,l}(\mathbf{z})$.
We study the properties of the potential function $J$ for $\mathrm{x}$ sufficiently close to $0$. Let $\mathscr{U}'$ be the multi-dimensional horizontal band defined by
$$
\mathscr{U}' = \prod_{k=2}^N (\RR + i (0,\pi)).
$$
\begin{proposition}\label{concavSxgen}
For each $\mathrm{x}$ (not necessarily close to $0$), the real part of $J(\mathrm{x}, y_2,\dots, y_N)$ is strictly concave
on 
every multi-dimensional real horizontal piece for the real parts of the variables $y_i$. 
\end{proposition}
\begin{proof}
The proof is similar to that of Proposition \ref{concavSgen}.

\end{proof}

From Lemma \ref{xbiholo}, given $\mathrm{x}= w^{loc}_m + n_1 w^{loc}_l$ sufficiently close to $0$, we have a corresponding value 
$$w^{loc}_l(\mathrm{x}):=\left(\Psi^{loc}_{l,m}+n_1\id\right)^{-1}(\mathrm{x}).$$
Observe that $w^{loc}_l$, the complex holonomy of the curve $l$, may not be purely imaginary as in the previous section.

Let 
$$\mathbf{y_c}(\mathrm{x}) = (y_{c,1}(\mathrm{x}), y_{c,2}(\mathrm{x}), \dots, y_{c,N}(\mathrm{x}))$$ 
be the corresponding critical point of $\tilde S(y_1,\dots, y_N ; w^{loc}_l(\mathrm{x}) )$ with respect to $y_1,\dots, y_N$.

Observe that when $w^{loc}_l(\mathrm{x})$ is purely imaginary, we retrieve the usual angular holonomy $\frac{1}{i} \mathrm{H}^\C_{X,l}(\mathbf{z}) =\mathrm{H}^\R_{X,l}(\alpha)$.

\begin{proposition}\label{JandNZ}
For any $\mathrm{x}$ sufficiently close to 0, the point $(y_{c,2}(\mathrm{x}), \dots, y_{c,N}(\mathrm{x}))$ is a critical point of $J(\mathrm{x}, y_2,\dots, y_N)$ with respect to $y_2,\dots, y_N$. Furthermore, 
the critical value $J(\mathrm{x}, y_{c,2}(\mathrm{x}), \dots, y_{c,N}(\mathrm{x}))$ satisfies
\begin{align*}
\frac{\partial}{\partial \mathrm{x}} J(\mathrm{x}, y_{c,2}(\mathrm{x}), \dots, y_{c,N}(\mathrm{x})) = \frac{i w^{loc}_l(\mathrm{x})}{2} \qquad\text{and}\qquad \Re J(0, y_{c,2}(0), \dots, y_{c,N}(0)) = - \Vol(M).
\end{align*}
In particular, we have $ J(\mathrm{x}, y_{c,2}(\mathrm{x}), \dots, y_{c,N}(\mathrm{x})) = i\phi_{m,l}(\mathrm{x}) + C$ for some imaginary constant $C\in \CC$.
\end{proposition}

\begin{proof}
Note that, since $S^J=S \circ L^{-1}$ and $L$ is an affine isomorphism, we have 
$$\nabla_{\mathbf{y}} S(y_1,\dots, y_N) = ( S_{y_1} ,\dots, S_{y_N}) = 0$$
 if and only if
$$  \nabla S^J(\mathrm{x}, y_2, \dots, y_N ) = ( J_{\mathrm{x}} -i\xi/2, J_{y_2},\dots, J_{y_N}) = 0.$$
This implies that 
$$ J_{y_2} (\mathrm{x}, y_{c,2}(\mathrm{x}), \dots, y_{c,N}(\mathrm{x})) = \dots= 
J_{y_N}(\mathrm{x}, y_{c,2}(\mathrm{x}), \dots, y_{c,N}(\mathrm{x})) = 0$$
for all $\mathrm{x}$. Moreover, 
$$
\frac{\partial}{\partial \mathrm{x}} \Big( J(\mathrm{x}, y_{c,2}(\mathrm{x}), \dots, y_{c,N}(\mathrm{x}) \Big)
= J_{\mathrm{x}} + \sum_{k=2}^N J_{y_k} \cdot y_k'(\mathrm{x})  
= J_{\mathrm{x}} ,
$$
where the last equality follows from the fact that $(y_{c,2}(\mathrm{x}), \dots, y_{c,N}(\mathrm{x}))$ is the critical point of $J$ with respect to $y_2,\dots, y_N$. As a result, we have
$$
\frac{\partial}{\partial \mathrm{x}} \Big(J(\mathrm{x}, y_{c,2}(\mathrm{x}), \dots, y_{c,N}(\mathrm{x})) \Big)
= \frac{i\xi}{2} = \frac{i w_l^{loc}(\mathrm{x})}{2},
$$
where the last equality follows from Proposition \ref{critThurscorrespondence}.
Finally, when $w^{loc}_l=0$, by Lemma \ref{xbiholo}, we have $\mathrm{x} = 0$. Moreover, when $\mathrm{x}=0$, we have 
\begin{align*}
S\left(0 , y_2(0),\dots, y_N(0)\right) 
&= J\left(0 , y_2(0),\dots, y_N(0)\right). 
\end{align*}
Thus, by Proposition \ref{dilogVolgen}, we have
\begin{align}\label{ReJgiveVol}
\mathrm{Re} \left(J\left(0 , y_{c,2}(0), \dots, y_{c,N}(0)\right)  \right)
= -\Vol(M). 
\end{align}
Finally, by (\ref{NZprop}), since 
$$
\frac{\partial}{\partial \mathrm{x}} 
\Big(J(\mathrm{x}, y_{c,2}(\mathrm{x}), \dots, y_{c,N}(\mathrm{x})) - i\phi_{m,l}(\mathrm{x}) \Big)
= \frac{iw_l^{loc}(\mathrm{x})}{2} - \frac{iw_l^{loc}(\mathrm{x})}{2} = 0,
$$
we have $J(\mathrm{x}, y_{c,2}(\mathrm{x}), \dots, y_{c,N}(\mathrm{x})) - i\phi_{m,l}(\mathrm{x})) - i\phi_{m,l}(\mathrm{x}) = C$ for some constant $C$. Note that at $\mathrm{x}=0$, by (\ref{NZprop}) and (\ref{ReJgiveVol}), we have
$\mathrm{Re}(C) = 0$. Thus, $C$ is an imaginary constant.
\end{proof}

Recall that $\mathrm{x} = \sum_{k=1}^N q_k y_k - in_2\pi$ with $q_1 \neq 0$.
The following proposition is an analogue of \cite[Lemma 3.3]{WY} and should be compared with the change of curve formula of the torsion.
\begin{proposition}\label{1loopJ}
For $x$ sufficiently close to $0$,
we have
$$ \det(\Hess_{\mathbf{y}} \tilde S) (\mathbf{y_c}(\mathrm{x})) = \frac{i}{2 (q_1)^2 } \frac{\partial \mathrm{H}^\C_{X,l}(\mathbf{y_c}(\mathrm{x}))}{\partial\mathrm{x}}  \det (\Hess_{{y_2,\ldots,y_N}
} J)(\mathrm{x}, y_{c,2}(\mathrm{x}), \dots, y_{c,N}(\mathrm{x})) . $$
In particular, if $n_1$ in Lemma \ref{xmeaning} is equal to 0, then 
$$ \det(\Hess_{\mathbf{y}} \tilde S) (\mathbf{y_c}(\mathrm{x})) = \frac{i}{2 (q_1)^2 } \frac{\partial w^{loc}_l(\mathrm{x})}{\partial w^{loc}_m}  \det (\Hess_{y_2,\ldots,y_N} J)(\mathrm{x}, y_{c,2}(\mathrm{x}), \dots, y_{c,N}(\mathrm{x})) . $$
\end{proposition}

\begin{proof}
Recall that $\tilde S^J(x,y_2,\dots, y_N) = \tilde S(L(y_1,\dots, y_N))$, where $L: \CC^N \to \CC^N$ is the affine isomorphism with determinant $q_1 \neq 0$ defined by  
$$(y_1,y_2,\dots, y_N) \mapsto (\mathrm{x}, y_2,\dots, y_N) = \left(\sum_{k=1}^N {q_k} y_k -in_2\pi, y_2,\dots, y_N \right).$$
In particular, we have
$$ (\Hess_{\mathrm{x},y_2,\ldots,y_N} \tilde S^J) (\mathrm{x}, y_{c,2}(\mathrm{x}), \dots, y_{c,N}(\mathrm{x}))  
= L^{\!\top} (\Hess_{\mathbf{y}} \tilde S) (\mathbf{y_c}(\mathrm{x})) L $$
and 
\begin{align*}
\det(\Hess_{\mathrm{x},y_2,\ldots,y_N} \tilde S^J) (\mathrm{x}, y_{c,2}(\mathrm{x}), \dots, y_{c,N}(\mathrm{x}))
&= \det(L^{\!\top})  \det(\Hess_{\mathbf{y}} \tilde S) (\mathbf{y_c}(\mathrm{x})) \det(L) \\
&= (q_1)^2\det(\Hess_{\mathbf{y}} \tilde S) (\mathbf{y_c}(\mathrm{x})) .
\end{align*}
Thus, 
$$
\det(\Hess_{\mathbf{y}}  \tilde S) (\mathbf{y}(\mathrm{x})) 
= \frac{1}{(q_1)^2}\det(\Hess_{\mathrm{x},y_2,\ldots,y_N} \tilde S^J) (\mathrm{x}, y_{c,2}(\mathrm{x}), \dots, y_{c,N}(\mathrm{x})) 
$$
and it suffices to compute $\det(\Hess_{\mathrm{x},y_2,\ldots,y_N} \tilde S^J) (\mathrm{x}, y_{c,2}(\mathrm{x}), \dots, y_{c,N}(\mathrm{x}))$. 
We claim that 
\begin{enumerate}
\item[(1)] for $i,j \in \{2,\dots, N\}$, 
\begin{align*}
\frac{\partial^2 \tilde S^J}{\partial y_i \partial y_j} (\mathrm{x}, y_{c,2}(\mathrm{x}), \dots, y_{c,N}(\mathrm{x}))= 
\frac{\partial^2 J}{\partial y_i \partial y_j} 
(\mathrm{x}, y_{c,2}(\mathrm{x}), \dots, y_{c,N}(\mathrm{x})),
\end{align*}
\item[(2)] for $i \in \{2,\dots, N\}$, 
\begin{align*}
\frac{\partial^2 \tilde S^J}{\partial \mathrm{x} \partial y_i} (\mathrm{x}, y_{c,2}(\mathrm{x}), \dots, y_{c,N}(\mathrm{x})) = - 
\sum_{l=2}^N\frac{\partial^2 J}{ \partial y_l\partial y_i}(\mathrm{x}, y_{c,2}(\mathrm{x}), \dots, y_{c,N}(\mathrm{x}) ) \cdot \Bigg( \frac{\partial y_{c,l}(\mathrm{x})}{\partial \mathrm{x}} \Bigg),
\end{align*}
\item[(3)] for the second derivative with respect to $\mathrm{x}$,
\begin{align*}
&\frac{\partial^2 \tilde S^J}{\partial \mathrm{x}^2}
(\mathrm{x}, y_{c,2}(\mathrm{x}), \dots, y_{c,N}(\mathrm{x}))\\
&= \frac{i}{2} \frac{\partial w^{loc}_{l}(\mathrm{x})}{\partial \mathrm{x}} + \sum_{l_1,l_2=2}^N \Bigg(\frac{\partial^2 J}{\partial y_{l_1} \partial y_{l_2}} (\mathrm{x}, y_{c,2}(\mathrm{x}), \dots, y_{c,N}(\mathrm{x})) \cdot \bigg( \frac{\partial y_{c,l_1}(\mathrm{x})}{\partial \mathrm{x}} \bigg)  \bigg( \frac{\partial y_{c,l_2}(\mathrm{x})}{\partial \mathrm{x}} \bigg) \Bigg).
\end{align*}
\end{enumerate}
Assuming these claims, if we write 
$$k_i(\mathrm{x}) := - \frac{\partial y_{c,i}(\mathrm{x})}{\partial \mathrm{x}}
$$
 for $i=2,\dots, N$ and $\tilde D := (\Hess_{y_2,\ldots,y_N} J)(\mathrm{x}, y_{c,2}(\mathrm{x}), \dots, y_{c,N}(\mathrm{x})) $,  then we have
$$
\det(\Hess_{\mathrm{x},y_2,\ldots,y_N} \tilde S^J) (\mathrm{x},  y_{c,2}(\mathrm{x}), \dots, y_{c,N}(\mathrm{x})))
= P \cdot D \cdot P^{\!\top},
$$
where 
$
P=
\begin{pNiceArray}{c | c c c}
1& k_2(\mathrm{x}) & \dots &   k_N(\mathrm{x})  \\ \hline
  0  &  \Block{3-3}{I}  \\
  \vdots & \\
  0 & \\
\end{pNiceArray}
\quad \text{ and } \quad
D=
\begin{pNiceArray}{c | c c c}
\frac{i}{2}\frac{\partial w^{loc}_l}{\partial w^{loc}_m}  & 0 & \dots &  0 \\ \hline
  0  &  \Block{3-3}{\tilde D} \\
  \vdots & \\
  0 & \\
\end{pNiceArray}.
$

This implies the desired result. Thus, it suffices to prove Claims (1)-(3). Claim (1) follows from the definition of $S^J$. Next, for $i\in \{2,\dots, N\}$, since
$$  \frac{\partial \tilde S^J}{ \partial y_i}(\mathrm{x}, y_{c,2}(\mathrm{x}), \dots, y_{c,N}(\mathrm{x})  ) = 0,$$
by differentiating both sides with respect to $\mathrm{x}$, by the chain rule, we have
$$
\frac{\partial^2 \tilde S^J}{\partial \mathrm{x} \partial y_i}(\mathrm{x}, y_{c,2}(\mathrm{x}), \dots, y_{c,N}(\mathrm{x})) \cdot \Bigg( \frac{\partial \mathrm{x}}{\partial \mathrm{x}} \Bigg) + 
\sum_{l=2}^N\frac{\partial^2 \tilde S^J}{ \partial y_l\partial y_i}(\mathrm{x}, y_{c,2}(\mathrm{x}), \dots, y_{c,N}(\mathrm{x}))\cdot \Bigg( \frac{\partial y_{c,l}(\mathrm{x})}{\partial \mathrm{x}} \Bigg) = 0
$$
and Claim (2) follows. Finally, by Proposition \ref{JandNZ}, since 
$$\frac{\partial}{\partial \mathrm{x}} J(\mathrm{x}, y_{c,2}(\mathrm{x}), \dots, y_{c,N}(\mathrm{x})) = \frac{i w^{loc}_l(\mathrm{x})}{2},$$ 
by differentiating both sides with respect to $\mathrm{x}$, we have
$$
\frac{\partial^2 J}{\partial \mathrm{x}^2}(\mathrm{x}, y_{c,2}(\mathrm{x}), \dots, y_{c,N}(\mathrm{x})) 
+ \sum_{l=2}^N \frac{\partial^2 J}{\partial \mathrm{x} \partial y_l}(\mathrm{x}, y_{c,2}(\mathrm{x}), \dots, y_{c,N}(\mathrm{x}))
\cdot \Bigg( \frac{\partial y_{c,l}(\mathrm{x})}{\partial \mathrm{x}} \Bigg) 
= \frac{i}{2} \frac{\partial  w^{loc}_l(\mathrm{x})}{\partial \mathrm{x}}.
$$
Observe that
$$
\frac{\partial^2 \tilde S^J}{\partial \mathrm{x}^2}(\mathrm{x}, y_{c,2}(\mathrm{x}), \dots, y_{c,N}(\mathrm{x})) = \frac{\partial^2 J}{\partial \mathrm{x}^2}(\mathrm{x}, y_{c,2}(\mathrm{x}), \dots, y_{c,N}(\mathrm{x}))$$
and
$$
\frac{\partial^2 \tilde S^J}{\partial \mathrm{x} \partial y_l}(\mathrm{x}, y_{c,2}(\mathrm{x}), \dots, y_{c,N}(\mathrm{x}))=\frac{\partial^2 J}{\partial \mathrm{x} \partial y_l}(\mathrm{x},y_{c,2}(\mathrm{x}), \dots, y_{c,N}(\mathrm{x})).$$
Thus, we have
$$
 \frac{\partial^2 \tilde S^J}{\partial \mathrm{x}^2}(\mathrm{x}, y_{c,2}(\mathrm{x}), \dots, y_{c,N}(\mathrm{x})) 
= \frac{i}{2} \frac{\partial w^{loc}_l(\mathrm{x})}{\partial \mathrm{x}} - 
 \sum_{l=2}^N\frac{\partial^2 \tilde S^J}{\partial \mathrm{x} \partial y_l}(\mathrm{x}, y_{c,2}(\mathrm{x}), \dots, y_{c,N}(\mathrm{x}) )\cdot \Bigg( \frac{\partial y_{c,l}(\mathrm{x})}{\partial \mathrm{x}} \Bigg).
$$
Claim (3) then follows from Claim (2).
\end{proof}

\begin{corollary}
Assume that $\mathbf{B}$ is invertible and $\mathscr{G} = \mathbf{B}^{-1} \mathbf{A}$. Then 
\begin{align*}
&\frac{1}{\sqrt{\pm \det (\Hess_{y_2,\ldots,y_N} J)(\mathrm{x}, y_{c,2}(\mathrm{x}), \dots, y_{c,N}(\mathrm{x}))}}\\
&=  
\frac{1}{q_1 \sqrt{\pm 4 i^{N-1} 
 \left(\prod_{i=1}^N z_i^{-f_i''} z_i''^{f_i - 1} \right) \det( \mathbf{B}^{-1})
 \tau(M, m+n_1 l, X, \mathbf{z^c})}},
\end{align*}
where $n_1 $ is the rational number in Lemma \ref{xmeaning}.
\end{corollary}
\begin{proof}
By Propositions \ref{1loopJ} and \ref{Hesstotor},
\begin{align*}
&\frac{1}{\sqrt{\pm \det (\Hess_{y_2,\ldots,y_N} J)(\mathrm{x}, y_{c,2}(\mathrm{x}), \dots, y_{c,N}(\mathrm{x}))}}\\
&=
\frac{1}{\sqrt{\pm \frac{2(q_1)^2 }{i} \frac{\partial\mathrm{x}}{\partial{w^{loc}_l}}
\det(\Hess_{\mathbf{y}} \tilde S) (\mathbf{y_c}(\mathrm{x}), \xi})}\\
&= 
\frac{1}{q_1 \sqrt{\pm 4 i^{N-1} \det \mathbf{B}^{-1}
 \left(\prod_{i=1}^N z_i^{-f_i''} z_i''^{f_i - 1} \right) \frac{\partial\mathrm{x}}{\partial w^{loc}_l}\tau(M,l,X, \mathbf{z^c})}}.
\end{align*}
This result follows from the change of curve formula of the 1-loop invariants \cite[Theorem 1.19]{PW}.
\end{proof}

\subsection{Asymptotic expansion formula of the Jones function}\label{subsect:AEFJ}

In this section we will prove the following Proposition \ref{AEFJ}, which covers part (2) of our main Theorem \ref{thm:Jones:FAMED}.

\begin{proposition}\label{AEFJ}
We have
\begin{align*}
|\mathfrak{J}_{X}(\hbar, \mathrm{x})|
= \left|\frac{ e^{\frac{i}{\pi}R(\mathbf{z}_\mathrm{x})}}{2\pi {\sqrt\hbar (\det \mathscr{A} \sqrt{ 4 \det\mathbf{B}^{-1}})} } \cdot \frac{e^{\frac{i}{2\pi {\hbar}}\phi_{m,l}(\mathrm{x})}}{\sqrt{\pm \tau(\SS^3 \setminus K, m+n_1 l, \mathbf{z^c}, X)}}  \Big(1 + O({\hbar})\Big) \right|,
\end{align*}
where 
$R$
is defined in Proposition \ref{hvalue}. In particular,
$$
\lim_{\hbar\to 0} 2\pi \hbar \log|\mathfrak{J}_{X}(\hbar,0)|
= - \Vol(M\setminus K).
$$
\end{proposition}

To prove Proposition \ref{AEFJ} we will proceed as in Steps 3 and 4 of the proof of Theorem \ref{mainthmZ}:
\begin{itemize}
    \item we express $|\mathfrak{J}_{X}(\hbar, \mathrm{x})|$ in function of the Jones potential function and an error term between classical and quantum dilogarithms,
    \item we cut the integral into a compact and non-compact part around the complete structure,
    \item we bound the non-compact part by something negligible, using concavity properties and uniform bounds on the error terms,
    \item on the compact piece we push the Taylor expansion further and use the saddle point method.
\end{itemize}
The main differences with the situation of Theorem \ref{mainthmZ} is that we work with a change of variables, one variable fewer, and one term $e^{\frac{\mathrm{x} \lambda_X(\alpha)}{4\pi {\hbar}}}$ fewer.

\begin{proof}

\underline{Step 1: Re-writing the Jones function}

We start by re-writing the Jones function with a convenient form.

Recall from (\ref{ZtoJconv2}) that
\begin{align*}
&\mathfrak{J}_X(\mathrm{x},\hbar)= \frac{q_1}{\det \mathcal{A}}\Big(\frac{1}{2\pi \sqrt{\hbar}}\Big)^{N} \notag\\
&\times \int_{\mathcal{Y}_\alpha'} \left( e^{\frac{1}{2\pi\hbar}\left(i \mathbf{y}^{\!\top} Q \mathbf{y} + \mathbf{y}^{\!\top}( \mathbf{B}^{-1}\boldsymbol \nu  - 
\mathscr{G}
\boldsymbol{\pi})
- i\sum_{k=1}^N \left(\frac{\varepsilon(T_k)+1}{4}\right) y_k^2\right)}
\prod_{k=1}^N \Phi_\B\left(\frac{y_k}{2\pi \sqrt{\hbar}}\right) \right)
 d y_2 \dots dy_N
\end{align*}
with 
$$ y_1 = \frac{\mathrm{x} -  \sum_{k=2}^N q_k y_k + in_2\pi }{q_1}. $$ 

Recall that the Jones potential function is
$$J(\mathrm{x},y_2,\dots, y_N) 
:=\ i \mathbf{y}^{\!\top}
Q \mathbf{y}
 + \mathbf{y}^{\!\top} (\mathbf{B}^{-1}\boldsymbol \nu  - 
\mathscr{G}
\boldsymbol{\pi}) - i\sum_{k=1}^N \left(\frac{\varepsilon(T_k)+1}{4}\right) y_k^2
 - i \sum_{k=1}^N \mathrm{Li}_2\left(-e^{y_k}\right).$$
 Hence we have
 $$\mathfrak{J}_X(\mathrm{x},\hbar)= \frac{q_1}{\det \mathcal{A}}\Big(\frac{1}{2\pi \sqrt{\hbar}}\Big)^{N} 
\int_{\mathcal{Y}_\alpha'} 
 e^{\frac{1}{2\pi \hbar} J(\mathrm{x},y_2,\dots, y_N) } 
e^{
\sum_{k=1}^N \Log \left( \Phi_\B\left(\frac{y_k}{2\pi \sqrt{\hbar}}\right) \right)
-\frac{-i}{2\pi\hbar} \mathrm{Li}_2\left(-e^{y_k}\right)
}
 d y_2 \dots dy_N
 $$

Recall that, as in the proof of Theorem \ref{mainthmZ},
by Proposition \ref{prop:quant:dilog:uniform} (4) we have for all $\B \in (0,\sqrt{\delta})$,
$$
e^{-N(C_\delta + (C/\delta+C')\B^2)}
\leqslant
\left | 
\exp \left (
\sum_{k=1}^N \Log \left( \Phi_\B\left(\frac{y_k}{2\pi \sqrt{\hbar}}\right) \right)
-\frac{-i}{2\pi\hbar} \mathrm{Li}_2\left(-e^{y_k}\right)
\right) \right |
\leqslant
e^{N(C_\delta + (C/\delta+C')\B^2)}
$$
where $\delta= \pi-\max_{k=1, \ldots,N}(|\Im(y_k)|)>0$ and $C,C',C_\delta>0$ as in Proposition \ref{prop:quant:dilog:uniform}.

\underline{Step 2: Reducing the proof  to a compact integral}

Let $\mathscr{O} \subset \CC$ be a sufficiently small neighborhood of $0$ such that for any $\mathrm{x} \in \mathscr{O}$, 
there exist shape parameters $\mathbf{z}$ with positive imaginary parts such that 
$$\mathrm{H}^\C_{X,m+n_1 l}(\mathbf{z})
=\mathrm{H}^\C_{X,m}(\mathbf{z})+n_1 \mathrm{H}^\C_{X,l}(\mathbf{z})
= \mathrm{x}.
$$
Such a neighborhood exists due to Lemma \ref{xbiholo}. For each $\mathrm{x} \in \mathscr{O}$,
let $\alpha(\mathbf{y_c}(\mathrm{x}))$ be the angle structure associated to the critical logarithmic shape parameters $\mathbf{y_c}(\mathrm{x})$.
Let $\mathscr{Y}'_{\mathrm{x}}$ be the multi-contour defined by
$$
\mathscr{Y}'_{\mathrm{x}}:=
\mathscr{Y}'_{\alpha(\mathbf{y_c}(\mathrm{x}))}
=  \prod_{k=1}^N \Big(\RR + i \mathrm{Im}(y_{c,k}(\mathrm{x}))\Big) .
$$

We let $r_0>0$ and 
$$\gamma_{\mathrm{x}} = \left\{\begin{pmatrix}
    y_2\\ \vdots\\ y_N
\end{pmatrix}\in \mathscr{Y}'_{\mathrm{x}} ; \left \| \begin{pmatrix}
    y_2\\ \vdots\\ y_N
\end{pmatrix} - 
\begin{pmatrix}
    y_{c,2}(\mathrm{x})\\ \vdots \\ y_{c,N}(\mathrm{x})
\end{pmatrix}
\right \| \leq r_0\right\}$$
be a $(N-1)$-dimensional ball inside $\mathscr{Y}'_{\mathrm{x}}$ containing the critical point $(y_{c,2}(\mathrm{x}),\ldots,y_{c,N}(\mathrm{x}))$.

By the bound on the error term in Step 1, similarly to \cite[Lemma 7.10]{BAGPN}, there exists $A_3,B_3 > 0$ such that for all $\B \in (0,A_3)$,
$$
\Bigg|
\int_{\mathcal{Y}_{\mathrm{x}}'\setminus \gamma_{\mathrm{x}}} 
 e^{\frac{1}{2\pi \hbar} J(\mathrm{x},y_2,\dots, y_N) } 
e^{
\sum_{k=1}^N \Log \left( \Phi_\B\left(\frac{y_k}{2\pi \sqrt{\hbar}}\right) \right)
-\frac{-i}{2\pi\hbar} \mathrm{Li}_2\left(-e^{y_k}\right)
}
 d y_2 \dots dy_N
\Bigg|
\leq B_3 e^{\frac{M_3}{2\pi \hbar}},
$$
where $M_3 = \max\{\Re S(\mathbf{y}) \mid \mathbf{y} \in \partial \gamma_{\mathrm{x}}\} $. By Propositions \ref{concavSxgen} and \ref{JandNZ}, we have that $M_3$ is strictly smaller than  $\Re S(\mathrm{x}, \mathbf{y}(\mathrm{x}))$ which is a global strict maximum (since it is a critical point of a strictly concave function). 
Thus we have $M_3 < \Re(\phi_{m,l}(\mathrm{x}))$ and the above integral on the non-compact part is negligible relative to the expected asymptotics.

\underline{Step 3: Proving the expected asymptotics on the compact integral}

Finally, by applying Proposition \ref{semihbar} to the compact integral we express it as such:
\begin{align*}
&  \int_{\gamma_{\mathrm{x}}} 
 e^{\frac{1}{2\pi \hbar} J(\mathrm{x},y_2,\dots, y_N) } 
e^{
\sum_{k=1}^N \Log \left( \Phi_\B\left(\frac{y_k}{2\pi \sqrt{\hbar}}\right) \right)
-\frac{-i}{2\pi\hbar} \mathrm{Li}_2\left(-e^{y_k}\right)
}
 d y_2 \dots dy_N
\\
&=   \int_{\gamma_{\mathrm{x}}} h(\mathbf{y})
 e^{\frac{1}{2\pi \hbar} J(\mathrm{x},y_2,\dots, y_N) } 
e^{
O(\hbar)
}
 d y_2 \dots dy_N,
\end{align*}
like in Step 4 of the proof of Theorem \ref{mainthmZ}.

Now we apply Proposition \ref{saddle} to the previous integral (with the same assumptions on bounds over compact sets as in the proof of Theorem \ref{mainthmZ}), and 
we obtain:
\begin{align*}
&\ \mathfrak{J}_{X}(\hbar, \mathrm{x}) \\
=&\ \frac{q_1}{\det \mathscr{A}}\Big(\frac{1}{2\pi {\sqrt{\hbar}}}\Big)^{N} \big(2\pi {\hbar}\big)^{\frac{N-1}{2}} \frac{h(L^{-1}(\mathrm{x}, y_2(\mathrm{x}),\dots,y_N(\mathrm{x})))}{\sqrt{\pm \frac{\det(\Hess J)(\mathrm{x},y_2(\mathrm{x}),\dots,y_N(\mathrm{x}))}{(2\pi)^{N-1}}}} e^{\frac{1}{2\pi {\hbar}}J(\mathrm{x}, y_2(\mathrm{x}),\dots,y_N(\mathrm{x}))} \Big(1 + O({\hbar})\Big).
\end{align*}
The result then follows from Proposition \ref{JandNZ}, \ref{1loopJ} and \ref{hvalue}.

\end{proof}

We can now conclude with the proof of Theorem \ref{thm:Jones:FAMED}.

\begin{proof}[Proof of Theorem \ref{thm:Jones:FAMED}]
Proposition \ref{pfexistJ} covers part (1) of Theorem \ref{thm:Jones:FAMED}.
Proposition \ref{AEFJ} covers part (2) of Theorem \ref{thm:Jones:FAMED}.
\end{proof}

\subsection{Asymptotic expansion and AJ conjecture}\label{sub:AJ}
Consider the set of functions 
$$ \mathscr{F} = \{f \mid f: \RR_{>0} \times \CC \to \CC \}. $$
Let $\hat{M}, \hat{L}: \mathscr{F} \to \mathscr{F}$ be two operators defined by
$$ \hat{M}( f(\hbar, \mathrm{x}) ) = e^{\mathrm{x}} f(\hbar,  \mathrm{x}), \quad \hat{L}( f(\hbar,  \mathrm{x}) ) = f(\hbar,  \mathrm{x} + 4\pi i\hbar). $$
Let $q=e^{4\pi i\hbar}$. Note that $\hat{L}\hat{M} =  q\hat{M} \hat{L} $. In particular, $\hat{M}$ and $ \hat{L}$ define a quantum torus
$$ \mathscr{T} = \langle \hat{M}, \hat{L} \mid \hat{L}\hat{M} =  q\hat{M} \hat{L}  \rangle.$$

\begin{proof}[Proof of Theorem \ref{thm:AJ}]
Suppose there exists $\hat{A}(\hat{M}, \hat{L}, q) = \sum_{n=0}^d a_n(\hat{M},q) \hat{L}^n \in \ZZ[ \hat{M}, \hat{L}, q]$ such that 
$$ \hat{A}(\hat{M}, \hat{L}, q) (\mathfrak{J}_X(\hbar,  \mathrm{x})) 
= \sum_{n=0}^d a_n(\hat{M},q) (\mathfrak{J}_X(\hbar,  \mathrm{x} + 4\pi i n \hbar))
= 0.$$
Note that for $n=0,1,\dots,d$, 
\begin{align*}
&\frac{a_n(\hat{M},q) (\mathfrak{J}_X(\hbar, \mathrm{x} + 4n\pi i \hbar))}{\mathfrak{J}_X(\hbar, \mathrm{x})}\\
&= \frac{a_n(\hat{M},q) C( \mathrm{x} + 4n\pi i \hbar)}{C( \mathrm{x})}
\exp \bigg( \frac{i}{2\pi \hbar} (\phi_{m,l}( \mathrm{x} + 4n \pi i \hbar) - \phi_{m,l}( \mathrm{x})) \bigg) \bigg( 1 + O(\hbar) \bigg)\\
&= \frac{a_n(\hat{M},q)  C( \mathrm{x} + 4n\pi i \hbar)}{C( \mathrm{x})} \bigg[ \exp\bigg( 2 \phi_{m,l}'( \mathrm{x}) \bigg)\bigg]^n  \bigg( 1 + O(\hbar) \bigg) \\
&\to a_n(M, 1) L^n
\end{align*}
as $\hbar \to 0$, where $M = e^{\mathrm{x}} = e^{ w_m^{loc}}$ and $L = e^{ w_l^{loc}}$. 
As a result, since
$$
\lim_{\hbar \to 0} \frac{\hat{A}(\hat{M}, \hat{L}, q) (\mathfrak{J}_X(\hbar, \mathrm{x}))  }{\mathfrak{J}_X(\hbar, \mathrm{x})}
= 0,
$$
we have
$$ \hat{A}(M,L,1) = 0.$$
In particular, since $\hat{A}$ is a polynomial equation satisfied by $M$ and $L$ for all $\mathrm{x}$ sufficiently close to $0$,  we have $A_K^0(M,L) \mid \hat{A}(M,L,1)$, where $A_K^0(M,L)$ is the geometric component of the $PSL(2;\C)$ $A$-polynomial of the knot $K$.
\end{proof}

\section{FAMED triangulations for hyperbolic twist knots}\label{ITXn}
In this section we cover the geometric triangulations $X_n$ for the infinite family of hyperbolic twist knots $K_n$ constructed in \cite{BAGPN}. We prove that they are FAMED for the preferred longitude, which implies that all the theorems in this paper apply to the twist knots, and as such expand and generalize the results of \cite{BAGPN}.

\begin{proposition}\label{prop:FAMED:twist:knots}
For every $n\geqslant 2$, the ideal triangulation $X_n$ of the twist knot complement $\SS^3 \setminus K_n$ defined in \cite{BAGPN} is FAMED for the preferred longitude $l$ of $K_n$. Moreover, in Lemma \ref{xmeaning}, we can take $n_1=n_2=0$.
\end{proposition}

\begin{proof} We first consider the case where $n$ is odd. 
By a direct computation (see \cite[Section 4.3]{BAGPN} for details), the edge equations are given by
\begin{align*}
\omega_s^\CC(\tilde{\mathbf{z}}) &= 2\Log(z_U) + \Log(z_V') + \Log(z_V'') + \Log(z_W) + \Log(z_W') \\
\omega_0^\CC(\tilde{\mathbf{z}}) &= 2\Log(z_1) + \Log(z_1') + 2\Log(z_2) + \dots + 2\Log(z_p) + \Log(z_V) + \Log(z_W'') \\
\omega_1^\CC(\tilde{\mathbf{z}}) &= 2\Log(z_1'') + \Log(z_2') \\
\omega_k^\CC(\tilde{\mathbf{z}}) &= \Log(z_{k-1}') + 2\Log(z_k'') + \Log(z_{k+1}') \quad(\text{for $2\leq k \leq p-1$}) \\
\omega_p^\CC(\tilde{\mathbf{z}}) &= \Log(z_{p-1}') + 2\Log(z_p'') + \Log(z_U') + \Log(z_V') + \Log(z_W) \\
\omega_{p+1}^\CC(\tilde{\mathbf{z}}) &= \Log(z_p') + \Log(z_U') + 2\Log(z_U'') + \Log(z_V) + \Log(z_V'') + \Log(z_W') + \Log(z_W'')
\end{align*}
Thus, the matrices $\mathbf{A}$ and $\mathbf{B}$ are given by
\begin{align*}
\mathbf{A} = \ \
\begin{pNiceMatrix}[first-row, first-col]
 &  z_1& z_2 & z_3 & z_4 & \dots & z_{p-1} &z_{p}  & z_{U} & z_V & z_W  \\
\mathscr{E}_{X_n, s}  & 0 & 0 & 0 & 0 & \dots & 0 & 0 & 2 & -1 & 0 \\
\mathscr{E}_{X_n, 0}  & 1 & 2 & 2 & 2 & \dots & 2 & 2 & 0 & 1 & 0 \\
\mathscr{E}_{X_n, 1}  & 0 & -1 & 0 & 0 & \dots & 0 & 0 & 0 & 0 & 0 \\
\mathscr{E}_{X_n, 2}  & -1 & 0 & -1 & 0 & \dots & 0 & 0 & 0 & 0 & 0 \\
\mathscr{E}_{X_n, 3}  & 0 & -1 & 0 & -1 & \dots & 0 & 0 & 0 & 0 & 0 \\
\vdots  & \vdots & \vdots & \vdots & \vdots & \vdots & \vdots & \vdots & \vdots & \vdots & \vdots \\
\mathscr{E}_{X_n, p-1}  & 0 & 0 & 0 & 0 & \dots & 0 & -1 & 0 & 0 & 0 \\
\mathscr{E}_{X_n, p+1}  & 0 & 0 & 0 & 0 & \dots & 0 & -1 & -1 & 1 & -1 \\
l & 0 & 0 & 0 & 0 & \dots & 0 & 0 & 2 & -2 & 2
\end{pNiceMatrix}
\end{align*}
and
\begin{align*}
\mathbf{B} = \ \
\begin{pNiceMatrix}[first-row, first-col]
 &  z_1& z_2 & z_3 & z_4 & \dots & z_{p-1} &z_{p}  & z_{U} & z_V & z_W  \\
\mathscr{E}_{X_n, s}  & 0 & 0 & 0 & 0 & \dots & 0 & 0 & 0 & 0 & -1 \\
\mathscr{E}_{X_n, 0}  & -1 & 0 & 0 & 0 & \dots & 0 & 0 & 0 & 0 & 1 \\
\mathscr{E}_{X_n, 1}  & 2 & -1 & 0 & 0 & \dots & 0 & 0 & 0 & 0 & 0 \\
\mathscr{E}_{X_n, 2}  & -1 & 2 & -1 & 0 & \dots & 0 & 0 & 0 & 0 & 0 \\
\mathscr{E}_{X_n, 3}  & 0 & -1 & 2 & -1 & \dots & 0 & 0 & 0 & 0 & 0 \\
\vdots  & \vdots & \vdots & \vdots & \vdots & \vdots & \vdots & \vdots & \vdots & \vdots & \vdots \\
\mathscr{E}_{X_n, p-1}  & 0 & 0 & 0 & 0 & \dots & 2 & -1 & 0 & 0 & 0 \\
\mathscr{E}_{X_n, p+1}  & 0 & 0 & 0 & 0 & \dots & 0 & -1 & 1 & 1 & 0 \\
l & 0 & 0 & 0 & 0 & \dots & 0 & 0 & 0 & 2 & 2 
\end{pNiceMatrix}.
\end{align*}
Thus, we have
$$
\mathbf{B}^{-1} = \ \
\begin{pmatrix}
-1 & -1 & 0 & 0 & 0 & \dots & 0 & 0 & 0 & 0 \\
-2 & -2 & -1 & 0 & 0 & \dots & 0 & 0 & 0 & 0  \\
-3 & -3 & -2 & -1 & 0 & \dots & 0 & 0 & 0 & 0 \\
-4 & -4 & -3 & -2 & -1 & \dots & 0 & 0 & 0 & 0 \\
\vdots & \vdots & \vdots & \vdots & \vdots & \dots & \vdots & \vdots & \vdots & \vdots  \\
-p+1 & -p+1 & -p+2 & -p+3 & -p+4 & \dots & -1 & 0 & 0 & 0 \\
-p & -p & -p+1 & -p+2 & -p+3 & \dots & -2 & -1 & 0 & 0 \\
-p-1 & -p & -p+1 & -p+2 & -p+3 & \dots & -2 & -1 & 1 & -1/2 \\
1 & 0 & 0 & 0 & 0 & \dots & 0 & 0 & 0 & 1/2 \\
-1 & 0 & 0 & 0 & 0 & \dots & 0 & 0 & 0 & 0 \\
\end{pmatrix}
$$
and
$$
\mathbf{B}^{-1} \mathbf{A} = \ \
\begin{pmatrix}
 -1 & -2 &  \dots & -2 & -2 & -2 & 0 & 0 \\
 -2 & -3 &  \dots & -4 & -4 & -4 & 0 & 0 \\
 \vdots & \ddots & \vdots & \vdots & \vdots & \vdots & \vdots & \vdots \\
 -2 & -4 & \dots & -2p+3 & -2p+2 & -2p+2 & 0 & 0 \\
 -2 & -4 & \dots & -2p+2 & -2p+1 & -2p & 0 & 0 \\
 -2 & -4 &  \dots & -2p+2 & -2p & -2p-4 & 3 & -2 \\
 0 & 0 &  \dots & 0 & 0 & 3 & -2 & 1 \\
 0 & 0 &  \dots & 0 & 0 & -2 & 1 & 0
\end{pmatrix}.
$$
Besides, from \cite[Lemma 5.5]{BAGPN}, 
$$
\mathcal{E}Q\mathcal{E}=\ \ 
\begin{pNiceMatrix}[first-row, first-col]
 &  t_1& t_2 & \dots & t_{p-1} &t_{p}  & t_{U} & t_V & t_W  \\
t_1  & 1 & 1 &  \dots & 1 & 1 & -1 & 0 & 0 \\
t_2 & 1 & 2 &  \dots & 2 & 2 & -2 & 0 & 0 \\
\vdots  & \vdots & \ddots & \vdots & \vdots & \vdots & \vdots & \vdots & \vdots \\
t_{p-1}  & 1 & 2 & \dots & p-1 & p-1 & -(p-1) & 0 & 0 \\
t_p  & 1 & 2 & \dots & p-1 & p & -p & 0 & 0 \\
t_U & -1 & -2 &  \dots & -(p-1) & -p & p+2 & -3/2 & 1 \\
t_V & 0 & 0 &  \dots & 0 & 0 & -3/2 & 1 & -1/2 \\
t_W & 0 & 0 &  \dots & 0 & 0 & 1 & -1/2 & 0
\end{pNiceMatrix}.
$$
As a result,
\begin{align*}
\mathscr{G}
&= 
-2
Q
+
\begin{pmatrix}
\mathbf{1} & \mathbf{0} \\
\mathbf{0} & \mathbf{0}
\end{pmatrix}\\
&= \ \
\begin{pNiceMatrix}[first-row, first-col]
 &  t_1& t_2 & \dots & t_{p-1} &t_{p}  & t_{U} & t_V & t_W  \\
t_1  & -1 & -2 &  \dots & -2 & -2 & -2 & 0 & 0 \\
t_2 & -2 & -3 &  \dots & -4 & -4 & -4 & 0 & 0 \\
\vdots  & \vdots & \ddots & \vdots & \vdots & \vdots & \vdots & \vdots & \vdots \\
t_{p-1}  & -2 & -4 & \dots & -2p+3 & -2p+2 & -2p+2 & 0 & 0 \\
t_p  & -2 & -4 & \dots & -2p+2 & -2p+1 & -2p & 0 & 0 \\
t_U & -2 & -4 &  \dots & -2p+2 & -2p & -2p-4 & 3 & -2 \\
t_V & 0 & 0 &  \dots & 0 & 0 & 3 & -2 & 1 \\
t_W & 0 & 0 &  \dots & 0 & 0 & -2 & 1 & 0
\end{pNiceMatrix}.
\end{align*}
The computations above show that $X_n$ is FAMED for odd $n$. 
Similarly, when $n$ is even, by a direct computation (see \cite[Section 8.2]{BAGPN} for details)
\begin{align*}
\omega_s(\tilde{\mathbf{z}}) &= 2\Log(z_U) + \Log(z_V') + \Log(z_V'') + \Log(z_W) + \Log(z_W') \\
\omega_0(\tilde{\mathbf{z}}) &= 2\Log(z_1) + \Log(z_1') + 2\Log(z_2) + \dots + 2\Log(z_p) + \Log(z_V) + \Log(z_W'') \\
\omega_1(\tilde{\mathbf{z}}) &= 2\Log(z_1'') + \Log(z_2') \\
\omega_k(\tilde{\mathbf{z}}) &= \Log(z_{k-1}') + 2\Log(z_k'') + \Log(z_{k+1}') \quad(\text{for $2\leq k \leq p-1$}) \\
\omega_p(\tilde{\mathbf{z}}) &= \Log(z_{p-1}') + 2\Log(z_p'') + 2\Log(z_U') + \Log(z_U'') + \Log (z_V) + \Log(z_V') \\ & \quad + \Log(z_W) + \Log(z_W'') \\
\omega_{p+1}(\tilde{\mathbf{z}}) &= \Log(z_p') + \Log(z_U'') + \Log(z_V'') + \Log(z_W') 
\end{align*}
Thus, the matrices $\mathbf{A}$ and $\mathbf{B}$ are given by
\begin{align*}
\mathbf{A} = \ \
\begin{pNiceMatrix}[first-row, first-col]
 &  z_1& z_2 & z_3 & z_4 & \dots & z_{p-1} &z_{p}  & z_{U} & z_V & z_W  \\
\mathscr{E}_{X_n, s}  & 0 & 0 & 0 & 0 & \dots & 0 & 0 & 2 & -1 & 0 \\
\mathscr{E}_{X_n, 0}  & 1 & 2 & 2 & 2 & \dots & 2 & 2 & 0 & 1 & 0 \\
\mathscr{E}_{X_n, 1}  & 0 & -1 & 0 & 0 & \dots & 0 & 0 & 0 & 0 & 0 \\
\mathscr{E}_{X_n, 2}  & -1 & 0 & -1 & 0 & \dots & 0 & 0 & 0 & 0 & 0 \\
\mathscr{E}_{X_n, 3}  & 0 & -1 & 0 & -1 & \dots & 0 & 0 & 0 & 0 & 0 \\
\vdots  & \vdots & \vdots & \vdots & \vdots & \vdots & \vdots & \vdots & \vdots & \vdots & \vdots \\
\mathscr{E}_{X_n, p-1}  & 0 & 0 & 0 & 0 & \dots & 0 & -1 & 0 & 0 & 0 \\
\mathscr{E}_{X_n, p+1}  & 0 & 0 & 0 & 0 & \dots & 0 & -1 & 0 & 0 & -1 \\
l & 0 & 0 & 0 & 0 & \dots & 0 & 0 & - 2 & 4 & 2
\end{pNiceMatrix}
\end{align*}
and
\begin{align*}
\mathbf{B} = \ \
\begin{pNiceMatrix}[first-row, first-col]
 &  z_1& z_2 & z_3 & z_4 & \dots & z_{p-1} &z_{p}  & z_{U} & z_V & z_W  \\
\mathscr{E}_{X_n, s}  & 0 & 0 & 0 & 0 & \dots & 0 & 0 & 0 & 0 & -1 \\
\mathscr{E}_{X_n, 0}  & -1 & 0 & 0 & 0 & \dots & 0 & 0 & 0 & 0 & 1 \\
\mathscr{E}_{X_n, 1}  & 2 & -1 & 0 & 0 & \dots & 0 & 0 & 0 & 0 & 0 \\
\mathscr{E}_{X_n, 2}  & -1 & 2 & -1 & 0 & \dots & 0 & 0 & 0 & 0 & 0 \\
\mathscr{E}_{X_n, 3}  & 0 & -1 & 2 & -1 & \dots & 0 & 0 & 0 & 0 & 0 \\
\vdots  & \vdots & \vdots & \vdots & \vdots & \vdots & \vdots & \vdots & \vdots & \vdots & \vdots \\
\mathscr{E}_{X_n, p-1}  & 0 & 0 & 0 & 0 & \dots & 2 & -1 & 0 & 0 & 0 \\
\mathscr{E}_{X_n, p+1}  & 0 & 0 & 0 & 0 & \dots & 0 & -1 & 1 & 1 & -1 \\
l & 0 & 0 & 0 & 0 & \dots & 0 & 0 & 0 & 2 & 0 
\end{pNiceMatrix}.
\end{align*}
Thus, we have
$$
\mathbf{B}^{-1} = \ \
\begin{pmatrix}
-1 & -1 & 0 & 0 & 0 & \dots & 0 & 0 & 0 & 0 \\
-2 & -2 & -1 & 0 & 0 & \dots & 0 & 0 & 0 & 0  \\
-3 & -3 & -2 & -1 & 0 & \dots & 0 & 0 & 0 & 0 \\
-4 & -4 & -3 & -2 & -1 & \dots & 0 & 0 & 0 & 0 \\
\vdots & \vdots & \vdots & \vdots & \vdots & \dots & \vdots & \vdots & \vdots & \vdots  \\
-p+1 & -p+1 & -p+2 & -p+3 & -p+4 & \dots & -1 & 0 & 0 & 0 \\
-p & -p & -p+1 & -p+2 & -p+3 & \dots & -2 & -1 & 0 & 0 \\
-p-1 & -p & -p+1 & -p+2 & -p+3 & \dots & -2 & -1 & 1 & -1/2 \\
0 & 0 & 0 & 0 & 0 & \dots & 0 & 0 & 0 & 1/2 \\
-1 & 0 & 0 & 0 & 0 & \dots & 0 & 0 & 0 & 0 \\
\end{pmatrix}
$$
and
$$
\mathbf{B}^{-1} \mathbf{A} = \ \
\begin{pmatrix}
 -1 & -2 &  \dots & -2 & -2 & -2 & 0 & 0 \\
 -2 & -3 &  \dots & -4 & -4 & -4 & 0 & 0 \\
 \vdots & \ddots & \vdots & \vdots & \vdots & \vdots & \vdots & \vdots \\
 -2 & -4 & \dots & -2p+3 & -2p+2 & -2p+2 & 0 & 0 \\
 -2 & -4 & \dots & -2p+2 & -2p+1 & -2p & 0 & 0 \\
 -2 & -4 &  \dots & -2p+2 & -2p & -2p-1 & -1 & -2 \\
 0 & 0 &  \dots & 0 & 0 & -1 & 2 & 1 \\
 0 & 0 &  \dots & 0 & 0 & -2 & 1 & 0
\end{pmatrix}.
$$

Besides, from \cite[Theorem 8.4]{BAGPN}, 
$$
\mathcal{E}Q\mathcal{E}=\ \
\begin{pNiceMatrix}[first-row, first-col]
 &  t_1& t_2 & \dots & t_{p-1} &t_{p}  & t_{U} & t_V & t_W  \\
t_1  & 1 & 1 &  \dots & 1 & 1 & 1 & 0 & 0 \\
t_2 & 1 & 2 &  \dots & 2 & 2 & 2 & 0 & 0 \\
\vdots  & \vdots & \ddots & \vdots & \vdots & \vdots & \vdots & \vdots & \vdots \\
t_{p-1}  & 1 & 2 & \dots & p-1 & p-1 & p-1 & 0 & 0 \\
t_p  & 1 & 2 & \dots & p-1 & p & p & 0 & 0 \\
t_U & 1 & 2 &  \dots & p-1 & p & p+1 & -1/2 & -1 \\
t_V & 0 & 0 &  \dots & 0 & 0 & -1/2 & -1 & -1/2 \\
t_W & 0 & 0 &  \dots & 0 & 0 & -1 & -1/2 & 0
\end{pNiceMatrix}.
$$
As a result,
\begin{align*}
\mathscr{G}
&= 
-2
Q
+
\begin{pmatrix}
\mathbf{1} & \mathbf{0} \\
\mathbf{0} & \mathbf{0}
\end{pmatrix}\\
&= \ \
\begin{pNiceMatrix}[first-row, first-col]
 &  t_1& t_2 & \dots & t_{p-1} &t_{p}  & t_{U} & t_V & t_W  \\
t_1  & -1 & -2 &  \dots & -2 & -2 & -2 & 0 & 0 \\
t_2 & -2 & -3 &  \dots & -4 & -4 & -4 & 0 & 0 \\
\vdots  & \vdots & \ddots & \vdots & \vdots & \vdots & \vdots & \vdots & \vdots \\
t_{p-1}  & -2 & -4 & \dots & -2p+3 & -2p+2 & -2p+2 & 0 & 0 \\
t_p  & -2 & -4 & \dots & -2p+2 & -2p+1 & -2p & 0 & 0 \\
t_U & -2 & -4 &  \dots & -2p+2 & -2p & -2p-1 & -1 & -2 \\
t_V & 0 & 0 &  \dots & 0 & 0 & -1 & 2 & 1 \\
t_W & 0 & 0 &  \dots & 0 & 0 & -2 & 1 & 0
\end{pNiceMatrix}.
\end{align*}
The computations above show that $X_n$ is FAMED for even $n$. 

By \cite[Theorem 4.1, Theorem 8.2]{BAGPN}, $X_n$ is geometric at the complete hyperbolic structure. By continuity, $X_n$ is still geometric on a small neighborhood around the complete hyperbolic structure. Moreover, the last column of $\mathbf{B}^{-1}$, which is the same as the last row of $(\mathbf{B}^{-1})^{\!\top}$, is given by $(0,\dots,0,1/2,-1/2,0)$ and $(0,\dots,0,-1/2,1/2,0)$ for odd and even $n$ respectively. Note that in \cite{BAGPN}, in Figure 14 (when $n$ is odd) and Figure 23 (when $n$ is even), the algebraic intersection number of the meridian and longitude are $1$ and $-1$ respectively. Thus, by suitably choosing the orientations of the meridian and longitude such that their algebraic intersection number is $1$, in Lemma \ref{xmeaning}, we can take $n_1=n_2=0$. 

\end{proof}

\end{document}